\documentclass[11pt]{amsart}
\pdfoutput=1 

\usepackage[utf8]{inputenc}
\usepackage{amssymb,amsmath, amsthm, mathabx}
\usepackage[dvipsnames]{xcolor}
\usepackage{enumerate}
\usepackage[pagebackref,hidelinks]{hyperref}
\usepackage{caption}
\usepackage{subcaption}
\usepackage{a4wide}
\usepackage{accents}
\usepackage{tikz}
\usepackage{graphicx}
\usepackage[percent]{overpic}
\usepackage{chngcntr}
\usepackage{datetime2}
\usepackage{scalerel}
\usepackage{bm}
\usepackage{tcolorbox}
\usepackage{microtype} 
\usepackage{mathtools} 

\newtheorem{thm}{Theorem}[section]
\newtheorem{prop}[thm]{Proposition}
\newtheorem{lem}[thm]{Lemma}
\newtheorem{cor}[thm]{Corollary}
\theoremstyle{definition}
\newtheorem{rem}[thm]{Remark}

\newtheorem{exmpl}[thm]{Example}

\counterwithin{figure}{section}
\counterwithin{equation}{section}

\newcommand{\bbR}{\mathbb{R}}

\newcommand{\bbZ}{\mathbb{Z}}

\newcommand{\bfE}{\mathbf{E}}

\newcommand{\bfP}{\mathbf{P}}

\newcommand{\bfX}{\mathbf{X}}
\newcommand{\bfY}{\mathbf{Y}}

\newcommand{\bfa}{\mathbf{a}}
\newcommand{\bfb}{\mathbf{b}}
\newcommand{\bfc}{\mathbf{c}}

\newcommand{\bfk}{\mathbf{k}}

\newcommand{\bfp}{\mathbf{p}}

\newcommand{\bfr}{\mathbf{r}}
\newcommand{\bfs}{\mathbf{s}}

\newcommand{\bfu}{\mathbf{u}}
\newcommand{\bfv}{\mathbf{v}}

\newcommand{\bfx}{\mathbf{x}}
\newcommand{\bfy}{\mathbf{y}}
\newcommand{\bfz}{\mathbf{z}}

\newcommand{\mfc}{\mathfrak{c}}

\newcommand{\rightset}{\mathcal E^\rightarrow}

\newcommand{\upset}{\mathcal E^\uparrow}

\newcommand{\sI}{\mathcal{I}}
\newcommand{\sJ}{\mathcal{J}}

\newcommand{\sS}{\mathcal{S}}

\newcommand{\sU}{\mathcal{U}}

\allowdisplaybreaks

\newcommand{\one}{\mathbf{1}}

\newcommand{\wt}[1]{\widetilde{#1}}
\newcommand{\wh}[1]{\widehat{#1}}
\newcommand{\wc}[1]{\widecheck{#1}}

\newcommand{\f}{\frac}
\newcommand{\deq}{\stackrel{\rm{dist.}}{=}} 
\newcommand{\Z}{\mathbb Z}
\newcommand{\R}{\mathbb R}

\newcommand{\I}{\mathrm{I}}
\newcommand{\J}{\mathrm{J}}
\newcommand{\B}{\mathrm{B}}

\newcommand{\dd}{\mathrm{d}}
\newcommand{\hor}{\mathrm{hor}}
\newcommand{\ver}{\mathrm{ver}}
\newcommand{\A}{\mathrm{A}}
\newcommand{\Bh}{\mathcal{I}}
\newcommand{\Bv}{\mathcal{J}}
\newcommand{\rec}{\mathrm{r}}
\newcommand{\gr}[1]{\grave{#1}}

\newcommand{\w}{\mathrm{w}}
\newcommand{\Lp}{\mathrm{L}}

\newcommand{\ul}[1]{\underline{#1}}
\newcommand{\ol}[1]{\overline{#1}}

\newcommand{\Id}{\mathrm{Id}}

\newcommand{\col}{\mathrm{col}}
\newcommand{\row}{\mathrm{row}}

\newcommand{\ii}{\mathbf{i}}

\newcommand{\qedex}{\hfill $\triangle$}

\newcommand{\ZZ}{\mathbb{Z}}

\newcommand{\tw}{\omega^*}

\newcommand{\tq}{\tilde{q}}
\newcommand{\tQ}{\tilde{Q}}

\newcommand{\Exp}{\operatorname{Exp}}

\newcommand{\cM}{\mathcal{M}}

\allowdisplaybreaks

\title[Permutation invariance of LPP and the Busemann process]{Permutation invariance in last-passage percolation \\ and the distribution of the Busemann process}
\author{Erik Bates}
\address{Erik Bates\\ North Carolina State University \\ Department of Mathematics \\ 2311 Stinson Drive \\ Raleigh, NC 27695-8205 \\ USA}
\email{ebates@ncsu.edu}
\urladdr{https://www.ewbates.com/}
\author{Elnur Emrah}
\address{Elnur Emrah\\ University College Dublin \\ School of Mathematics and Statistics \\ Dublin 4\\ Ireland}
\email{elnur.emrah@ucd.ie}
\urladdr{https://sites.google.com/view/elnur-emrah}
\author{James Martin}
\address{James Martin\\Department of Statistics\\
University of Oxford\\24-29 St Giles'\\Oxford OX1 3LB\\United Kingdom}
\email{martin@stats.ox.ac.uk}
\urladdr{https://www.stats.ox.ac.uk/~martin}
\author{Timo Sepp\"al\"ainen}
\address{Timo Sepp\"al\"ainen\\ University of Wisconsin--Madison\\  Mathematics Department\\ Van Vleck Hall\\ 480 Lincoln Dr.\\   Madison, WI 53706-1388\\ USA}
\email{seppalai@math.wisc.edu}
\urladdr{https://www.math.wisc.edu/~seppalai}\author{Evan Sorensen}
\address{Evan Sorensen \\ Columbia University \\ Department of Mathematics \\ Room 624, MC 4432, 2990 Broadway, New York, NY 10027, USA}
\email{evan.sorensen@columbia.edu}
\urladdr{https://sites.google.com/view/evan-sorensen}


\begin{document}

\begin{abstract}
In i.i.d.\ exponential last-passage percolation, we describe the joint distribution of Busemann functions, over all edges and over all directions, in terms of a joint last-passage problem in a finite inhomogeneous environment.
More specifically, the Busemann increments within a $k\times\ell$ grid, and associated to $d$ different directions, are equal in distribution to a particular collection of last-passage increments inside a $(k+d-1)\times(\ell+d-1)$ grid. 
The joint Busemann distribution was previously described  
along a horizontal line by Fan and the fourth author, using certain queueing maps. 
By contrast, our new description explicitly gives the joint distribution for any collection of edges (not just along a horizontal line) using only finitely many 
random variables. Our result thus provides an exact and accessible way to sample from the joint distribution. 
In the proof, we rely on one-directional marginal distributions of the inhomogeneous Busemann functions recently studied by Janjigian and the second and fourth authors. 
The second ingredient of our proof is a novel joint invariance of inhomogeneous last-passage times under permutations of the inhomogeneity parameters. Our proof of the invariance is different from earlier proofs of such results, using the Burke property
instead of the RSK correspondence,
and leading to an explicit coupling
of the weights before and after the permutation of the parameters.
\end{abstract}

\maketitle

\tableofcontents

\section{Introduction}
\subsection{Busemann functions and infinite geodesics} \label{Ss:history}

In the 1950s, Herbert Busemann \cite{Busemann-1955} developed what we now know as \textit{Busemann functions} to study the geometry of geodesics in non-Euclidean spaces.
Roughly speaking, the Busemann function associated to a direction $r$ is the limiting difference of distances from two fixed initial points $x$ and $y$ to a common terminal point that moves to 
infinity in direction $r$. 
This yields a quantity $B_{x,y}^r$ that depends on all three parameters, as well as the underlying space.
In recent decades, these functions have been immensely profitable in the study of random growth models, where the non-Euclidean structure arises from a random metric.

This was first seen in the influential works of Newman \cite{newman95} and Hoffman \cite{hoff-05,hoff-08} for first-passage percolation (FPP).
Since then, the construction and distribution of Busemann functions have
been a key focus in both FPP \cite{damr-hans-14,damr-hans-17} and its relatives in the Kardar--Parisi--Zhang (KPZ) universality class: discrete last-passage percolation (LPP) \cite{Emra_Janj_Sepp_25, geor-rass-sepp-lppbuse, geor-rass-sepp-17-geod, janj-rass-sepp-23}, Brownian LPP \cite{busa-sepp-sore-24, sepp-sore-23-aihp,sepp-sore-23-pmp}, and the directed landscape \cite{busa-sepp-sore-24, Rahman-Virag-21}.
Applications include the existence, uniqueness, and coalescence of semi-infinite geodesics, as well as nonexistence of bi-infinite geodesics.
There have been parallel developments for positive-temperature models, namely directed polymers \cite{bates_fan_seppalainen25, georgiou_rassoulagha_seppalainen16, geor-rass-sepp-yilm-15, Groa_Janj_Rass_25b, janj-rass-20-aop} and the KPZ equation  
\cite{janj-rass-sepp-22-arxiv}, where geodesics are replaced by Gibbs measures satisfying the Dobrushin--Lanford--Ruelle equations.

While some of the literature mentioned above considers general dimensions, the Busemann theory is best understood---and most powerful---in the planar case: one spatial dimension and one time dimension.
Two reasons for this stand out:
\begin{enumerate}[\normalfont (i)]

\item \label{reason1}
Planarity implies that geodesics moving in opposing directions must eventually intersect.  This leads to monotonicity of the function $r\mapsto B_{x,y}^r$ when $x$ and $y$ are separated by certain space-time trajectories (down-right paths, in our setting), which simplifies the construction of the Busemann process and makes many of its properties more accessible. 

\item \label{reason2}
For certain exactly solvable models, the invariant measures for the growth dynamics are of product form, e.g.~\cite{blpp_utah, aldo-diac95, bala-cato-sepp, Ferr_Spoh_06,sepp-12-aop-corr}.
This means that as $y$ is varied along those same space-time trajectories (down-right paths), the increments of the process $(B_{x,y}^r)_{y}$ are independent.
This enables, for instance, the derivation of fluctuation exponents for the temporal evolution of these solvable models, which has been carried out for both zero-temperature models \cite{bala-cato-sepp, Cato_Groe_06, emra-janj-sepp-23} and positive-temperature models \cite{Land_Noac_Soso_23, noac-soso-20-arxiv1,noac-soso-20-arxiv2, sepp-valk-10}.

\end{enumerate}
The combination of \eqref{reason1} and \eqref{reason2} yields effectively complete understanding of the one-direction marginals of the Busemann process.
But the joint distribution, which we discuss next and is the main focus of this paper, is much more complicated.

\subsection{Jointly invariant measures}
The invariant measures in \eqref{reason2} are parameterized by the direction variable $r$, and the full collection of Busemann functions $(B_{\bullet}^r)_{r}$ encodes the natural coupling of all these measures.
This joint process is also essential for understanding the geometry of geodesics across all directions simultaneously; for example, the \textit{non}-uniqueness of semi-infinite geodesics in random exceptional directions \cite{busa-sepp-sore-24, janj-rass-sepp-23,sepp-sore-23-pmp} is not apparent from the one-direction marginals alone.
Therefore, there is great interest in obtaining explicit descriptions of the joint distribution of Busemann functions in multiple directions.

For LPP with i.i.d.\ exponential weights (which is the setting of this paper), Fan and the fourth author \cite{Fan_Sepp_20} described this joint distribution on a horizontal line, i.e.\ when $x$ and $y$ vary only on a single level of the lattice. 
The description is in terms of queueing mappings of a sequence of independent exponential random walks with different rate parameters corresponding to different values of $r$. 
The motivation for this type of description came from the earlier work of Ferrari and the third author \cite{ferr-mart-06,ferr-mart-07,ferr-mart-09}, which gave queueing interpretations of the multi-type invariant measures for several particle systems, including the Hammersley process and the totally asymmetric simple exclusion process (TASEP), in both periodic and full-space settings.

The approach of these works is to couple the relevant joint evolution of the model with an auxiliary Markov process whose invariant measures are of product-form with respect to both spatial increments \text{and} direction. 
A suitable transformation (called the ``intertwining mapping'' in \cite{Fan_Sepp_20}) pushes forward the auxiliary Markov process to the desired joint evolution, whose invariant measures are then the image of the product measure under the transformation. 
This technique has now been adapted to several settings, including Brownian last-passage percolation and the directed landscape \cite{busa-sepp-sore-24, sepp-sore-23-pmp}, the inverse-gamma polymer \cite{bates_fan_seppalainen25}, the O'Connell--Yor polymer and KPZ equation \cite{groa-rass-sepp-sore-23+}, and the periodic O'Connell--Yor polymer and KPZ equation \cite{Corw_Gu_Sore_26}. 

In each of these cases, the description of the joint Busemann distribution is explicit only on a single horizontal line. On a path that crosses many levels,    the joint distribution can only be accessed implicitly by evolving the model from the stationary initial condition. 
Since each step of the evolution involves an additional infinite collection of random inputs, this implicit description is not conducive to many potentially useful calculations.
To address this difficulty, the main result of this article (Theorem \ref{T:BusMar}) gives a new and qualitatively different description of the joint distribution using only finitely many random variables.
Specifically, we prove that every finite-dimensional joint distribution of Busemann functions across both space and directions is identical to a joint distribution of point-to-point last-passage increments in an inhomogeneous environment on a finite grid. A positive temperature analogue of the inhomogeneous environment in our construction appeared earlier in Chaumont's PhD thesis \cite{Chaumont-thesis} and in a work of Barraquand and Le Doussal \cite{Barraquand-LeDoussal-2023} in a different context, namely to describe the stationary multi-path partition functions for the inverse-gamma polymer (see Remark \ref{R:StCoup}). 

\subsection{Permutation invariance and inhomogeneous Busemann functions}

Our strategy for accessing the joint distribution of Busemann functions is fundamentally different from that of previous works.
Instead of looking for invariant measures on the infinite space, we make a direct comparison to a finite-space process.
Namely, we compare the joint CDF of the Busemann functions to the joint CDF of last-passage increments in the inhomogeneous environment. 
The proof utilizes two crucial inputs: 
\begin{itemize} \itemsep=3pt
\item 
The first input is a seemingly new permutation invariance in LPP, presented as Theorem \ref{T:LppInv} and proved in Section~\ref{S:PfInv}. 
It states that the joint law of last-passage times in inhomogeneous exponential LPP is preserved by certain permutations of the inhomogeneity parameters.
Given that the invariance is of the entire last-passage process and not just about Busemann functions, we anticipate that this result will be useful beyond this paper.

\item The second input is the existence and properties of Busemann functions  
in the inhomogeneous exponential LPP, proved recently by the second and fourth authors with Janjigian  \cite{Emra_Janj_Sepp_25}.
These results are reviewed in Section~\ref{S:iBusFn}.
\end{itemize}

The proof proceeds by a series of steps that involves shifting the  
terminal points of last-passage increments, introducing boundary columns and rows with auxiliary rates, and permuting the rates of columns and rows. 
We provide a detailed sketch of the proof in Section~\ref{Ss:BMdiscuss} for the basic case of two directions. Although we start with i.i.d.\ exponential last-passage percolation, the inhomogeneous model naturally appears in the argument.  
When only a single direction is considered, the inhomogeneity only appears on the boundary, and the familiar stationary model from \eqref{reason2} in Section~\ref{Ss:history} is recovered (Remark~\ref{R:OneDir}).
For multiple directions, the boundary becomes multi-layered, and the resulting inhomogeneous environment provides a coupling of these stationary models (Remark~\ref{R:StCoup}).
In this way, our argument is built on a coupling of the entire LPP process rather than of just the Busemann functions. 

Our invariance result generalizes the well-known invariances of last-passage times that are visible from exact distributional formulas, for example, in \cite{Boro_Pech_08}. These invariances originate from the properties of Schur measures \cite{Okou_01} and Schur processes \cite{Okou_Resh_03}. A recent work \cite{Dauv_22} of Dauvergne identified and studied a large family of invariances for various integrable LPP and directed polymer models. These combine and generalize the aforementioned invariances with the shift-invariance property discovered earlier by Borodin, Gorin, and Wheeler \cite{Boro_Gori_Whee_22}. In Remark \ref{R:DauvInv} ahead, we will indicate that a somewhat weaker version of our result is a corollary of \cite[Theorem 1.5]{Dauv_22}. 
However, as far as we can tell, the complete result is not a straightforward consequence of \cite{Dauv_22}. 
Our method of proof is also 
quite different from the previous works, which rely on the properties of the Schur functions and the RSK correspondence. Instead, we utilize an explicit coupling along with 
the Burke property of the exponential LPP. 

\subsection{Organization of the paper}
The main result on the joint distribution of Busemann functions is Theorem \ref{T:BusMar} in Section \ref{Ss:BusMar}. Some immediate consequences follow. In particular, in Section \ref{S:edgeBus} we demonstrate that a complete characterization of the Busemann process on a lattice edge follows relatively effortlessly from Theorem \ref{T:BusMar}. Furthermore, in Section \ref{sec:Shen_indep}, we demonstrate a  
short proof of a special case of an independence property for Busemann functions proved previously by Shen \cite{shen25}.  Section \ref{Ss:BMdiscuss} gives an outline of the proof of Theorem \ref{T:BusMar} in the simplest nontrivial case of two directions. To assist the exposition, Figures \ref{F:CDFBd}--\ref{F:TBusDisId2} present serial pictorial representations of the main steps.  

The result on the permutation invariance of the inhomogeneous exponential LPP is 
Theorem \ref{T:LppInv} in Section \ref{Ss:LppInv}. Its proof follows in Section \ref{S:PfInv}. 

The proof of Theorem \ref{T:BusMar} occupies Section \ref{S:PfBuse}. As preparation for the proof, 
Section \ref{S:iBusFn} reviews properties of the Busemann functions of inhomogeneous exponential LPP.

\subsection{Notation and conventions}

Let $[n] =\{i \in \bbZ_{>0}: i \le n\}$, and let $\sS_n$ denote the set of all permutations of $[n]$. 
A random variable $X$ has the exponential distribution with  rate  $a>0$ if $\bfP(X>t)=e^{-at}$ for $t\ge0$, abbreviated $X\sim\Exp\{a\}$. The statement $X\sim\Exp\{0\}$ means $X=\infty$ almost surely.
  
  Throughout the paper, $\w$ represents deterministic real weights, whereas $\omega$ 
  represents exponentially distributed random weights (always independent, but not necessarily identically distributed). We also use $\eta$ for several auxiliary random weights, for example, to describe the inhomogeneous environment in Theorem \ref{T:BusMar}.

\subsection{Acknowledgments and funding}
 E.B. was partially supported by National Science Foundation grant DMS-2412473.
 E.E.\ was supported by the EPSRC grant EP/W032112/1. E.S.\ was partially supported by the Fernholz foundation and by Ivan Corwin's Simons Investigator Grant 929852. This work was initiated in May 2024 while E.S.\ was visiting the University of Bristol and the University of Oxford, which he thanks for their hospitality. The travel of E.S. during that trip was supported by AMS-Simons travel grant AMMS CU23-2401.
T.S.\ was partially supported by  National Science Foundation grants  DMS-2152362 and DMS-2448375, by Simons Foundation grant 1019133, and by  the Wisconsin Alumni Research Foundation.
The authors would also like to thank Xiao Shen for helpful discussions, and the referees for helpful comments and corrections.

\subsection{Publication note}
This version of the article has been accepted for publication, after peer review (when applicable) but is not the Version of Record and does not reflect post-acceptance improvements, or any corrections. The Version of Record is available online at: \hyperref[http://dx.doi.org/10.1007/s00440-026-01516-7]{http://dx.doi.org/10.1007/s00440-026-01516-7}. 


\section{Main results for the Busemann process}
\label{S:BMain}

\subsection{Last-passage times and increments}
\label{Ss:LppInc}

Let $\le$ denote the coordinatewise partial order on $\bbZ^2$, meaning $\bfu = (u_1, u_2) \le \bfv = (v_1, v_2)$ if both $u_1 \le v_1$ and $u_2 \le v_2$. 
We write $\bfu < \bfv$ if $\bfu \le \bfv$ and $\bfu \neq \bfv$.
For $\bfu, \bfv \in \bbZ^2$ with $\bfu \le \bfv$, let $[\bfu, \bfv] = \{\bfp \in \bbZ^2: \bfu \le \bfp \le \bfv\}$. A nonempty, finite, totally ordered subset $\pi \subset \bbZ^2$ is called an \emph{up-right path} if $\pi$ contains $\bfp + (1, 0)$ or $\bfp + (0, 1)$ for each $\bfp \in \pi \smallsetminus \{\max \pi\}$. Let $\Pi_{\bfu, \bfv}$ denote the set of all up-right paths $\pi$ with $\min \pi = \bfu$ and $\max \pi = \bfv$. Given 
a collection of real weights $\w = \{\w_\bfp: \bfp \in \bbZ^2\} \in \bbR^{\bbZ^2}$, define the associated last-passage time by 
\begin{align}
\label{E:Lpp}
\Lp_{\bfu, \bfv}[\w] = \max_{\pi \in \Pi_{\bfu, \bfv}} \sum_{\bfp \in \pi} \w_\bfp, 
\end{align}
which depends only on the restriction $\w|_{[\bfu, \bfv]}$ of the weights to the rectangular grid $[\bfu, \bfv]$.
When the inequality $\bfu \le \bfv$ fails, we set
\begin{align}
\label{E:LppInft}
\Lp_{\bfu, \bfv}[\w] = -\infty. 
\end{align}

Define the last-passage increments with respect to the initial point by 
\begin{align}
\label{E:IncInit}
\begin{split}
\ul{\I}_{\bfu, \bfv}[\w] &= \Lp_{\bfu, \bfv}[\w]-\Lp_{\bfu+(1, 0), \bfv}[\w], \\ 
\ul{\J}_{\bfu, \bfv}[\w] &= \Lp_{\bfu, \bfv}[\w]-\Lp_{\bfu+(0, 1), \bfv}[\w], 
\end{split}
\end{align}
and with respect to the terminal point by 
\begin{align}
\label{E:IncTerm}
\begin{split}
\ol{\I}_{\bfu, \bfv}[\w] &= \Lp_{\bfu, \bfv}[\w]-\Lp_{\bfu, \bfv-(1, 0)}[\w], \\ 
\ol{\J}_{\bfu, \bfv}[\w] &= \Lp_{\bfu, \bfv}[\w]-\Lp_{\bfu, \bfv-(0, 1)}[\w].  
\end{split}
\end{align}
Definitions \eqref{E:IncInit} and \eqref{E:IncTerm} imply that the last-passage increments satisfy the recursions 
\begin{align}
\label{E:IncRec}
\begin{split}
\ul{\I}_{\bfx, \bfy}[\w] &= \w_\bfx + \max \{\ul{\I}_{\bfx + (0, 1), \bfy}[\w]- \ul{\J}_{\bfx + (1, 0), \bfy}[\w], 0\}, \\
\ul{\J}_{\bfx, \bfy}[\w] &= \w_\bfx + \max \{\ul{\J}_{\bfx + (1, 0), \bfy}[\w]-\ul{\I}_{\bfx + (0, 1), \bfy}[\w], 0\}, \\
\ol{\I}_{\bfx, \bfy}[\w] &= \w_\bfy + \max \{\ol{\I}_{\bfx, \bfy-(0, 1)}[\w]- \ol{\J}_{\bfx, \bfy-(1, 0)}[\w], 0\}, \\
\ol{\J}_{\bfx, \bfy}[\w] &= \w_\bfy + \max \{\ol{\J}_{\bfx, \bfy-(1, 0)}[\w]-\ol{\I}_{\bfx, \bfy-(0, 1)}[\w], 0\}
\end{split}
\end{align}
for $\bfx, \bfy \in [\bfu, \bfv]$ with $\bfx + (1, 1) \le \bfy$. One also has the identities  
\begin{align}
\label{E:wRec}
\min \{\ul{\I}_{\bfx, \bfy}, \ul{\J}_{\bfx, \bfy}\}  = \w_\bfx \quad \text{ and } \quad \min \{\ol{\I}_{\bfx, \bfy}, \ol{\J}_{\bfx, \bfy}\} = \w_\bfy \quad \text{ for } \bfx < \bfy,  
\end{align}
which recover the weights from the increments. 

The following lemma recalls the monotonicity of the increments $\ul{\I}$ and $\ul{\J}$ with respect to the terminal vertices. 
\begin{lem}
[{\cite[Lemma 6.2]{Rass_18} and \cite[Lemma 4.6]{sepp-cgm-18}}]
\label{L:Comp}
The following inequalities hold for $\bfx, \bfy \in [\bfu, \bfv]$ with $\bfx \le \bfy$.
\begin{enumerate}[\normalfont (a)]
\item \label{L:Comp_a}
If $\bfy + (1, 0) \le \bfv$ then $\ul{\I}_{\bfx, \bfy}[\w] \ge \ul{\I}_{\bfx, \bfy + (1, 0)}[\w]$. 
\item \label{L:Comp_b}
If $\bfy + (0, 1) \le \bfv$ then $\ul{\I}_{\bfx, \bfy}[\w] \le \ul{\I}_{\bfx, \bfy + (0, 1)}[\w]$.
\item \label{L:Comp_c}
If $\bfy + (1, 0) \le \bfv$ then $\ul{\J}_{\bfx, \bfy}[\w] \le \ul{\J}_{\bfx, \bfy + (1, 0)}[\w]$. 
\item \label{L:Comp_d}
If $\bfy + (0, 1) \le \bfv$ then $\ul{\J}_{\bfx, \bfy}[\w] \ge \ul{\J}_{\bfx, \bfy + (0, 1)}[\w]$.
\end{enumerate}
\end{lem}

\subsection{Joint distribution of Busemann process in exponential LPP}
\label{Ss:BusMar}
We now consider random weights, specifically independent $\Exp\{1\}$ random variables $\omega = \{\omega_{\bfv}: \bfv \in \bbZ_{>0}^2\}$. Proposition \ref{P:BusFn} below recalls the limits of the increments $\ul{\I}_{\bfu, \bfv}[\omega]$ and $\ul{\J}_{\bfu, \bfv}[\omega]$ when the endpoint $\bfv$ tends to infinity in a given  
direction represented here by an inverse-slope parameter  
$r \in [0, \infty]$. As discussed in the introduction, the limits $\sI_{\bfu}^r[\omega]$ and $\sJ_{\bfu}^r[\omega]$ are the values of the Busemann function in the direction $r$ for the edges $\{\bfu, \bfu+(1, 0)\}$ and $\{\bfu, \bfu+(0, 1)\}$, respectively. Since the Busemann function is entirely determined from these values through additivity, there is no loss in restricting attention to pairs of nearest-neighbor vertices. In part \eqref{P:BusFn_b} of the proposition, the exponential rates are expressed in terms of the function
\begin{align}
\label{E:zeta}
\zeta(r) = \frac{\sqrt{r}}{1+\sqrt{r}},
\end{align}
extended continuously to $[0, \infty]$ so that $\zeta(0) = 0$ and $\zeta(\infty) = 1$. Concerning part \eqref{P:BusFn_c}, call a set $\nu \subset \bbZ^2$ a \emph{down-right path} if its reflection $\{(i, -j): (i, j) \in \nu\}$ across the horizontal axis is an up-right path. 

\begin{prop} 
\label{P:BusFn}
Fix $r  \in [0, \infty]$ and $\bfu \in \bbZ^2_{>0}$. The following statements hold. 
\begin{enumerate}[\normalfont (a)]\itemsep=4pt 
\item \textup{(Existence of Busemann functions.)} \label{P:BusFn_a} 
There exist random variables $\sI_{\bfu}^r[\omega]$ and $\sJ_{\bfu}^r[\omega]$ such that the   limits  
\begin{align} \label{E:def_Bus}
\lim_{k \to \infty} \ul{\I}_{\bfu, \bfv_k}[\omega] = \sI^r_\bfu[\omega] \quad \text{ and } \quad \lim_{k \to \infty} \ul{\J}_{\bfu, \bfv_k}[\omega] = \sJ^r_\bfu[\omega]
\end{align}
hold almost surely, simultaneously for all sequences $\bfv_k = (m_k, n_k) \in \bbZ^2$ such that \linebreak  $\min \{m_k, n_k\} \to \infty$ and $\dfrac{m_k}{n_k} \to r$ as $k \to \infty$. 


\item \textup{(Marginal distributions.)} \label{P:BusFn_b}
$\sI_{\bfu}^r[\omega] \sim \Exp\{\zeta(r)\}$, and $\sJ_{\bfu}^r[\omega] \sim \Exp\{1-\zeta(r)\}$. 


\item \textup{(Independence along down-right paths.)} \label{P:BusFn_c}
For any down-right path $\nu \subset \bbZ_{>0}^2$, the collection 
\begin{align*}
\{\sI_{\bfu}^r[\omega]: \bfu, \bfu + (1, 0) \in \nu\} \cup \{\sJ_{\bfu}^r[\omega]: \bfu, \bfu + (0, 1) \in \nu\}
\end{align*}
is independent. 
\end{enumerate}
\end{prop}

\begin{rem}[Justification of Proposition~\ref{P:BusFn}]
\label{R:BusFn_HV}   
The main case $r \in (0, \infty)$ of the proposition is contained in \cite[Theorem~4.2]{sepp-cgm-18}, with the edge cases $r \in \{0, \infty\}$ being an easy extension. In fact, for 
directions $r \in \{0, \infty\}$, the corresponding Busemann functions are trivial in the sense that 
$\sI_{\bfu}^0[\omega] = \infty = \sJ_{\bfu}^{\infty}[\omega]$ and $\sI_{\bfu}^\infty[\omega] = \omega_{\bfu} = \sJ_{\bfu}^{0}[\omega]$. 

For bibliographic completeness, we point out that parts \eqref{P:BusFn_b} and \eqref{P:BusFn_c} 
were previously recorded within \cite[Lemma 3.3]{Cato_Pime_13}. The article \cite{Cato_Pime_13} attributed part \eqref{P:BusFn_a} to \cite{coup-11, Ferr_Pime_05} because the existence of $r$-directed Busemann functions follows from the existence, uniqueness, and coalescence properties of $r$-directed semi-infinite geodesics established in those works. 
For instance, an argument to this effect appeared earlier in the setting of Poissonian LPP \cite[Theorem 1.1]{Wuth_02}.

Finally, we point out that \eqref{E:def_Bus} fails for a random, countably infinite, dense set of directions $r$ depending on $\omega$.
This is related to the non-uniqueness of semi-infinite geodesics in exceptional directions.
Since we are concerned with \textit{finite}-dimensional distributions of the Busemann process (defined below), such exceptional behaviors will not play a role in this paper.
For precise results, see \cite[Theorems 3.10 and 3.11]{janj-rass-sepp-23}.
\qedex 
\end{rem}

Our goal is to dramatically extend parts \eqref{P:BusFn_b} and \eqref{P:BusFn_c} of Proposition \ref{P:BusFn} by describing all finite-dimensional distributions of the \emph{Busemann process} 
\begin{align}
\label{E:BusPr}
\{\sI_{\bfu}^r[\omega], \sJ_{\bfu}^r[\omega]: \bfu \in \bbZ^2_{>0} \text{ and } r\in\R_{>0}\}. 
\end{align}
Allowing the axis directions $\{0, \infty\}$ in \eqref{E:BusPr} would involve a simple extension of our description as explained in Remark \ref{R:BusMarHV} below. 
By translation invariance of the i.i.d.\ environment $\omega$, it suffices to consider vertices belonging to a finite grid $[k]\times[\ell]$ whose lower left corner is $(1,1)$ and upper right corner is $(k,\ell)$. When discussing horizontal increments such as $\ul{\I}_{\bfu, \bfv}$ or $\sI_\bfu^r$, we demand that both $\bfu$ and $\bfu+(1,0)$ belong to this grid, meaning $\bfu$ belongs to 
\begin{align} \label{E:rightset}
\rightset_{k,\ell} = [k-1]\times[\ell].
\end{align} 
For vertical increments, we will use $\bfv$ instead of $\bfu$, and require 
that $\bfv$ belongs to
\begin{align} \label{E:upset}
\upset_{k,\ell} = [k]\times[\ell-1].
\end{align}
Finally, fix a finite list of directions $\bfr = (r_1<r_2<\cdots<r_d)\in \bbR^d_{>0}$.
Our aim is to describe the joint distribution of
\begin{align}\label{E:BuseBox}
\Bigl(\Bh^{r_p}_{\bfu}[\omega],\Bv^{r_p}_{\bfv}[\omega]: \bfu\in \rightset_{k,\ell},\,\bfv\in\upset_{k,\ell}, \,p \in [d]\Bigr).
\end{align}

The theorem below states that the collection \eqref{E:BuseBox} is distributionally identical to the increments of an inhomogeneous last-passage process on the slightly larger grid $[k+d-1]\times[\ell+d-1]$.
The setup is illustrated in Figure~\ref{F:eta} and defined as follows.
First define parameter sequences $(a_i)_{i \in [k+d - 1]}$ and $(b_j)_{j \in [\ell + d - 1]}$ by 
\begin{subequations} \label{E:def_eta}
\begin{align}
a_i &= \begin{cases}
0 &1 \le i  < k \\
-\zeta(r_{i - k + 1}) &k \le i \le k+d - 1,
\end{cases} \label{E:def_eta_a}\\
b_j &= \begin{cases}
1 &1 \le j < \ell \\
\rlap{$\zeta(r_{\ell + d - j})$}\hphantom{-\zeta(r_{i - k + 1})} &\ell \le j \le \ell + d - 1.
\end{cases}
\label{E:def_eta_b}
\end{align}
Then consider a collection of independent weights $\eta = \bigl\{\eta_{(i, j)}: i \in [k+d-1], j \in [\ell+d-1]\bigr\}$ (depending on $d$, $k$, $\ell$ and $\bfr$) such that 
\begin{equation}\label{E:auxw}
\begin{aligned}
    \eta_{(i,j)} &\sim \Exp\{a_i + b_j\} \quad\text{if}\quad a_i + b_j > 0, \\
    \eta_{(i,j)} &= 0 \qquad\qquad\qquad\text{otherwise.}
\end{aligned}
\end{equation} 
\end{subequations}
Figures \ref{F:eta} and \ref{F:BusMar}   illustrate the next theorem. 

\begin{figure}[h]
    \centering
    \includegraphics[width=0.61\linewidth]{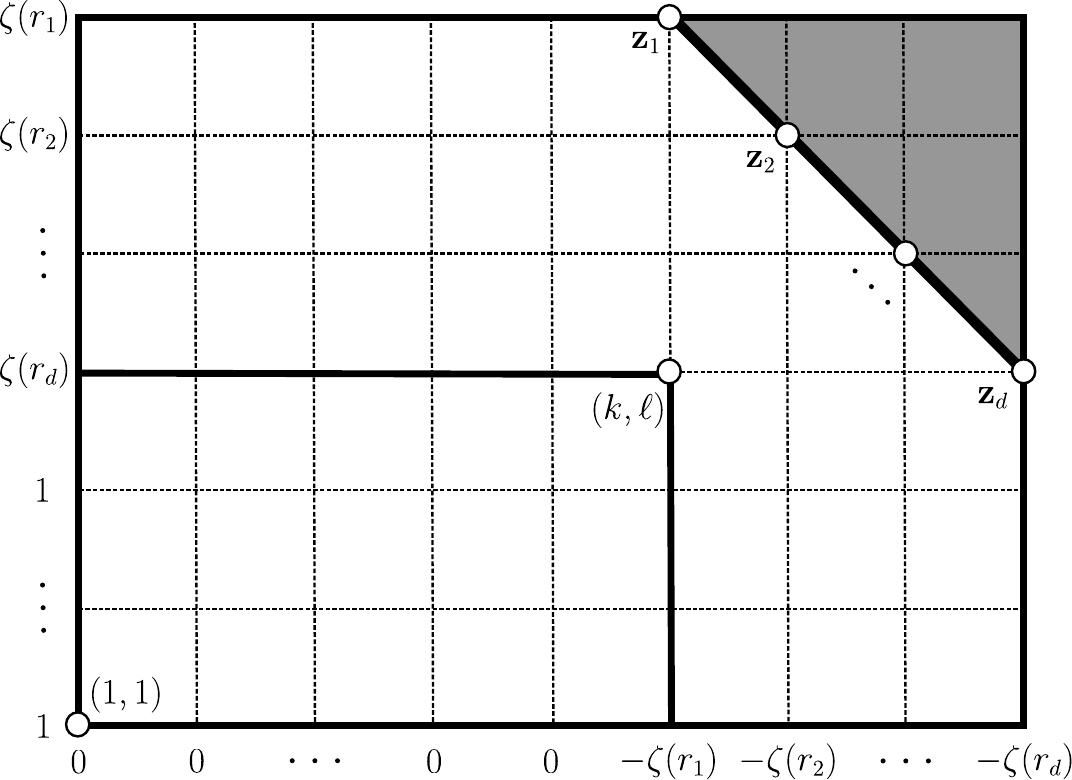}
    \caption{\small  Illustration of the geometric arrangement in Theorem \ref{T:BusMar} with $d=4$, $k = 6$, and $\ell = 4$. 
    Theorem \ref{T:BusMar} gives the distribution of the Busemann functions in the $[k] \times [\ell]$ grid with lower left corner at  $(1,1)$ and upper right corner at $(k,\ell)$, by considering inhomogeneous exponential LPP terminating at the vertices $\bfz_1,\dots,\bfz_d$.
    Outside of the shaded region, the weight at vertex $(i,j)$ is exponential with rate $a_i + b_j$, where the value of $a_i$ is shown at the bottom of column $i\in[k+d-1]$, and the value of $b_j$ is shown to the left of row $j\in[\ell+d-1]$.
    Inside the shaded region (including its boundary), we have $a_i+b_j\le0$, so weights are instead set to $0$.}
    \label{F:eta}
\end{figure}

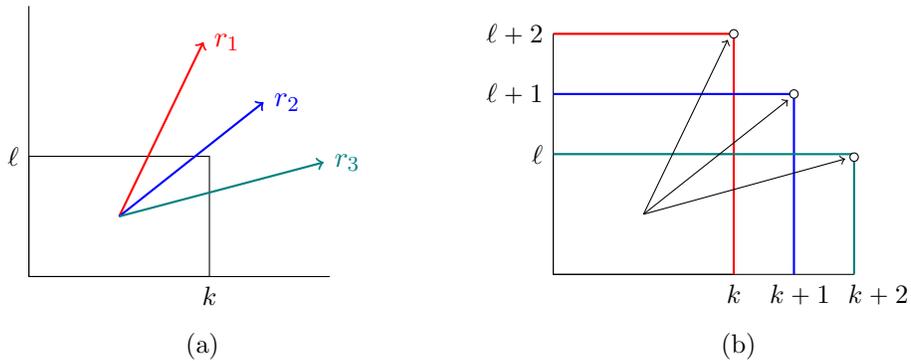
\begin{figure}[h]
\begin{subfigure}[t]{0.45\textwidth}
\hspace*{0.25in}
\begin{tikzpicture}[scale = 0.8]
\draw[](0, 0)--(0, 4.5);
\draw[](0, 0)--(5, 0);
\draw[](0, 0)rectangle(3, 2);
\draw[red, thick, ->](1.5, 1)--(3.5, 5)node[right]{$r_1$};
\draw[blue, thick, ->](1.5, 1)--(4, 4.5)node[right]{$r_2$};
\draw[teal, thick, ->](1.5, 1)--(5.5, 1.5)node[right]{$r_3$};
\draw(3, 0)node[below]{\small $k$};
\draw(0, 2)node[left]{\small $\ell$};
\end{tikzpicture}
\subcaption{}
\end{subfigure}
\begin{subfigure}[t]{0.45\textwidth}
\begin{tikzpicture}[scale = 0.8]
\draw[](0, 0)--(0, 4);
\draw[](0, 0)--(5, 0);
\draw[](0, 0)rectangle(3, 2);
\draw[](0, 4)--(3, 4);
\draw[thick, red](3, 0)--(3, 4);
\draw[thick, red](0, 4)--(3, 4);
\draw[thick, blue](4, 0)--(4, 3);
\draw[thick, blue](0, 3)--(4, 3);
\draw[thick, teal](5, 0)--(5, 2);
\draw[thick, teal](0, 2)--(5, 2);
\draw[fill=white](3, 4)circle(0.07);
\draw[fill=white](4, 3)circle(0.07);
\draw[fill=white](5, 2-0.05)circle(0.07);
\draw[->](1.5, 1)--(3-0.1, 4-0.1);
\draw[->](1.5, 1)--(4-0.1, 3-0.1);
\draw[->](1.5, 1)--(5-0.15, 2-0.08);
\draw(3, 0)node[below]{\small $k$};
\draw(4.1, 0)node[below]{\small $k+1$};
\draw(5.4, 0)node[below]{\small $k+2$};
\draw(0, 2)node[left]{\small $\ell$};
\draw(0, 3)node[left]{\small $\ell+1$};
\draw(0, 4)node[left]{\small $\ell+2$};
\end{tikzpicture}
\subcaption{}
\end{subfigure}
\caption{\small Illustration of the conclusion of Theorem \ref{T:BusMar} with $d = 3$.  The following two collections of random variables are equal in distribution. (a) The Busemann increments within $[k] \times [\ell]$ for directions 
$r_1 < r_2 < r_3$ (red, blue, and teal 
arrows).  These increments depend on infinitely many variables in the i.i.d.\ environment $\omega$. (b) The last-passage increments (with respect to initial points within $[k] \times [\ell]$) to the terminal points $(k, \ell+2)$, $(k+1, \ell+1)$, and $(k+2, \ell)$, computed with $\eta$-weights. These increments depend on only finitely many variables in the inhomogeneous environment $\eta$. Zero weights are marked with small circles, and all other weights are exponentially distributed with rate parameters computed as in Figure~\ref{F:eta}. 
}
\label{F:BusMar}
\end{figure}

\begin{thm}[Joint distribution of Busemann functions]
\label{T:BusMar}
Let $d, k, \ell \in \bbZ_{>0}$, and let $(r_p)_{p \in [d]} \in \bbR^d_{>0}$ be a strictly  increasing sequence. For $p \in [d]$, define $\bfz_p = (k+p-1,\ell + d-p)$ so that $\bfz_1,\ldots,\bfz_d$ are the points along the antidiagonal segment
\[
\{(i,j): i+j = k+\ell + d- 1\text{ and }k \le i \le k+d-1\},
\]
ordered from upper left to lower right. 
Then, we have the following distributional identity:   \begin{align*}
&\Bigl(\Bh^{r_p}_{\bfu}[\omega],\Bv^{r_p}_{\bfv}[\omega]: \bfu\in \rightset_{k,\ell},\,\bfv\in\upset_{k,\ell}, \,p \in [d]\Bigr)\deq \Bigl(\ul{\I}_{\bfu, \mathbf z_p}[\eta],\ul{\J}_{\bfv, \mathbf z_p}[\eta]: \bfu\in \rightset_{k,\ell},\,\bfv\in\upset_{k,\ell}, \,p \in [d] \Bigr).
\end{align*}
That is, the Busemann increments in the i.i.d.\ environment $\omega$ are equal in joint law to the last-passage increments in the inhomogeneous environment $\eta$ from \eqref{E:def_eta} with terminal points $(\mathbf z_p)_{p\in[d]}$. 
\end{thm}

Before proceeding to some examples, we make a few remarks about Theorem \ref{T:BusMar}.  

\begin{rem}[Direction convention]
\label{R:DirConv}
The Busemann functions in \eqref{E:def_Bus} are defined by sending terminal points to the infinite northeast. 
The opposite convention is taken in \cite{Fan_Sepp_20}, where initial points are sent to the infinite southwest. The i.i.d.\ weights give rise to reflection invariance, so the two scenarios are equivalent.
\qedex
\end{rem}

\begin{rem}[Including the axis directions]
\label{R:BusMarHV}
One can obtain a version of Theorem \ref{T:BusMar} that includes the axis directions $\{0, \infty\}$ as follows. First, apply the theorem with the grid $[k+1] \times [\ell+1]$ to obtain 
\begin{align}
\label{E:155}
\begin{split}
\Bigl(\Bh^{r_p}_{\bfu}[\omega],\Bv^{r_p}_{\bfv}[\omega]\Bigr)_{\bfu, \bfv, p} \deq \Bigl(\ul{\I}_{\bfu, \mathbf z_p}[\eta],\ul{\J}_{\bfv, \mathbf z_p}[\eta]\Bigr)_{\bfu, \bfv, p},    
\end{split}
\end{align}
where the indices take values $\bfu \in \rightset_{k+1,\ell+1}$, $\bfv\in\upset_{k+1,\ell+1}$, $p \in [d]$, and the $\eta$-weights are defined on the grid $[k+d] \times [\ell+d]$ using the given directions $r_1 < \dotsc < r_d$. By the weight-recovery property \eqref{E:wRec}, we have
\begin{align}
\label{E:153}
\begin{split}
\omega_{\bfx} = \min \{\Bh^{r_p}_{\bfx}[\omega],\Bv^{r_p}_{\bfx}[\omega]\} \quad \text{ and } \quad \eta_\bfx = \min \{\ul{\I}_{\bfx, \mathbf z_p}[\eta],\ul{\J}_{\bfx, \mathbf z_p}[\eta]\}
\end{split}    
\end{align}
for each $\bfx \in \rightset_{k+1, \ell+1} \cap \upset_{k+1, \ell+1} = [k] \times [\ell]$. Therefore, \eqref{E:155} implies 
\begin{align}
\label{E:154}
\begin{split}
&\Bigl(\Bh^{r_p}_{\bfu}[\omega],\Bv^{r_p}_{\bfv}[\omega], \omega_\bfx\Bigr)_{\bfu, \bfv, \bfx, p} \deq \Bigl(\ul{\I}_{\bfu, \mathbf z_p}[\eta],\ul{\J}_{\bfv, \mathbf z_p}[\eta], \eta_\bfx\Bigr)_{\bfu, \bfv, \bfx, p},
\end{split}
\end{align}
where the additional index $\bfx$ takes values in the restricted grid $[k] \times [\ell]$.
As noted in Remark~\ref{R:BusFn_HV}, we have $\Bh_{\bfu}^\infty[\omega] = \omega_{\bfu} = \Bv_{\bfu}^{0}[\omega]$. 
Using this with \eqref{E:154} leads to   
\begin{align}
\label{E:156}
\begin{split}
&\Bigl(\Bh^{r_p}_{\bfu}[\omega],\Bv^{r_p}_{\bfv}[\omega], \Bh^{\infty}_{\bfu'}[\omega],\Bv^{0}_{\bfv'}[\omega]\Bigr)_{\bfu, \bfv, \bfu', \bfv', p} \deq \Bigl( \ul{\I}_{\bfu, \mathbf z_p}[\eta],\ul{\J}_{\bfv, \mathbf z_p}[\eta], \eta_{\bfu'},\eta_{\bfv'}\Bigr)_{\bfu, \bfv, \bfu', \bfv', p},
\end{split}
\end{align}
where $\bfu \in \rightset_{k+1, \ell+1}$, $\bfv\in\upset_{k+1,\ell+1}$, $\bfu'\in \rightset_{k, \ell}$, $\bfv' \in \upset_{k, \ell}$, and $p \in [d]$. The preceding identity can be viewed as an analogue of Theorem \ref{T:BusMar} that allows axis directions. The left-hand side of \eqref{E:156} excludes $\Bh^{0}_{\bfu'}[\omega]$ and $\Bv^{\infty}_{\bfv'}[\omega]$ because these are $\infty$-valued. \qedex
\end{rem}

\begin{rem}[Recovery  
of Proposition \ref{P:BusFn}\eqref{P:BusFn_b} and \eqref{P:BusFn_c}]
\label{R:OneDir}
In the case $d = 1$, we have $\mathbf z_1 = (k,\ell)$ (see Figure \ref{fig:d=1case}). 
On the north and east boundaries, the LPP increments to $\bfz_1$ are simply the weights:
\begin{align*}
\ul \I_{(i,\ell), (k, \ell)}[\eta] &= \Lp_{(i,\ell),(k,\ell)}[\eta] - \Lp_{(i+1,\ell),(k,\ell)}[\eta] = \eta_{(i,\ell)}\sim\Exp\{\zeta(r_1)\} \quad\text{for $i\in[k-1]$},\\
\ul \J_{(k,j), (k, \ell)}[\eta] &= \Lp_{(k,j),(k,\ell)}[\eta] - \Lp_{(k,j+1),(k,\ell)}[\eta] = \eta_{(k,j)}\sim\Exp\{1 - \zeta(r_1)\} \quad\text{for $j\in[\ell-1]$}. 
\end{align*}
Therefore, Theorem~\ref{T:BusMar} implies Proposition \ref{P:BusFn}\eqref{P:BusFn_b}.
Similarly to the proof of \cite[Lemma 4.2]{bala-cato-sepp}, part~\eqref{P:BusFn_c} follows inductively from part~\eqref{P:BusFn_b} by corner-flipping, where the base case is provided by the independence of the increments on the boundary.
\qedex

\begin{figure}[h]
    \centering
    \includegraphics[width=0.5\linewidth]{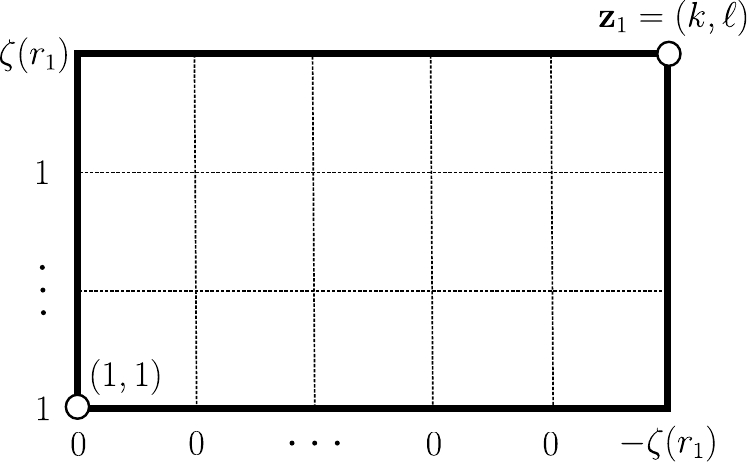}
    \caption{\small Collapsing of Figure \ref{F:eta} when $d = 1$.}
    \label{fig:d=1case}
\end{figure}
\end{rem}

\begin{rem}[Coupling of stationary models] 
\label{R:StCoup}
The $\eta$-weights on or below the anti-diagonal $i + j = k+\ell+d-1$ form $d$ coupled stationary inhomogeneous last-passage models. Indeed, for each $p \in [d]$, the $\eta$-weights on the grid $[(1, 1), \bfz_p]$ can be viewed as the environment of a stationary inhomogeneous exponential LPP considered in \cite[Section 4.2]{Emra_Janj_Sepp_25} 
by interpreting $\zeta(r_p)$ as the boundary parameter $z$. See Figure \ref{fig:LPP_coupling}.  The positive-temperature analogue of this coupling previously appeared in \cite{Chaumont-thesis, Barraquand-LeDoussal-2023}, and was shown there to give a stationary measure for the multi-path inverse-gamma polymer. See, specifically, Figure 20 and Lemma 4.1 in \cite{Chaumont-thesis} and Figure 4 and Proposition 3.5 in \cite{Barraquand-LeDoussal-2023}. The joint distribution of Busemann functions for multiple directions was not discussed in those works, although such a connection is conjectured around Equation (17) in \cite{Barraquand-LeDoussal-2023}.  
\qedex
\end{rem}

\begin{figure}
    \centering
    \includegraphics[width=0.5\linewidth]{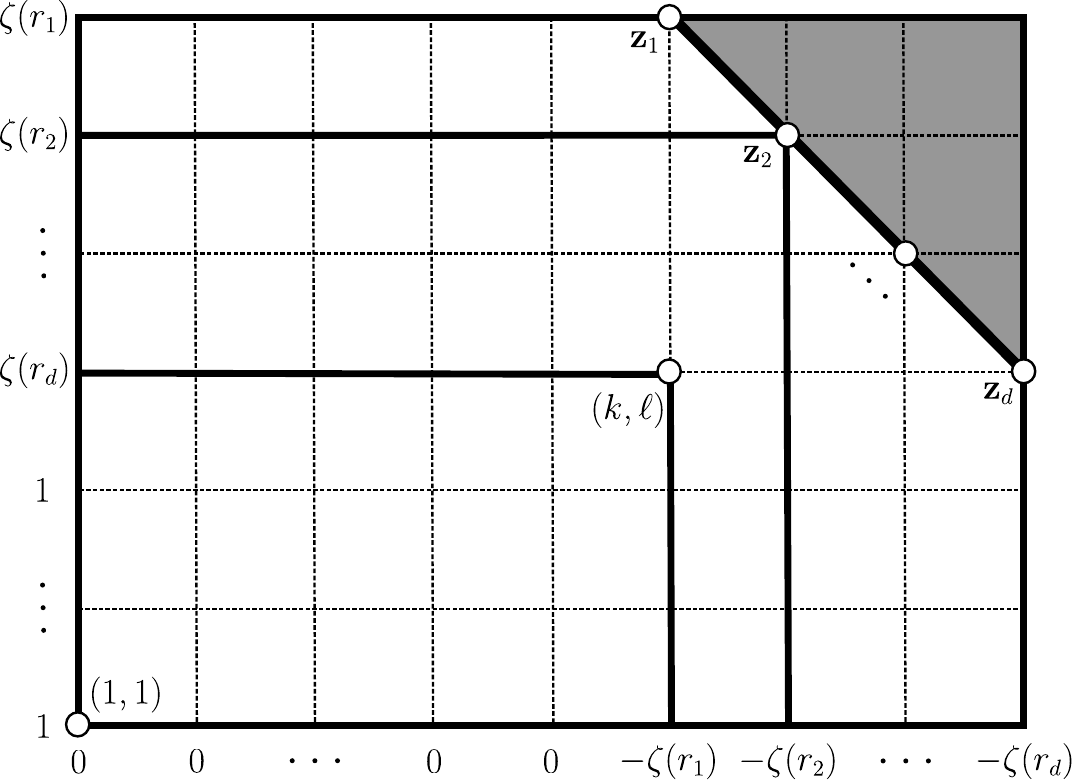}
    \caption{\small The $\eta$-weights on the grid $[(1, 1), \bfz_p]$ form a stationary inhomogeneous LPP model for each $p \in [d]$. In the picture, $p = 2$. 
    }
    \label{fig:LPP_coupling}
\end{figure}

\begin{rem}[Monotonicity of Busemann functions]
As mentioned in Section~\ref{Ss:history}, a key feature of Busemann functions is monotonicity with respect to direction: $\sI_{\bfu}^{r_1}[\omega] \ge \sI_{\bfu}^{r_2}[\omega]$ and $\sJ_{\bfu}^{r_1}[\omega] \le \sJ_{\bfu}^{r_2}[\omega]$ for $r_1 \le r_2$.
This monotonicity is implied by Lemma~\ref{L:Comp}, and is also visible from the $\eta$-side of the distributional identity in Theorem \ref{T:BusMar}. 
Indeed, an application of Lemma~\ref{L:Comp}\eqref{L:Comp_a} and then Lemma~\ref{L:Comp}\eqref{L:Comp_b} shows that 
\begin{align*}
\ul{\I}_{\bfu, \bfz_1}[\eta] 
\ge \ul{\I}_{\bfu, \bfz_1+(1,0)}[\eta] 
= \ul{\I}_{\bfu, \bfz_2+(0,1)}[\eta]
\ge \ul{\I}_{\bfu, \bfz_2}[\eta] \quad \text{for any $\bfu\in\upset_{k, \ell}$}.
\end{align*}
The analogous logic for vertical increments uses parts~\eqref{L:Comp_c} and \eqref{L:Comp_d} of Lemma~\ref{L:Comp}.
\qedex
\end{rem}

We next illustrate Theorem \ref{T:BusMar} with two quick examples in the two-direction case.

\begin{exmpl}[Upper-right corner]
Here we examine 
one of the simplest cases of our result, yet a case that is more difficult to explicitly treat 
using the previous description in \cite{Fan_Sepp_20}. 
Namely, we consider the joint distribution of the collection
\[
\Bigl(\sI_{(1,2)}^{r_1}[\omega],\sI_{(1,2)}^{r_2}[\omega],\sJ_{(2,1)}^{r_1}[\omega],\sJ_{(2,1)}^{r_2}[\omega]\Bigr) \deq \Bigl(\ul{\I}_{(1,2),\mathbf z_1}[\eta],\ul{\I}_{(1,2),\mathbf z_2}[\eta], \ul{\J}_{(2,1),\mathbf z_1}[\eta],\ul{\J}_{(2,1),\mathbf z_2}[\eta]  \Bigr).
\]
In this case, $\ell = k = d = 2$, so $\mathbf z_1 = (2,3)$ and $\mathbf z_2 = (3,2)$. 
It is already known from Proposition~\ref{P:BusFn}\eqref{P:BusFn_c} that $\sI_{(1,2)}^{r_1}[\omega]$ and $\sJ_{(2,1)}^{r_1}[\omega]$ are independent, and that $\sI_{(1,2)}^{r_2}[\omega]$ and $\sJ_{(2,1)}^{r_2}[\omega]$ are independent.
In fact, it is also known that $\sI_{(1,2)}^{r_2}[\omega]$ and $\sJ_{(2,1)}^{r_1}[\omega]$ are independent, by \cite[Theorem~2.1]{shen25}.
We can now immediately infer all of these properties by writing the LPP increments as explicit functions of five independent exponential variables.
Specifically, by comparing paths in Figure~\ref{fig:2by2} we obtain
\begin{align*}
\ul \I_{(1,2),\mathbf z_1}[\eta]
 &= \eta_{(1,2)} + \max \{\eta_{(1,3)} - \eta_{(2,2)}, 0\}, &
 \ul \I_{(1,2),\mathbf z_2}[\eta] &= \eta_{(1,2)}, \\
 \ul \J_{(2,1),\mathbf z_1}[\eta] &= \eta_{(2,1)}, &
 \ul \J_{(2,1),\mathbf z_2}[\eta] &= \eta_{(2,1)} + \max \{\eta_{(3,1)} - \eta_{(2,2)}, 0\},
 \end{align*}
 where 
 \begin{align*}
 \eta_{(1,2)} &\sim \Exp\{\zeta(r_2)\}, & \eta_{(1,3)} &\sim \Exp\{\zeta(r_1)\}, \\
 \eta_{(2,1)} &\sim \Exp\{1 - \zeta(r_1)\}, & \eta_{(3,1)} &\sim \Exp\{1 - \zeta(r_2)\},  \\
 \eta_{(2,2)} &\sim \Exp\{\zeta(r_2) - \zeta(r_1)\}, & \eta_{(2,3)} &= \eta_{(3,2)} = 0.
 \end{align*}
\qedex
\begin{figure}[h]
    \centering
    \includegraphics[width=0.3\linewidth]{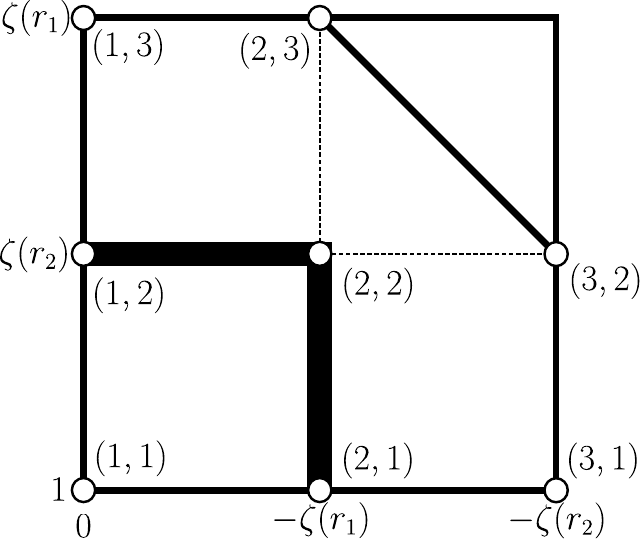}
    \caption{\small Special case of Figure \ref{F:eta} where $k = \ell = d = 2$. We are interested in the joint distribution of the Busemann functions in directions $r_1$ and $r_2$ on the two edges marked by thick lines.}
    \label{fig:2by2}
\end{figure}
\end{exmpl}

 \begin{exmpl}[Adjacent horizontal increments]
 We now investigate the joint distribution of
 \[
 \Bigl(\sI_{(1,1)}^{r_1}[\omega],\sI_{(1,1)}^{r_2}[\omega],\sI_{(2,1)}^{r_1}[\omega],\sI_{(2,1)}^{r_2}[\omega]  \Bigr) \deq \Bigl(\ul \I_{(1,1),\mathbf z_1}[\eta],\ul \I_{(1,1),\mathbf z_2}[\eta],\ul \I_{(2,1),\mathbf z_1}[\eta],\ul \I_{(2,1),\mathbf z_2}[\eta]\Bigr).
 \]
 In this case, $k = 3$, $\ell = 1$, $d = 2$, so $\mathbf z_1 = (3,2)$ and $\mathbf z_2 = (4,1)$. 
 It is already known from Proposition~\ref{P:BusFn}\eqref{P:BusFn_c} that $\sI_{(1,1)}^{r_1}[\omega]$ and $\sI_{(2,1)}^{r_1}[\omega]$ are independent, and that $\sI_{(1,1)}^{r_2}[\omega]$ and $\sI_{(2,1)}^{r_2}[\omega]$ are independent.
In fact, is also known that $\sI_{\bfu}^{r_2}[\omega]$ and $\sI_{\bfu}^{r_1}[\omega]-\sI_{\bfu}^{r_2}[\omega]$ are independent for any $\bfu\in\Z^2$; see Corollary~\ref{C:FS_IndInc}.
We can immediately infer all of these properties by writing the LPP increments as explicit functions of five independent exponential variables.
Specifically, by comparing paths in Figure~\ref{fig:3by1} (the recursions in \eqref{E:IncRec} simplify the calculation), we obtain
 \begin{align*}
 \ul \I_{(1,1),\mathbf z_1}[\eta] 
 &= \eta_{(1,1)} + \max \{\eta_{(1,2)} - \ul{\J}_{(2, 1), \bfz_1}[\eta], 0\}, &
 \ul \I_{(1,1),\mathbf z_2}[\eta] &= \eta_{(1,1)}, \\
 \ul \I_{(2,1),\mathbf z_1}[\eta] &= \eta_{(2,1)} +\max\{\eta_{(2,2)} - \eta_{(3,1)}, 0\}, &
 \ul \I_{(2,1),\mathbf z_2}[\eta] &= \eta_{(2,1)},
 \end{align*}
 where $\ul{\J}_{(2, 1), \bfz_1}[\eta] = \eta_{(2,1)} + \max\{\eta_{(3,1)} - \eta_{(2,2)}, 0\}$, and
\begin{align*}
\eta_{(1,1)} &\sim \Exp\{\zeta(r_2)\}, & \eta_{(1,2)} &\sim \Exp\{\zeta(r_1)\}, \\
\eta_{(2,1)} &\sim \Exp\{\zeta(r_2)\}, & \eta_{(2,2)} &\sim \Exp\{\zeta(r_1)\}, \\
\eta_{(3,1)} &\sim \Exp\{\zeta(r_2) - \zeta(r_1)\}, & \eta_{(3,2)} &= \eta_{(4,1)} = 0.
\end{align*}
Also, it follows from 
\cite[Lemma 4.1]{bala-cato-sepp} that 
$\ul{\J}_{(2, 1), \bfz_1}[\eta]$ and $\ul \I_{(2,1),\mathbf z_1}[\eta]$ 
are independent exponential random variables with rates $\zeta(r_2) - \zeta(r_1)$ and $\zeta(r_1)$, respectively. This increment pair 
is independent of $(\eta_{(1,1)},\eta_{(1,2)})$, but not independent of $\eta_{(2,1)}$. \qedex    

 \begin{figure}[h]
     \centering
     \includegraphics[width=0.5\linewidth]{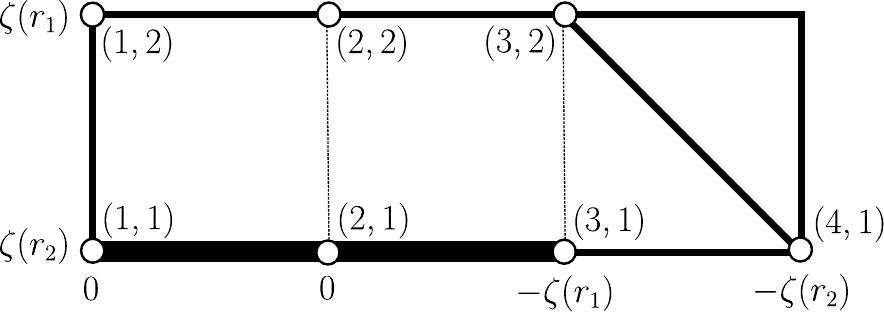}
     \caption{\small Special case of Figure \ref{F:eta} where $k = 3$, $\ell =1$, and  $d = 2$. We are interested in the joint distribution of the Busemann functions along the two thick horizontal edges.}
     \label{fig:3by1}
 \end{figure}
 \end{exmpl}
 
\subsection{Busemann process on a lattice edge} \label{S:edgeBus}
We demonstrate the utility of Theorem \ref{T:BusMar} by recovering the distribution of the Busemann process as a function of direction, on a single horizontal lattice edge.
We present below two corollaries that together give a complete description of the process.
In \cite{janj-rass-sepp-23} this description led to the complete characterization of the existence, uniqueness, and coalescence of semi-infinite geodesics in exponential LPP.

The following results were previously realized in \cite{Fan_Sepp_20} by utilizing the queueing formulation of the Busemann process, as well as an auxiliary triangular array of bi-infinite sequences of random variables. 
By contrast, our proofs will use only a couple finite-dimensional calculations and a straightforward induction on the number of directions.

\begin{cor}[Distribution of a single increment]
\label{C:FS_IncCDF}
Let $r, s \in \bbR_{>0}$ with $r < s$. Then 
\begin{align*}
\bfP\{\sI_{(1, 1)}^{r}[\omega]-\sI_{(1, 1)}^{s}[\omega] \ge x\} = \one_{\{x = 0\}} \cdot \frac{\zeta(r)}{\zeta(s)} + \bigg(1-\frac{\zeta(r)}{\zeta(s)}\bigg) \cdot \exp\{-x \zeta(r)\} \quad \text{ for } x \in \bbR_{\ge 0}. 
\end{align*}
\end{cor}

\begin{cor}[Independent increments with respect to direction]
\label{C:FS_IndInc} For any $d \in \bbZ_{>0}$ and increasing sequence $(r_p)_{p \in [d]} \in \bbR^d_{>0}$, 
the following random variables are independent:
\begin{align*}
\sI_{(1, 1)}^{r_1}[\omega]-\sI_{(1, 1)}^{r_2}[\omega], \quad
\sI_{(1, 1)}^{r_2}[\omega]-\sI_{(1, 1)}^{r_3}[\omega], \quad \ldots \quad
\sI_{(1, 1)}^{r_{d-1}}[\omega]-\sI_{(1, 1)}^{r_d}[\omega], \quad
\sI_{(1, 1)}^{r_d}[\omega].
\end{align*}
\end{cor}

The two corollaries together give a complete characterization of the Busemann process on a single lattice edge (here the edge is $\{(1,1),(2,1)\}$).
This characterization was reinterpreted in \cite[Theorem 3.4]{Fan_Sepp_20} to express $(\sI_{\bfu}^r)_{r\in\R_{>0}}$ as a marked Poisson point process.
It remains an open problem to give the joint distribution of these point processes across multiple $\bfu$.

\begin{proof}[Proof of Corollary \ref{C:FS_IncCDF}]
Taking $d = 2$, $k = 2$, $\ell = 1$, $r_1 = r$ and $r_2 = s$ in Theorem \ref{T:BusMar}, one obtains that 
\begin{align}
\label{E:82}
(\sI_{(1, 1)}^{r}[\omega], \sI_{(1, 1)}^{s}[\omega]) \deq (\ul{\I}_{(1, 1), (2, 2)}[\eta], \ul{\I}_{(1, 1), (3, 1)}[\eta]). 
\end{align}
By \eqref{E:auxw}, the $\eta$-weights are defined on the grid $[3] \times [2]$ and are independent such that 
\begin{align}
\label{E:83}
\eta_{(1, 1)} \sim \Exp\{\zeta(s)\}, \quad \eta_{(2, 1)} \sim \Exp\{\zeta(s)-\zeta(r)\}, \quad \eta_{(1, 2)} \sim \Exp\{\zeta(r)\}, 
\end{align}
and the remaining $\eta$-weights are zero. The increments on the right-hand side of \eqref{E:82} can be written as 
\begin{align}
\label{E:84}
\begin{split}
\ul{\I}_{(1, 1), (2, 2)}[\eta] &= \eta_{(1, 1)} + \max\{\eta_{(1, 2)}-\eta_{(2, 1)}, 0\}, \\ 
\ul{\I}_{(1, 1), (3, 1)}[\eta] &= \eta_{(1, 1)}. 
\end{split}
\end{align}
As a consequence of \eqref{E:82} and \eqref{E:84}, 
\begin{align}
\label{E:85}
\begin{split}
\sI_{(1, 1)}^{r}[\omega] - \sI_{(1, 1)}^{s}[\omega] &\deq \ul{\I}_{(1, 1), (2, 2)}[\eta] - \ul{\I}_{(1, 1), (3, 1)}[\eta] 
= \max\{\eta_{(1, 2)}-\eta_{(2, 1)}, 0\}. 
\end{split}
\end{align}
For any $x \in \bbR_{>0}$, it follows from \eqref{E:83} and \eqref{E:85} that 
\begin{align*}
&\bfP\{\sI_{(1, 1)}^{r}[\omega] - \sI_{(1, 1)}^{s}[\omega] \ge x\} = \bfP\{\eta_{(1, 2)}-\eta_{(2, 1)} \ge x\}, \\ 
&= (\zeta(s)-\zeta(r)) \int_{0}^\infty \dd u \exp\{-(\zeta(s)-\zeta(r))u\}\int_{u+x}^\infty \dd v \zeta(r) \exp\{-\zeta(r) v\} \\ 
&= (\zeta(s)-\zeta(r)) \exp\{-\zeta(r) x\} \int_0^\infty \dd u \exp\{-\zeta(s) u\} 
= \bigg(1-\frac{\zeta(r)}{\zeta(s)}\bigg) \exp\{-\zeta(r) x\}, 
\end{align*}
which implies the claim. 
\end{proof}

\begin{proof}[Proof of Corollary \ref{C:FS_IndInc}]
Arguing inductively, let $d > 1$ and assume that the claimed independence holds for $d-1$ directions. The base case $d = 1$ is trivial. To verify the claim for $d$ directions, pick any increasing sequence $(r_p)_{p \in [d]} \in \bbR_{>0}^d$. It follows from Theorem \ref{T:BusMar} applied with $k = 2$ and $\ell = 1$ that 
\begin{align}
\label{E:87}
\begin{split}
&\{\sI_{(1, 1)}^{r_p}[\omega]-\one_{\{p < d\}}\sI_{(1, 1)}^{r_{p+1}}[\omega]: p \in [d]\} \\
&\deq \{\ul{\I}_{(1, 1), (p+1, d+1-p)}[\eta] - \one_{\{p < d\}}\ul{\I}_{(1, 1), (p+2, d-p)}[\eta]: p \in [d]\}
\end{split}
\end{align}
where $\eta$ denotes independent weights on $[d+1] \times [d]$ with the marginals 
\begin{align}
\label{E:106}
\begin{split}
\eta_{(i, j)} &\sim \Exp\{\zeta(r_{d+1-j})\} \qquad \quad \hspace{0.85cm} \text{ if } i = 1, \\ 
\eta_{(i, j)} &\sim \Exp\{\zeta(r_{d+1-j})-\zeta(r_{i-1})\} \quad \text{ if } i > 1 \text{ and } i+j < d+2, \\ 
\eta_{(i, j)} &= 0 \qquad \qquad \qquad \qquad \hspace{1.3cm} \text{ otherwise. }
\end{split}
\end{align}
Hence, it suffices to prove the independence of the latter collection in \eqref{E:87}. 

Definition \eqref{E:IncInit} implies that the $p = d$ term on the second line of \eqref{E:87} is
\begin{align}
\label{E:91}
\ul{\I}_{(1, 1), (d+1, 1)}[\eta] = \eta_{(1, 1)}.  
\end{align}
Next consider the $\eta^0$-weights on $[d+1] \times [d]$ given by 
\begin{align}
\label{E:92}
\eta^0_{(i, j)} 
= \begin{cases}
0 \quad &\text{ if } i = j = 1, \\ 
\eta_{(i, j)} \quad &\text{ otherwise. }
\end{cases}
\end{align}
Using definition \eqref{E:IncInit} again, one can write    
\begin{align}
\label{E:88}
\begin{split}
&\ul{\I}_{(1, 1), (p+1, d+1-p)}[\eta] - \ul{\I}_{(1, 1), (p+2, d-p)}[\eta] \\ 
&= (\ul{\I}_{(1, 1), (p+1, d+1-p)}[\eta]-\eta_{(1, 1)}) - (\ul{\I}_{(1, 1), (p+2, d-p)}[\eta]-\eta_{(1, 1)}) \\ 
&= \ul{\I}_{(1, 1), (p+1, d+1-p)}[\eta^0] - \ul{\I}_{(1, 1), (p+2, d-p)}[\eta^0]
\end{split}
\end{align}
for each $p \in [d-1]$. Because the last expression does not use the weight $\eta_{(1, 1)}$, the quantity in \eqref{E:91} is independent of the collection of $d-1$ differences given by 
\begin{align}
\label{E:102}    
\{\ul{\I}_{(1, 1), (p+1, d+1-p)}[\eta] - \ul{\I}_{(1, 1), (p+2, d-p)}[\eta]: p \in [d-1]\}.
\end{align}
Hence, it remains to verify that the preceding collection is also independent. 

Introduce new weights $\eta^1$ on $[d+1] \times [d]$ 
by setting 
\begin{align}
\label{E:103}
\eta^1_{(i, j)} = \zeta(r_d) \cdot \eta_{(i, j)} \quad \text{ for } i \in [d+1] \text{ and } j \in [d]. 
\end{align}
Recall from definition \eqref{E:zeta} that $\zeta:(0, \infty) \to (0, 1)$ is an increasing bijection. Since also the sequence $(r_p)_{p \in [d]}$ is increasing, 
\begin{align}
\label{E:104}
\frac{\zeta(r_p)}{\zeta(r_{d})} = \zeta(r_p'), \quad \text{ for } p \in [d-1]
\end{align}
for some (unique) increasing sequence $(r_p')_{p \in [d-1]} \in \bbR_{>0}^{d-1}$. It follows from \eqref{E:106}, \eqref{E:103}, and \eqref{E:104} that the $\eta^1$-weights are independent with the marginals 
\begin{align}
\label{E:105}
\begin{split}
\eta^1_{(i, j)} &\sim \Exp\{1\} \qquad \qquad \hspace{2.25cm} \text{ if } i = 1 \text{ and } j = 1, \\ 
\eta^1_{(i, j)} &\sim \Exp\{\zeta(r_{d+1-j}')\} \qquad \hspace{1.6cm} \text{ if } i = 1 \text{ and } j > 1, \\ 
\eta^1_{(i, j)} &\sim \Exp\{1-\zeta(r_{i-1}')\} \qquad \hspace{1.4cm} \text{ if } i > 1 \text{ and } j = 1, \\ 
\eta^1_{(i, j)} &\sim \Exp\{\zeta(r_{d+1-j}')-\zeta(r_{i-1}')\} \qquad \text{ if } i > 1, j > 1 \text{ and } i + j < d+2, \\ 
\eta^1_{(i, j)} &= 0 \qquad \qquad \qquad  \hspace{2.5cm} \text{ otherwise. }
\end{split}
\end{align}
The $\eta^1$-weights restricted to $[d] \times [d]$ are distributionally identical to the $\eta$-weights defined with $k = 2$, $\ell = 2$ and $d-1$ directions $(r_p')_{p \in [d-1]}$. This is seen by comparing \eqref{E:105} with \eqref{E:auxw}. 

Now, with another appeal to Theorem \ref{T:BusMar}, one obtains that 
\begin{align}
\label{E:100}
\begin{split}
&\{\ul{\I}_{(1, 1), (p+1, d+1-p)}[\eta^1]: p \in [d-1]\} \deq \{\sI_{(1, 1)}^{r_p'}[\omega]: p \in [d-1]\}. 
\end{split}
\end{align} 
Then it follows from the induction hypothesis for $d-1$ directions that the collection 
\begin{align}
\label{E:101}
\begin{split}
&\{\ul{\I}_{(1, 1), (p+1, d+1-p)}[\eta^1]-\one_{\{p < d-1\}} \ul{\I}_{(1, 1), (p+2, d-p)}[\eta^1] : p \in [d-1]\} \\ 
&\deq \{\sI_{(1, 1)}^{r_p'}[\omega]-\one_{\{p < d-1\}}\sI_{(1, 1)}^{r_{p+1}'}[\omega]: p \in [d-1]\}
\end{split}
\end{align}
is independent. It is clear from definitions \eqref{E:IncInit} and \eqref{E:103} that 
\begin{align}
\label{E:107}
\ul{\I}_{(1, 1), (p+1, d+1-p)}[\eta] = \frac{1}{\zeta(r_d)} \cdot \ul{\I}_{(1, 1), (p+1, d+1-p)}[\eta^1] \quad \text{ for } p \in [d-1]. 
\end{align}
Consequently, the independence of the first collection in \eqref{E:101} implies the same for
\begin{align}
\label{E:108}
\{\ul{\I}_{(1, 1), (p+1, d+1-p)}[\eta]-\one_{\{p < d-1\}} \ul{\I}_{(1, 1), (p+2, d-p)}[\eta] : p \in [d-1]\}. 
\end{align}

One can now compute the joint characteristic function for the collection \eqref{E:102} as follows. For any $(s_p)_{p \in [d-1]} \in \bbR^{d-1}$, 
\begin{align}
\label{E:100-2}
\begin{split}
&\bfE\bigg[\prod_{p=1}^{d-1}\exp\{\ii s_p \cdot (\ul{\I}_{(1, 1), (p+1, d+1-p)}[\eta] - \ul{\I}_{(1, 1), (p+2, d-p)}[\eta])\}\bigg] \\
&= \frac{\bfE\bigg[\prod \limits_{p=1}^{d-1}\exp\{\ii s_p \cdot (\ul{\I}_{(1, 1), (p+1, d+1-p)}[\eta] - \ul{\I}_{(1, 1), (p+2, d-p)}[\eta])\} \cdot \exp\{\ii s_{d-1} \eta_{(1, 1)}\}\bigg]}{\bfE[\exp\{\ii s_{d-1} \eta_{(1, 1)}\}]} \\ 
&= \frac{\bfE\bigg[\prod \limits_{p=1}^{d-1}\exp\{\ii s_p \cdot (\ul{\I}_{(1, 1), (p+1, d+1-p)}[\eta] - \one_{\{p < d-1\}}\ul{\I}_{(1, 1), (p+2, d-p)}[\eta])\}\bigg]}{\bfE[\exp\{\ii s_{d-1} \eta_{(1, 1)}\}]} \\ 
&= \frac{\prod \limits_{p=1}^{d-1}\bfE[\exp\{\ii s_p \cdot (\ul{\I}_{(1, 1), (p+1, d+1-p)}[\eta] - \one_{\{p < d-1\}}\ul{\I}_{(1, 1), (p+2, d-p)}[\eta])\}]}{\bfE[\exp\{\ii s_{d-1} \eta_{(1, 1)}\}]} \\ 
&= \prod \limits_{p=1}^{d-1} \bfE[\exp\{\ii s_p \cdot (\ul{\I}_{(1, 1), (p+1, d+1-p)}[\eta] - \ul{\I}_{(1, 1), (p+2, d-p)}[\eta])\}]. 
\end{split}
\end{align}
The first equality in \eqref{E:100-2} 
follows because the collection given by \eqref{E:102} is independent of $\eta_{(1, 1)}$. The second equality is due to \eqref{E:91}. The subsequent equality comes from the independence of the collection in \eqref{E:108}. The last step in \eqref{E:100-2} uses \eqref{E:91} again and the independence of $\eta_{(1, 1)}$ and the $(d-1)$th term 
\begin{align*}
\ul{\I}_{(1, 1), (d, 2)}[\eta] - \ul{\I}_{(1, 1), (d+1, 1)}[\eta] = \ul{\I}_{(1, 1), (d, 2)}[\eta] - \eta_{(1, 1)}
\end{align*}
in \eqref{E:108}. 

On account of \eqref{E:100-2}, the collection in \eqref{E:102} is independent. This completes the inductive step and thereby establishes the result for $d$ directions. 
\end{proof}

\subsection{Shen's independence theorem} \label{sec:Shen_indep}

The independence of Busemann functions along a down-right path is stated in Proposition~\ref{P:BusFn}\eqref{P:BusFn_c} for a fixed direction $r$.
It was discovered in \cite{shen25} that the direction can be varied in a monotone fashion as one proceeds along the down-right path, and independence still holds.
Here is a special case that is easily obtainable from Theorem~\ref{T:BusMar}.

\begin{cor} \label{C:shen}
For any $s_1\ge \cdots \ge s_n$ in $\R_{>0}$, the random variables
\[
\sI_{(1,1)}^{s_1}[\omega],\quad
\sI_{(2,1)}^{s_2}[\omega],\quad\ldots\quad
\sI_{(n,1)}^{s_n}[\omega]
\]
are independent.
\end{cor}

\begin{proof}
We induct on the number of distinct values among the list $s_1\ge\cdots\ge s_n$.
The base case $s_1=\cdots=s_d$ is covered by Proposition~\ref{P:BusFn}\eqref{P:BusFn_c}.
Suppose there are $d$ distinct values, and call them $r_d > r_{d-1} > \cdots > r_1$. 
Define integers $m_1 < m_2 < \cdots < m_{d-1} < m_d = n$ by
\begin{align*}
s_1 = \cdots = s_{m_1} &= r_d, \\
s_{m_1+1} = \cdots = s_{m_2} &= r_{d-1}, \\
&\hspace{1.3ex}\vdots \\
s_{m_{d-1}+1} = \cdots = s_{n} &= r_1.
\end{align*}
Now apply Theorem~\ref{T:BusMar} with $k=n+1$ and $\ell=1$, to see that the random vector
\begin{align*}
&\Big(\big(\sI_{(1,1)}^{r_d}[\omega],\ldots,\sI_{(m_1,1)}^{r_d}[\omega]\big),
\big(\sI_{(m_1+1,1)}^{r_{d-1}}[\omega],\ldots,\sI_{(m_2,1)}^{r_{d-1}}[\omega],\ldots,\sI_{(m_{d-1}+1,1)}^{r_1}[\omega],\ldots,\sI_{(n,1)}^{r_1}[\omega]\big)\Big)
\end{align*}
has the same distribution as $(\bfX,\bfY)$, where
\begin{align*}
\bfX &= \big(\ul{\I}_{(1,1),\mathbf z_d}[\eta],\ldots,\ul{\I}_{(m_1,1),\mathbf z_d}[\eta]\big), \\
\bfY &= \big(\ul{\I}_{(m_1+1,1),\mathbf z_{d-1}}[\eta],\ldots,\ul{\I}_{(m_2,1),\mathbf z_{d-1}}[\eta],\ldots,
\ul{\I}_{(m_{d-1}+1,1),\mathbf z_1}[\eta],\ldots,\ul{\I}_{(n,1),\mathbf z_1}[\eta]\big).
\end{align*}
Since $\bfz_d = (n+d,1)$ is on the same row as the vertices of interest, the increments in $\bfX$ are trivial to calculate: $\ul{\I}_{(i,1),\mathbf z_d}[\eta] = \eta_{(i,1)}$ for each $i$.
In particular, $\bfX$ is measurable with respect to $\{\eta_{(i,1)}:i\in[m_1]\}$.
Meanwhile, $\bfY$ is measurable with respect to $\{\eta_{(i,j)}:i\in[n+d]\setminus[m_1],\, j\in[d]\}$.
Since the $\eta$ variables are independent, it follows that $\bfX$ and $\bfY$ are independent, and that the individual coordinates of $\bfX$ are independent.
By our distributional equality, we conclude that the following $m_1+1$ quantities are independent:
\begin{align} \label{E:1peel}
\sI_{(1,1)}^{s_1}[\omega],\quad 
\sI_{(2,1)}^{s_2}[\omega],\quad\ldots\quad
\sI_{(m_1,1)}^{s_{m_1}}[\omega],\quad
\big(\sI_{(m_1+1,1)}^{s_{m_1+1}}[\omega],\cdots,
\sI_{(n,1)}^{s_n}[\omega]\big).
\end{align}
There are $d-1$ distinct values among $s_{m_1+1},\dots,s_n$, so induction (together with shift-invariance) implies that the coordinates of the last quantity in \eqref{E:1peel} are independent.
\end{proof}

\subsection{Proof sketch} 
\label{Ss:BMdiscuss}
For clarity, we now present 
our argument for Theorem \ref{T:BusMar}  
in the case of $d = 2$ directions. The complete proof is presented in Section \ref{S:PfBuse}. It follows the general structure presented here, with several inductive steps to generalize to an arbitrary number of directions. 
For notational simplicity, we focus here on the values of the Busemann functions along horizontal edges; those along vertical edges are treated analogously. 
Thus, our aim is to sketch a derivation of the distributional identity 
\begin{align}
\label{E:114}    
(\sI_{\bfu}^{r_1}[\omega], \sI_{\bfu}^{r_2}[\omega]: \bfu \in \rightset_{k, \ell}) \deq (\ul{\I}_{\bfu, (k, \ell+1)}[\eta], \ul{\I}_{\bfu, (k+1, \ell)}[\eta]: \bfu \in \rightset_{k, \ell})
\end{align}
for any two directions $0 < r_1 < r_2 < \infty$ and grid side lengths $k, \ell \in \bbZ_{>0}$ with $k > 1$ (so that the preceding collections are nonempty). Note that \eqref{E:114} is exactly the conclusion of Theorem \ref{T:BusMar} for $d = 2$ if one restricts to the set $\rightset_{k, \ell}$ indexing the horizontal edges within the grid $[k] \times [\ell]$. 

In the proof, it suffices to show that 
\begin{align}
\label{E:115}    
\begin{split}
&\bfP\{\sI_{\bfu}^{r_1}[\omega] > x_{\bfu, 1} \text{ and } \sI_{\bfu}^{r_2}[\omega] > x_{\bfu, 2} \text{ for } \bfu \in \rightset_{k, \ell} \} \\ 
&= \bfP\{\ul\I_{\bfu, (k, \ell+1)}[\eta] > x_{\bfu, 1} \text{ and } \ul\I_{\bfu, (k+1, \ell)}[\eta] > x_{\bfu, 2} \text{ for } \bfu \in \rightset_{k, \ell}\}
\end{split}
\end{align}
for any $x_{\bfu, 1}, x_{\bfu, 2} \in \bbR$ for $\bfu \in \rightset_{k, \ell}$. Our approach is to bound the prelimit version of the left-hand side in \eqref{E:115} from above and below with probabilities that will eventually converge to the right-hand side.
Here we sketch only the lower bound, which involves manipulations of rate parameters assigned to the columns of $\bbZ_{>0}^2$.
The upper bound is obtained by analogous manipulations on rows (see, for example, Lemma \ref{L:CDFBd}).

\subsubsection*{Prelimit lower bound} Pick positive integers $k_n^2 > k_n^1 \ge k$ for each $n \in \bbZ_{\ge \ell}$ such that $k_n^p/n \to r_p$ for $p \in \{1, 2\}$. It follows from Proposition \ref{P:BusFn}\eqref{P:BusFn_a} (and the continuity of the exponential distributions in part \eqref{P:BusFn_b}) that 
\begin{align}
\label{E:116}    
\begin{split}
&\bfP\{\sI_{\bfu}^{r_1}[\omega] > x_{\bfu, 1} \text{ and } \sI_{\bfu}^{r_2}[\omega] > x_{\bfu, 2} \text{ for } \bfu \in \rightset_{k, \ell}\} \\ 
&= \lim_{n \to \infty} \bfP\{\ul \I_{\bfu, (k_n^1, n)}[\omega] > x_{\bfu, 1} \text{ and } \ul \I_{\bfu, (k_n^2, n)}[\omega] > x_{\bfu, 2} \text{ for } \bfu \in \rightset_{k, \ell}\}. 
\end{split}
\end{align}
Recalling the function $\zeta$ from \eqref{E:zeta}, for any pair of distinct $c_1, c_2 \in \bbZ_{>0}$, let $\omega^{\col, c_1, c_2}$ denote independent exponential weights on $\bbZ_{>0}^2$ with rate $1-\zeta(r_1)$ on column $c_1$, rate $1-\zeta(r_2)$ on column $c_2$, and rate $1$ elsewhere. The prelimit probability in \eqref{E:116} can be bounded from below as follows. See Figure \ref{F:CDFBd} for an illustration. 
\begin{align}
\label{E:117}
\begin{split}
&\bfP\{\ul \I_{\bfu, (k_n^1, n)}[\omega] > x_{\bfu, 1} \text{ and } \ul \I_{\bfu, (k_n^2, n)}[\omega] > x_{\bfu, 2} \text{ for } \bfu \in \rightset_{k, \ell}\} \\ 
&= \bfP\{\ul \I_{\bfu, (k_n^1, n)}[\omega^{\col, k_n^2+1, k_n^2+2}] > x_{\bfu, 1} \text{ and } \ul \I_{\bfu, (k_n^2, n)}[\omega^{\col, k_n^2+1, k_n^2+2}] > x_{\bfu, 2} \text{ for } \bfu \in \rightset_{k, \ell}\} \\ 
&\ge \bfP\{\ul \I_{\bfu, (k_n^1, n)}[\omega^{\col, k_n^2+1, k_n^2+2}] > x_{\bfu, 1} \text{ and } \ul \I_{\bfu, (k_n^2+2, n)}[\omega^{\col, k_n^2+1, k_n^2+2}] > x_{\bfu, 2} \text{ for } \bfu \in \rightset_{k, \ell}\} \\ 
&= \bfP\{\ul \I_{\bfu, (k_n^1, n)}[\omega^{\col, k_n^1+1, k_n^2+2}] > x_{\bfu, 1} \text{ and } \ul \I_{\bfu, (k_n^2+2, n)}[\omega^{\col, k_n^1+1, k_n^2+2}] > x_{\bfu, 2} \text{ for } \bfu \in \rightset_{k, \ell}\} \\ 
&\ge \bfP\{\ul \I_{\bfu, (k_n^1+1, n)}[\omega^{\col, k_n^1+1, k_n^2+2}] > x_{\bfu, 1} \text{ and } \ul \I_{\bfu, (k_n^2+2, n)}[\omega^{\col, k_n^1+1, k_n^2+2}] > x_{\bfu, 2} \text{ for } \bfu \in \rightset_{k, \ell}\}. 
\end{split}
\end{align}
The first equality in \eqref{E:117} holds because the environment is modified only along columns $k_n^2+1$ and $k_n^2 + 2$, which are outside the grid $[(1, 1), (k_n^2, n)]$. The subsequent step shifts the terminal point of the second increment from $(k_n^2, n)$ to $(k_n^2+2, n)$ and invokes the standard increment-comparison lemma (Lemma \ref{L:Comp}). The third step interchanges the rates $1-\zeta(r_1)$ on column $k_n^2+1$ with the rates $1$ on column $k_n^1+1$, which preserves the joint distribution of the increments. This invariance is a special case of our second main result, Theorem \ref{T:LppInv} ahead. The last step in \eqref{E:117} shifts the terminal point of the first increment from $(k_n^1, n)$ to $(k_n^1+1, n)$ and applies the increment-comparison lemma again. The proof of Lemma \ref{L:CDFBd} will generalize the argument in \eqref{E:117} to an arbitrary number of directions. 


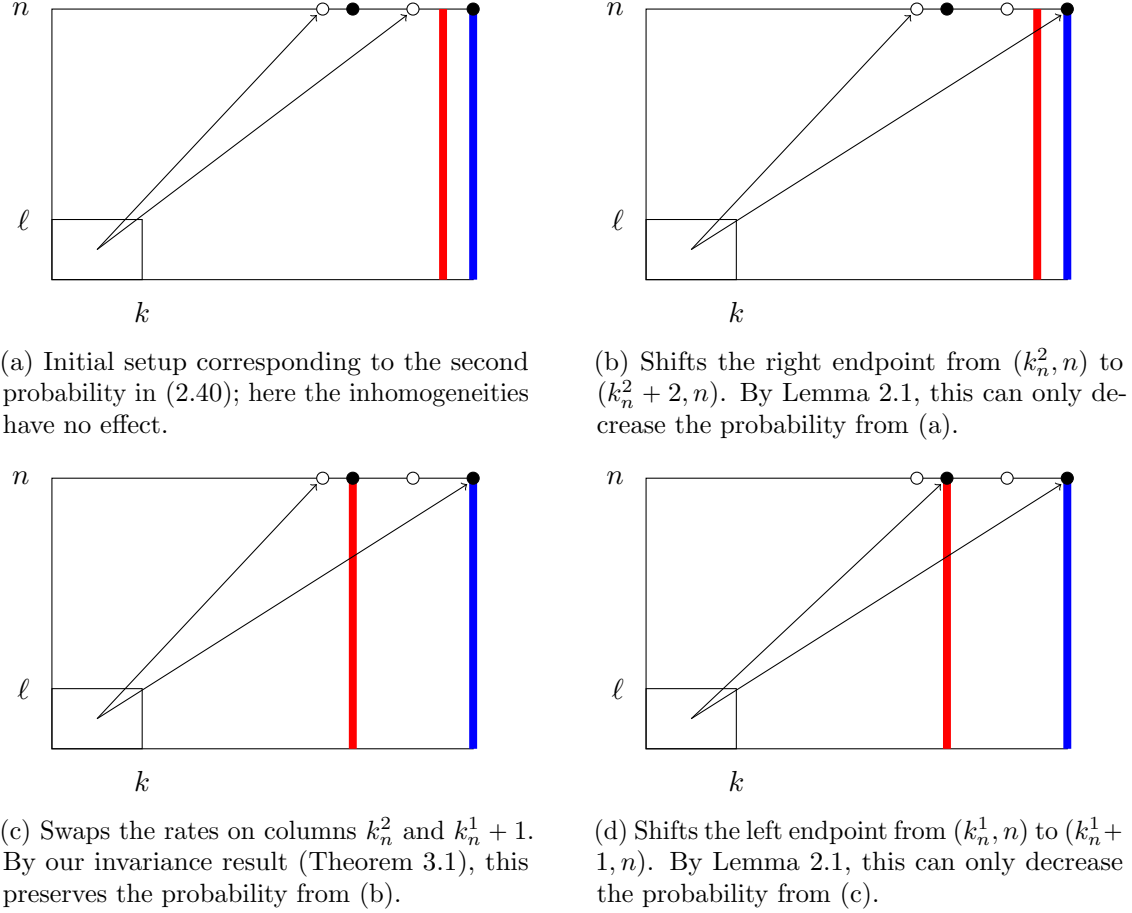
\begin{figure}[h!]
%
\vspace{0.2in}
\begin{subfigure}{0.45\textwidth}
\begin{tikzpicture}[scale = 0.8]
\draw[](0, 0)rectangle(7, 4.5);
\draw[](0, 0)rectangle(1.5, 1);
\draw[line width = 3, red](6.5, 0)--(6.5, 4.5);
\draw[line width = 3, blue](7, 0)--(7, 4.5);
\draw[fill=white](4.5, 4.5)circle(0.10);
\draw[->](0.75, 0.5)--(4.5-0.1, 4.5-0.1);
\draw[->](0.75, 0.5)--(6-0.1, 4.5-0.1);
\draw[fill=white](6, 4.5)circle(0.10);
\draw[fill=black](5, 4.5)circle(0.10);
\draw[fill=black](7, 4.5)circle(0.10);
\draw(1.5, 0-0.2)node[below]{$k$};
\draw(0-0.2, 1)node[left]{$\ell$};
\draw(0-0.2, 4.5)node[left]{$n$};
\end{tikzpicture}
\subcaption{Initial setup corresponding to the second probability in \eqref{E:117}; here the inhomogeneities have no effect.
}
\end{subfigure} \qquad
\begin{subfigure}{0.45\textwidth}
\begin{tikzpicture}[scale = 0.8]
\draw[](0, 0)rectangle(7, 4.5);
\draw[](0, 0)rectangle(1.5, 1);
\draw[line width = 3, red](6.5, 0)--(6.5, 4.5);
\draw[line width = 3, blue](7, 0)--(7, 4.5);
\draw[fill=white](4.5, 4.5)circle(0.10);
\draw[->](0.75, 0.5)--(4.5-0.1, 4.5-0.1);
\draw[->](0.75, 0.5)--(7-0.1, 4.5-0.1);
\draw[fill=white](6, 4.5)circle(0.10);
\draw[fill=black](5, 4.5)circle(0.10);
\draw[fill=black](7, 4.5)circle(0.10);
\draw(1.5, 0-0.2)node[below]{$k$};
\draw(0-0.2, 1)node[left]{$\ell$};
\draw(0-0.2, 4.5)node[left]{$n$};
\end{tikzpicture}
\subcaption{Shift the right endpoint from $(k_n^2, n)$ to $(k_n^2+2, n)$.
By Lemma \ref{L:Comp}, this can only decrease the probability from (a).
}
\end{subfigure} \\[0.5\baselineskip]
\begin{subfigure}{0.45\textwidth}
\begin{tikzpicture}[scale = 0.8]
\draw[](0, 0)rectangle(7, 4.5);
\draw[](0, 0)rectangle(1.5, 1);
\draw[line width = 3, red](5, 0)--(5, 4.5);
\draw[line width = 3, blue](7, 0)--(7, 4.5);
\draw[fill=white](4.5, 4.5)circle(0.10);
\draw[->](0.75, 0.5)--(4.5-0.1, 4.5-0.1);
\draw[->](0.75, 0.5)--(7-0.1, 4.5-0.1);
\draw[fill=white](6, 4.5)circle(0.10);
\draw[fill=black](5, 4.5)circle(0.10);
\draw[fill=black](7, 4.5)circle(0.10);
\draw(1.5, 0-0.2)node[below]{$k$};
\draw(0-0.2, 1)node[left]{$\ell$};
\draw(0-0.2, 4.5)node[left]{$n$};
\end{tikzpicture}
\subcaption{Swap the rates on columns $k_n^2$ and $k_n^1+1$.
By our invariance result (Theorem~\ref{T:LppInv}), this preserves the probability from (b).}
\end{subfigure} \qquad
%
\begin{subfigure}{0.45\textwidth}
\begin{tikzpicture}[scale = 0.8]
\draw[](0, 0)rectangle(7, 4.5);
\draw[](0, 0)rectangle(1.5, 1);
\draw[line width = 3, red](5, 0)--(5, 4.5);
\draw[line width = 3, blue](7, 0)--(7, 4.5);
\draw[fill=white](4.5, 4.5)circle(0.10);
\draw[->](0.75, 0.5)--(5-0.1, 4.5-0.1);
\draw[->](0.75, 0.5)--(7-0.1, 4.5-0.1);
\draw[fill=white](6, 4.5)circle(0.10);
\draw[fill=black](5, 4.5)circle(0.10);
\draw[fill=black](7, 4.5)circle(0.10);
\draw(1.5, 0-0.2)node[below]{$k$};
\draw(0-0.2, 1)node[left]{$\ell$};
\draw(0-0.2, 4.5)node[left]{$n$};
\end{tikzpicture}
\subcaption{Shift the left endpoint from $(k_n^1, n)$ to \linebreak $(k_n^1+1, n)$. 
By Lemma \ref{L:Comp}, this can only decrease the probability from (c). 
}
\end{subfigure}
\caption{\small Illustrates the argument in \eqref{E:117}.
The goal is to introduce inhomogeneities into the environment.
The white vertices are located at $(k_n^1,n)$ and $(k_n^2,n)$, and have limiting directions $r_1 < r_2$. 
The black vertices are located at $(k_n^1+1,n)$ and $(k_n^2+2,n)$. 
Arrows represent increments of last-passage times from the vertices in $[k] \times [\ell]$ to the indicated terminal vertices. The weights are independent exponential with rate $1$ everywhere except the two columns indicated by thick vertical stripes, where instead the rates are $1-\zeta(r_1)$ for the red/left column, and $1-\zeta(r_2)$ for the blue/right column.
}
\label{F:CDFBd}
\end{figure}

\subsubsection*{Limiting lower bound} Our next step is to pass to the limit in \eqref{E:117} as $n \to \infty$ to produce a limiting lower bound. To achieve this, we will take advantage of some recent results from \cite{Emra_Janj_Sepp_25} concerning the Busemann functions of inhomogeneous exponential LPP, which are reviewed in Section \ref{S:iBusFn}. 
See Figure~\ref{F:CDFLim} for an illustration.

Through another appeal to our permutation invariance result, the last probability in \eqref{E:117} is
\begin{align}
\label{E:118}
\begin{split}
&\bfP\{\ul \I_{\bfu, (k_n^1+1, n)}[\omega^{\col, k_n^1+1, k_n^2+2}] > x_{\bfu, 1} \text{ and } \ul \I_{\bfu, (k_n^2+2, n)}[\omega^{\col, k_n^1+1, k_n^2+2}] > x_{\bfu, 2} \text{ for } \bfu \in \rightset_{k, \ell}\} \\ 
&= \bfP\{\ul \I_{\bfu, (k_n^1+1, n)}[\omega^{\col, k, k_n^2+2}] > x_{\bfu, 1} \text{ and } \ul \I_{\bfu, (k_n^2+2, n)}[\omega^{\col, k, k_n^2+2}] > x_{\bfu, 2} \text{ for } \bfu \in \rightset_{k, \ell}\},
\end{split}
\end{align}
where we interchanged the rates $1-\zeta(r_1)$ on column $k_n^1+1$ with the rates $1$ on column $k$. Now, for any column $c \in \bbZ_{\ge k}$ and arbitrary weights $\w$ on $\bbZ_{>0}^2$, one can define the \emph{thin} Busemann function
\begin{align}
\label{E:119}
\sI^{c, \uparrow}_{\bfu}[\w] = \sup_{n \in \bbZ_{\ge \ell}} \ul{\I}_{\bfu, (c, n)}[\w] = \lim_{n \to \infty} \ul{\I}_{\bfu, (c, n)}[\w] \quad \text{ for } \bfu \in \rightset_{k, \ell}. 
\end{align}
The second equality in \eqref{E:119} is a consequence of the standard increment-comparison in Lemma \ref{L:Comp}. Since $(k_n^1+1)/n \to r_1$ as $n \to \infty$ and due to the marginal structure of the $\omega^{\col, k, k_n^2+2}$-weights on the grid $[(1, 1), (k_n^1+1, n)]$, the direction $r_1$ falls in (in fact, is the endpoint of) a closed interval of directions for which the Busemann functions coincide with certain thin Busemann functions. More specifically, it follows from Corollary \ref{C:ThinBuse} and Proposition \ref{P:iBusFlat} that 
\begin{align}
\label{E:120}
\lim_{n \to \infty} \{\ul \I_{\bfu, (k_n^1+1, n)}[\omega^{\col, k, k_n^2+2}] - \sI^{k, \uparrow}_{\bfu}[\omega^{\col, k, k_n^2+2}]\} = 0 \quad \text{a.s.\ for } \bfu \in \rightset_{k, \ell}. 
\end{align}
The limits in \eqref{E:120} and union bounds imply that the last probability in \eqref{E:118} is greater than or equal to 
\begin{align}
\label{E:121}
\begin{split}
&\bfP\{\sI_{\bfu}^{k, \uparrow}[\omega^{\col, k, k_n^2+2}] > x_{\bfu, 1} + \epsilon \text{ and } \ul \I_{\bfu, (k_n^2+2, n)}[\omega^{\col, k, k_n^2+2}] > x_{\bfu, 2} \text{ for } \bfu \in \rightset_{k, \ell}\} - o_\epsilon(1)
\end{split}
\end{align} 
for any $\epsilon > 0$, where $o_\epsilon(1)$ is an $\epsilon$-dependent term that vanishes as $n \to \infty$. Importantly for the subsequent step, the thin Busemann function in \eqref{E:121} depends only on the first $k$ columns. 

\begin{figure}[h!]
%
\begin{subfigure}[t]{0.45\textwidth}
\begin{tikzpicture}[scale = 0.8]
\draw[](0, 0)rectangle(7, 4.5);
\draw[](0, 0)rectangle(1.5, 1);
\draw[line width = 3, red](1.5, 0)--(1.5, 4.5);
\draw[line width = 3, blue](7, 0)--(7, 4.5);
\draw[->](0.75, 0.5)--(5-0.1, 4.5-0.1);
\draw[->](0.75, 0.5)--(7-0.1, 4.5-0.1);
\draw[fill=black](5, 4.5)circle(0.10);
\draw[fill=black](7, 4.5)circle(0.10);
\draw(1.5, 0-0.2)node[below]{$k$};
\draw(0-0.2, 1)node[left]{$\ell$};
\draw(0-0.2, 4.5)node[left]{$n$};
\end{tikzpicture}
\subcaption{From Figure \ref{F:CDFBd}(d), swap the rates on columns $k_n^1+1$ and $k$.
By our invariance result (Theorem~\ref{T:LppInv}), this preserves the probability under consideration, resulting in \eqref{E:118}.}
\end{subfigure} \qquad
\begin{subfigure}[t]{0.45\textwidth}
\begin{tikzpicture}[scale = 0.8]
\draw[](0, 6)--(0, 0)--(7, 0);
\draw[](1.5, 4.5)--(7, 4.5);
\draw[](0, 0)rectangle(1.5, 1);
\draw[](0, 4.5)--(0, 6);
\draw[line width = 3, red](1.5, 4.5)--(1.5, 6);
\draw[line width = 3, red](1.5, 0)--(1.5, 4.5);
\draw[line width = 3, blue](7, 0)--(7, 4.5);
\draw[->](0.75, 0.5)--(1.5-0.1, 6-0.1);
\draw[->](0.75, 0.5)--(7-0.1, 4.5-0.1);
\draw[fill=black](7, 4.5)circle(0.10);
\draw(1.5, 0-0.2)node[below]{$k$};
\draw(0-0.2, 1)node[left]{$\ell$};
\draw(0-0.2, 4.5)node[left]{$n$};
\draw(0-0.12, 5.8)node[left]{$\infty$};
\end{tikzpicture}
\subcaption{Replace all the increments from $[k] \times [\ell]$ to $(k_n^1+1, n)$ with their a.s.\ limits at the cost of a vanishing error; see \eqref{E:120} and \eqref{E:121}. 
By Section \ref{S:iBusFn}, the limits are thin Busemann functions along column $k$.}
\end{subfigure} \\[0.75\baselineskip]
%
\begin{subfigure}{0.45\textwidth}
\begin{tikzpicture}[scale = 0.8]
\draw[](0, 6)--(0, 0)--(7, 0)--(7, 4.5);
\draw[](1.5, 4.5)--(7, 4.5);
\draw[](0, 0)rectangle(1.5, 1);
\draw[](0, 4.5)--(0, 6);
\draw[line width = 3, red](1.5, 4.5)--(1.5, 6);
\draw[line width = 3, red](1.5, 0)--(1.5, 4.5);
\draw[line width = 3, blue](2, 0)--(2, 4.5);
\draw[->](0.75, 0.5)--(1.5-0.1, 6-0.1);
\draw[->](0.75, 0.5)--(7-0.1, 4.5-0.1);
\draw[fill=black](7, 4.5)circle(0.10);
\draw(1.5, 0-0.2)node[below]{$k$};
\draw(0-0.2, 1)node[left]{$\ell$};
\draw(0-0.2, 4.5)node[left]{$n$};
\draw(0-0.12, 5.8)node[left]{$\infty$};
\end{tikzpicture}
\subcaption{Swap the rates on columns $k_n^2+2$ and $k+1$, resulting in weights $\omega^{1}=\omega^{\col, k, k+1}$. 
Once again by invariance, this preserves the probability from (b), resulting in \eqref{E:122}. 
}
\end{subfigure} \qquad
\begin{subfigure}{0.45\textwidth}
\begin{tikzpicture}[scale = 0.8]
\draw[](0, 6)--(0, 0)--(7, 0)--(7, 4.5);
\draw[](2, 4.5)--(7, 4.5);
\draw[](0, 0)rectangle(1.5, 1);
\draw[](0, 4.5)--(0, 6);
\draw[line width = 3, red](1.5, 4.5)--(1.5, 6);
\draw[line width = 3, red](1.5, 0)--(1.5, 4.5);
\draw[line width = 3, blue](2, 0)--(2, 6);
\draw[->](0.75, 0.5)--(1.5-0.1, 6-0.1);
\draw[->](0.75, 0.5)--(2-0.1, 6-0.1);
\draw(1.5, 0-0.2)node[below]{$k$};
\draw(0-0.2, 1)node[left]{$\ell$};
\draw(0-0.12, 5.8)node[left]{$\infty$};
\end{tikzpicture}
\subcaption{Replace all the increments from $[k] \times [\ell]$ to $(k_n^2+2, n)$ with their a.s.\ limits (thin Busemann functions along column $k+1$) at the cost of a vanishing error; see \eqref{E:123}.}
\end{subfigure}
\caption{\small Illustrates the steps from \eqref{E:118} to \eqref{E:125}. The goal is to replace the increments from Figure~\ref{F:CDFBd}(d) with thin Busemann functions (defined in \eqref{E:119}).
}
\label{F:CDFLim}
\end{figure}

We can now interchange the rate $1-\zeta(r_2)$ presently on column $k_n^2+2$ with the rate $1$ on column $k+1$. As a consequence of our invariance result again, the probability in \eqref{E:121} equals 
\begin{align}
\label{E:122}    
&\bfP\{\sI_{\bfu}^{k, \uparrow}[\omega^{\col, k, k+1}] > x_{\bfu, 1} + \epsilon \text{ and } \ul \I_{\bfu, (k_n^2+2, n)}[\omega^{\col, k, k+1}] > x_{\bfu, 2} \text{ for } \bfu \in \rightset_{k, \ell}\} 
\end{align}
computed with the new environment $\omega^{\col, k, k+1}$. Similarly to \eqref{E:121}, it follows from the marginal structure of the $\omega^{\col, k, k+1}$-weights, the convergence $(k_n^2+2)/n \to r_2$ as $n \to \infty$, and the properties of the inhomogeneous Busemann functions that the probability in \eqref{E:122} is greater than or equal to
\begin{align}
\label{E:123}    
&\bfP\{\sI_{\bfu}^{k, \uparrow}[\omega^{\col, k, k+1}] > x_{\bfu, 1} + \epsilon \text{ and } \sI_{\bfu}^{k+1, \uparrow}[\omega^{\col, k, k+1}] > x_{\bfu, 2}+\epsilon \text{ for } \bfu \in \rightset_{k, \ell}\} - o_\epsilon(1). 
\end{align}

In summary, the steps from \eqref{E:117} to \eqref{E:123} have shown that 
\begin{align}
\label{E:124}
\begin{split}
&\bfP\{\ul \I_{\bfu, (k_n^1, n)}[\omega] > x_{\bfu, 1} \text{ and } \ul \I_{\bfu, (k_n^2, n)}[\omega] > x_{\bfu, 2} \text{ for } \bfu \in \rightset_{k, \ell}\}\\
&\ge \bfP\{\sI_{\bfu}^{k, \uparrow}[\omega^{1}] > x_{\bfu, 1} + \epsilon \text{ and } \sI_{\bfu}^{k+1, \uparrow}[\omega^{1}] > x_{\bfu, 2}+\epsilon \text{ for } \bfu \in \rightset_{k, \ell}\} - o_\epsilon(1),
\end{split}
\end{align}
where $\omega^1 = \omega^{\col, k, k+1}$. 
Letting $n \to \infty$ and then $\epsilon \to 0$ in \eqref{E:124} yields 
\begin{align}
\label{E:125}
\begin{split}
&\bfP\{\sI_{\bfu}^{r_1}[\omega] > x_{\bfu, 1} \text{ and } \sI_{\bfu}^{r_2}[\omega] > x_{\bfu, 2} \text{ for } \bfu \in \rightset_{k, \ell} \} \\ 
&\ge \bfP\{\sI_{\bfu}^{k, \uparrow}[\omega^{1}] > x_{\bfu, 1} \text{ and } \sI_{\bfu}^{k+1, \uparrow}[\omega^{1}] > x_{\bfu, 2}\text{ for } \bfu \in \rightset_{k, \ell}\}. 
\end{split}
\end{align}
The generalization of \eqref{E:125} to multiple directions appears in Lemma \ref{L:CDFLim}. 

\subsubsection*{Expressing the lower bound via finitely many independent variables} The last stage of our derivation is to express the lower bound in \eqref{E:125} in terms of the $\eta$-weights given by \eqref{E:def_eta}. This involves a sequence of technical steps that are illustrated in the accompanying Figures \ref{F:TBusDisId} and \ref{F:TBusDisId2} below. We will extend this argument to more than $2$ directions in the proof of Lemma \ref{L:BusDisId}. 

Key to our proof of this part is the notion of \textit{induced weights}. 
Specifically, let $\w = \{\w_{(i,j)}:i,j \ge 1\}$ be a collection of weights, and let $\bfy,\bfz$ be points in $\Z^2$ such that $\bfy \le \bfz$ coordinatewise. Then, for any $\bfx \le \bfy$, the passage time $\Lp_{\bfx,\bfz}[\w]$ can be computed as $\Lp_{\bfx,\bfy}[\wt \w^{\bfy,\bfz}] + \Lp_{\bfy,\bfz}[\w]$, where $\wt \w^{\bfy,\bfz}$ are the induced weights given in Equation \eqref{E:wInd}. 
These induced weights are equal to the original weights everywhere except directly west or directly south of $\bfy$, where they are instead defined as increments of passage times to $\bfz$.  The increment at $\bfx$ depends in a somewhat complicated way on the weights in the rectangular grid $[\bfx, \bfz]$, but
the advantage of this construction is that for \textit{any} points $\bfx_1,\bfx_2 \le \bfy$, we have
\[
\Lp_{\bfx_1,\bfz}[\w] - \Lp_{\bfx_2,\bfz}[\w] = \Lp_{\bfx_1,\bfy}[\wt \w^{\bfy,\bfz}] - \Lp_{\bfx_2,\bfy}[\wt \w^{\bfy,\bfz}].
\]
That is, the increment $\Lp_{\bfx_1,\bfz}[\w] - \Lp_{\bfx_2,\bfz}[\w]$ can be computed as an increment of last-passage times in a smaller (potentially much smaller) grid, albeit with different weights. In particular, we can send the point $\bfz \to \infty$ and represent thin Busemann functions via last-passage increments in a finite grid. From here, one can heuristically see why we can represent the joint distribution of Busemann functions via a finite LPP problem. What follows is a technical description of several steps, using facts about inhomogeneous exponential LPP and permutation invariance that describe how to get the appropriate vertex weight distributions.

\begin{figure}[h!]
\begin{subfigure}[t]{0.47\textwidth}
\centering
\begin{tikzpicture}[scale = 0.8]
\draw[](0, 0)--(0, 4.5);
\draw[](0, 0)--(4, 0);
\draw[](0, 0)rectangle(3, 2);
\draw[](0, 4.5)--(0, 6);
\draw[line width = 3, red](3, 4.5)--(3, 6);
\draw[line width = 3, red](3, 0)--(3, 4.5);
\draw[line width = 3, blue](4, 0)--(4, 6);
\draw[->](1.5, 1)--(3-0.1, 6-0.1);
\draw[->](1.5, 1)--(4-0.1, 6-0.1);
\draw(3, 0-0.2)node[below]{$k$};
\draw(4, 0-0.2)node[below]{$k+1$};
\draw(0-0.2, 2)node[left]{$\ell$};
\draw(0-0.12, 6)node[left]{$\infty$};
\end{tikzpicture}
\subcaption{Initial setup as in Figure \ref{F:CDFLim}(d) where the weights are $\omega^1$. 
Only the first $k+1$ columns are shown. The thin Busemann functions $\sI_{\bfu}^{k+p-1, \uparrow}[\omega^1]$ for $p \in \{1, 2\}$ and $\bfu \in \rightset_{k, \ell}$ are indicated.  The key identity to be used next is \eqref{E:126}.}
\end{subfigure}
\qquad 
\begin{subfigure}[t]{0.47\textwidth}
\centering
\begin{tikzpicture}[scale = 0.8]
\draw[](0, 0)--(0, 4.5);
\draw[](0, 0)--(4, 0);
\draw[](0, 0)rectangle(3, 2);
\draw[](0, 4.5)--(0, 6);
\draw[](0, 2)--(4, 2);
\draw[line width = 3, red](3, 0)--(3, 6);
\draw[line width = 3, blue](4, 0)--(4, 6);
\draw[->, dashed](1.5, 1)--(3-0.1, 2-0.1);
\draw[->, dashed](1.5, 1)--(4-0.1, 2-0.1);
\draw[->](1.5, 2)--(3-0.1, 6-0.1);
\draw[->](2, 2)--(4-0.1, 6-0.1);
\draw(3, 0-0.2)node[below]{$k$};
\draw(4, 0-0.2)node[below]{$k+1$};
\draw(0-0.2, 2)node[left]{$\ell$};
\draw(0-0.12, 6)node[left]{$\infty$};
\end{tikzpicture}
\subcaption{
Each thin Busemann function in environment $\omega^1$ is equal to an increment in an induced environment $\wt{\omega}^p$ given by \eqref{E:127-2}.
The resulting equality is \eqref{E:128}.
The induced weights depend on the increment endpoint ($(k,\ell)$ for $p=1$, or $(k+1,\ell)$ for $p=2$) but differ from $\omega^1$ only on row $\ell$.
}
\end{subfigure} \\[\baselineskip]
\begin{subfigure}[t]{0.47\textwidth}
\centering
\begin{tikzpicture}[scale = 0.8]
\draw[](0, 0)--(0, 4.5);
\draw[](0, 0)--(4, 0);
\draw[](0, 0)rectangle(3, 2);
\draw[](0, 4.5)--(0, 6);
\draw[](0, 2)--(4, 2);
\draw[line width = 3, red](3, 0)--(3, 6);
\draw[line width = 3, blue](4, 0)--(4, 6);
\draw[line width = 3, dashed, blue](0, 6)--(4, 6);
\draw[fill=gray](3, 6)circle(0.10);
\draw[fill=gray](4, 6)circle(0.10);
\draw[->, dashed](1.5, 1)--(3-0.1, 2-0.1);
\draw[->, dashed](1.5, 1)--(4-0.1, 2-0.1);
\draw[->](1.5, 2)--(3-0.1, 6-0.1);
\draw[->](2, 2)--(4-0.1, 6-0.1);
\draw(3, 0-0.2)node[below]{$k$};
\draw(4, 0-0.2)node[below]{$k+1$};
\draw(0-0.2, 2)node[left]{$\ell$};
\draw(0-0.2, 6)node[left]{$n$};
\end{tikzpicture}
\subcaption{Modify the rates on row $n$ to produce new weights  $\wh{\omega}^{n, \epsilon}$ as in \eqref{E:130}.
The modified rates are indicated with dashed blue. 
When $n$ is large, the modification has negligible impact on the prelimits of thin Busemann functions; see \eqref{E:133}.
Therefore, in (b), these prelimits can be used instead of thin Busemann functions, resulting in \eqref{E:135}. 
The corresponding induced weights are denoted by $\wh{\eta}^{p, n, \epsilon}$ as in \eqref{E:134}, for $p\in\{1,2\}$.
}
\end{subfigure}
\qquad 
\begin{subfigure}[t]{0.47\textwidth}
\centering
\begin{tikzpicture}[scale = 0.8]
\draw[](0, 0)--(0, 4.5);
\draw[](0, 0)--(4, 0);
\draw[](0, 6)--(4, 6);
\draw[](0, 4.5)--(0, 6);
\draw[line width = 3, red](3, 0)--(3, 6);
\draw[line width = 3, blue](4, 0)--(4, 6);
\draw[line width = 3, dashed, blue](0, 2)--(4, 2);
\draw[fill=gray](3, 6)circle(0.10);
\draw[fill=gray](4, 6)circle(0.10);
\draw[->, dashed](1.5, 1)--(3-0.1, 2-0.1);
\draw[->, dashed](1.5, 1)--(4-0.1, 2-0.1);
\draw[->](1.5, 2)--(3-0.1, 6-0.1);
\draw[->](2, 2)--(4-0.1, 6-0.1);
\draw(3, 0-0.2)node[below]{$k$};
\draw(4, 0-0.2)node[below]{$k+1$};
\draw(0-0.2, 2)node[left]{$\ell$};
\draw(0-0.2, 6)node[left]{$n$};
\end{tikzpicture}
\subcaption{Swap the rates on rows $\ell$ and $n$ to produce new weights $\wc{\omega}^\epsilon$.
The associated induced weights are denoted by $\wc{\eta}^{p, n, \epsilon}$ as in \eqref{E:137} and are equal in distribution to $\wh{\eta}^{p, n, \epsilon}$ by our invariance result (Theorem~\ref{T:LppInv}).
Therefore, the increments shown here are equal in distribution to those in (c); see \eqref{E:138}.
As $n\to\infty$,  $\wc{\eta}^{p, n, \epsilon}$ converges to the weights $\wc{\eta}^{p, \epsilon}$ induced by a thin Busemann function in environment $\wc{\omega}^\epsilon$, resulting in \eqref{E:140}.
}
\end{subfigure}
\caption{\small Illustrates \eqref{E:126}--\eqref{E:140}. 
The goal is to identify thin Busemann functions/their prelimits (solid arrows) with finite grid LPP increments (dashed arrows). 
}
\label{F:TBusDisId}
\end{figure}
As an application of this fact, for any column $c \in \bbZ_{\ge k}$ and row $r \in \bbZ_{\ge \ell}$, the following identity holds for arbitrary weights $\w$: (see \eqref{E:39}--\eqref{E:40} in the proof of Lemma \ref{L:BusDisId} for a complete proof) 
\begin{align}
\label{E:126}    
\sI_{\bfu}^{c, \uparrow}[\w] = \ul{\I}_{\bfu, (c, r)}[\wt{\w}^{c, r}] \quad \text{ for } \bfu \in \rightset_{k, \ell},
\end{align}
where $\wt{\w}^{c, r}$ denotes the induced weights on the grid $[(1, 1), (c, r)]$ that coincide with the $\w$-weights on rows $1, \dotsc, r-1$, and are given as follows on row $r$:
\begin{align}
\label{E:127}
\wt{\w}^{c, r}_{(i, r)} = \one_{\{i < c\}} \cdot \sI_{(i, r)}^{c, \uparrow}[\w] \quad \text{ for } i \in [c]
\end{align}
For $p \in \{1, 2\}$, we generate the induced weights $\wt{\omega}^p$ by applying the map $\wt{(\cdot)}^{k+p-1, \ell}$ to the weights $\omega^1 = \omega^{\col, k, k+1}$. Thus, $\wt{\omega}^p$ is the same as $\omega^1$ on $[k+p-1] \times [\ell-1]$, and 
\begin{align}
\label{E:127-2}    
\wt{\omega}^{p}_{(i, \ell)} = \one_{\{i < k+p-1\}} \cdot \sI_{(i, \ell)}^{k+p-1, \uparrow}[\omega^1] \quad \text{ for } i \in [k+p-1]. 
\end{align}
Through \eqref{E:126}, 
one can rewrite the probability on the right-hand side of \eqref{E:125} as 
\begin{align}
\label{E:128}
\begin{split}
&\bfP\{\sI_{\bfu}^{k, \uparrow}[\omega^{1}] > x_{\bfu, 1} \text{ and } \sI_{\bfu}^{k+1, \uparrow}[\omega^{1}] > x_{\bfu, 2}\text{ for } \bfu \in \rightset_{k, \ell}\} \\ 
&= \bfP\{\ul{\I}_{\bfu, (k, \ell)}[\wt{\omega}^1] > x_{\bfu, 1} \text{ and } \ul{\I}_{\bfu, (k+1, \ell)}[\wt{\omega}^2] > x_{\bfu, 2}\text{ for } \bfu \in \rightset_{k, \ell}\}. 
\end{split}
\end{align}

Next, pick $n \in \bbZ_{>\ell}$ and a small $\epsilon > 0$ such that $\zeta(r_1) + \epsilon < \zeta(r_2)$. Let $\wh{\omega}^{n, \epsilon}$ be independent weights on $[k+1] \times [n]$ that equal 
the $\omega^1$-weights on rows $1, \dotsc, n-1$, and have the following marginal distributions on row $n$: 
\begin{align}
\label{E:130}
\begin{split}
&\wh{\omega}^{n, \epsilon}_{(i, n)} \sim \Exp\{\zeta(r_2)+\epsilon\} \quad \text{ for } i \in [k-1], \\ 
&\wh{\omega}^{n, \epsilon}_{(k, n)} \sim \Exp\{\zeta(r_2)+\epsilon-\zeta(r_1)\} \quad \text{ and } \quad \wh{\omega}_{(k+1, n)}^{n, \epsilon} \sim \Exp\{\epsilon\}. 
\end{split}
\end{align}
Recall that the means of the $\omega^1$-weights on columns $k+1$ and $k$ are $\frac{1}{1-\zeta(r_2)}$ and $\frac{1}{1-\zeta(r_1)}$,  respectively, and the means of the $\omega^1$ weights on all columns less than $k$ are $1$. We show via a 
law of large numbers 
for inhomogeneous exponential LPP in  Lemma \ref{L:ThinShp} that this ordering of rates on the columns has the following implication: 
For $p \in \{1,2\}$ and $n \ge N_0$ for some sufficiently large random $N_0 \in \bbZ_{>0}$, 
geodesics to the point $(k+p-1,n)$ will travel through 
$(k+p-1,n-1)$, regardless of the weights on the top row $n$.  
Specifically, for $p \in \{1,2\}, i \in [k+p-1]$ and $n \ge N_0$, 
\begin{align*}
\Lp_{(i, \ell), (k+p-1, n)}[\wh{\omega}^{n, \epsilon}] = \Lp_{(i, \ell), (k+p-1, n-1)}[\omega^1] + \wh{\omega}^{n, \epsilon}_{(k+p-1, n)}.
\end{align*}
Taking differences above gives us  
\begin{align}
\label{E:132}
\ul{\I}_{(i, \ell), (k+p-1, n)}[\wh{\omega}^{n, \epsilon}] = \ul{\I}_{(i, \ell), (k+p-1, n-1)}[\omega^1],
\end{align}
for $p \in \{1, 2\}$, $i \in [k+p-2]$ and $n \ge N_0$. 
In the $n \to \infty$ limit of \eqref{E:132}, using \eqref{E:119}, one obtains that  
\begin{align}
\label{E:133}
\lim_{n \to \infty} \ul{\I}_{(i, \ell), (k+p-1, n)}[\wh{\omega}^{n, \epsilon}] \stackrel{\rm{a.s.}}{=} \sI^{k+p-1, \uparrow}_{(i, \ell)}[\omega^1].
\end{align}
Now, for each $p \in \{1, 2\}$, consider the weights $\wh{\eta}^{p, n, \epsilon}$ on $[k+p-1] \times [\ell]$ that agree with $\omega^1 = \wt{\omega}^p$ on rows $1, \dotsc, \ell-1$ and are given on row $\ell$ by 
\begin{align}
\label{E:134}
\begin{split}
\wh{\eta}^{p, n, \epsilon}_{(i, \ell)} &= \one_{\{i < k+p-1\}} \cdot \ul{\I}_{(i, \ell), (k+p-1, n)}[\wh{\omega}^{n, \epsilon}] \quad \text{ for } i \in [k+p-1].
\end{split}
\end{align}
By \eqref{E:127-2} and \eqref{E:133}, one has $\wh{\eta}^{p, n, \epsilon}_{(i, \ell)} \to \wt{\omega}^p_{(i, \ell)}$ a.s.\ as $n \to \infty$. Consequently and through the continuity of the last-passage times as a function of the weights, one can express the second probability in \eqref{E:128} as the following limit:   
\begin{align}
\label{E:135}    
\begin{split}
&\bfP\{\ul{\I}_{\bfu, (k, \ell)}[\wt{\omega}^1] > x_{\bfu, 1} \text{ and } \ul{\I}_{\bfu, (k+1, \ell)}[\wt{\omega}^2] > x_{\bfu, 2}\text{ for } \bfu \in \rightset_{k, \ell}\} \\
&= \lim_{n \to \infty}\bfP\{\ul{\I}_{\bfu, (k, \ell)}[\wh{\eta}^{1, n, \epsilon}] > x_{\bfu, 1} \text{ and } \ul{\I}_{\bfu, (k+1, \ell)}[\wh{\eta}^{2, n, \epsilon}] > x_{\bfu, 2}\text{ for } \bfu \in \rightset_{k, \ell}\}. 
\end{split}
\end{align}
Next, we let $\wc{\omega}^{\epsilon}$ denote the independent weights on $[k+1] \times \bbZ_{>0}$ that coincide with $\omega^1$ except on row $\ell$, where the marginals are the same as those in \eqref{E:130}. In other words, the weights $\wc{\omega}^{\epsilon}$ restricted to $[k+1] \times [n]$ are obtained from $\wh{\omega}^{n, \epsilon}$ by interchanging the exponential rates on rows $\ell$ and $n$. 

For each $p \in \{1, 2\}$, analogously to $\wh{\eta}^{p, n, \epsilon}$, consider the weights $\wc{\eta}^{p, n, \epsilon}$ on $[k+p-1] \times [\ell]$ that equal $\omega^1 = \wt{\omega}^p$ on rows $1, \dotsc, \ell-1$ and are given on row $\ell$ by 
\begin{align}
\label{E:137}
\begin{split}
\wc{\eta}^{p, n, \epsilon}_{(i, \ell)} &= \one_{\{i < k+p-1\}} \cdot \ul{\I}_{(i, \ell), (k+p-1, n)}[\wc{\omega}^{\epsilon}] \quad \text{ for } i \in [k+p-1].
\end{split}
\end{align}
Since we obtained the weights $\wc \omega^{\epsilon}$ from the weights $\wh \omega^{\epsilon}$ by exchanging the exponential rates on rows $\ell$ and $n$, it is a consequence of our invariance result (Theorem \ref{T:LppInv}) that the weights in \eqref{E:137} have the same joint distribution as the weights in \eqref{E:134} (as collections indexed by $i$ and $p$). Since both of these collections are also independent of the $\omega^1$-weights on rows $1, \dotsc, \ell-1$, the prelimit probability in \eqref{E:135} is
\begin{align}
\label{E:138}
\begin{split}
&\bfP\{\ul{\I}_{\bfu, (k, \ell)}[\wh{\eta}^{1, n, \epsilon}] > x_{\bfu, 1} \text{ and } \ul{\I}_{\bfu, (k+1, \ell)}[\wh{\eta}^{2, n, \epsilon}] > x_{\bfu, 2}\text{ for } \bfu \in \rightset_{k, \ell}\} \\ 
&= \bfP\{\ul{\I}_{\bfu, (k, \ell)}[\wc{\eta}^{1, n, \epsilon}] > x_{\bfu, 1} \text{ and } \ul{\I}_{\bfu, (k+1, \ell)}[\wc{\eta}^{2, n, \epsilon}] > x_{\bfu, 2}\text{ for } \bfu \in \rightset_{k, \ell}\}. 
\end{split}
\end{align}
As $n \to \infty$, the $\wc{\eta}^{n, p, \epsilon}$-weights converge a.s.\ to the $\wc{\eta}^{p, \epsilon}$-weights on $[k+p-1] \times [\ell]$ that equal $\omega^1$ on rows $1, \dotsc, \ell-1$, and are given as follows on row $\ell$: 
\begin{align}
\label{E:139}
\begin{split}
\wc{\eta}^{p, \epsilon}_{(i, \ell)} &= \one_{\{i < k+p-1\}} \cdot \sI_{(i, \ell)}^{k+p-1, \uparrow}[\wc{\omega}^{\epsilon}] \quad \text{ for } i \in [k+p-1].
\end{split}
\end{align}
Hence, letting $n \to \infty$ in \eqref{E:138} and using \eqref{E:135} yields 
\begin{align}
\label{E:140}    
\begin{split}
&\bfP\{\ul{\I}_{\bfu, (k, \ell)}[\wt{\omega}^1] > x_{\bfu, 1} \text{ and } \ul{\I}_{\bfu, (k+1, \ell)}[\wt{\omega}^2] > x_{\bfu, 2}\text{ for } \bfu \in \rightset_{k, \ell}\} \\
&= \bfP\{\ul{\I}_{\bfu, (k, \ell)}[\wc{\eta}^{1, \epsilon}] > x_{\bfu, 1} \text{ and } \ul{\I}_{\bfu, (k+1, \ell)}[\wc{\eta}^{2, \epsilon}] > x_{\bfu, 2}\text{ for } \bfu \in \rightset_{k, \ell}\}. 
\end{split}
\end{align}

\begin{figure}
\begin{subfigure}[t]{0.45\textwidth}
\centering
\begin{tikzpicture}[scale = 0.8]
\draw[](0, 0)--(0, 4.5);
\draw[](0, 0)--(4, 0);
\draw[](0, 0)rectangle(3, 2);
\draw[](0, 4.5)--(0, 6);
\draw[](0, 2)--(4, 2);
\draw[line width = 3, red](3, 0)--(3, 6);
\draw[line width = 3, blue](4, 0)--(4, 6);
\draw[fill=black](4, 2)circle(0.10);
\draw[->, dashed](1.5, 1)--(3-0.1, 2-0.1);
\draw[->, dashed](1.5, 1)--(4-0.1, 2-0.1);
\draw[->](1.5, 2)--(3-0.1, 6-0.1);
\draw[->](2.5, 2.1)--(4-0.15, 2.1)--(4-0.15, 6-0.1);
\draw(3, 0-0.2)node[below]{$k$};
\draw(4, 0-0.2)node[below]{$k+1$};
\draw(0-0.2, 2)node[left]{$\ell$};
\draw(0-0.12, 6)node[left]{$\infty$};
\end{tikzpicture}
\subcaption{
Everywhere except row $\ell$, the weights $\wc{\omega}^\epsilon$ agree with $\omega^1$.
On row $\ell$, there is a single blowup point $\wc{\omega}^\epsilon_{(k+1, \ell)}\to\infty$ as $\epsilon\to0$, but otherwise $\wc{\omega}^\epsilon$ converges to $\wc{\omega}$ defined in \eqref{E:143}.
The blowup forces all the geodesics from row $\ell$ to column $k+1$ to pass through $(k+1, \ell)$ as indicated.
These facts together allow one to compute the limits of the induced weights $\wc{\eta}^{1, \epsilon}$ and $\wc{\eta}^{2, \epsilon}$ as in \eqref{E:146}.
}
\end{subfigure}
\quad 
\begin{subfigure}[t]{0.45\textwidth}
\centering
\begin{tikzpicture}[scale = 0.8]
\draw[](0, 0)--(0, 4.5);
\draw[](0, 0)--(4, 0);
\draw[](0, 0)rectangle(3, 2);
\draw[](0, 4.5)--(0, 6);
\draw[](0, 2)--(4, 2);
\draw[line width = 3, red](3, 0)--(3, 6);
\draw[line width = 3, blue](4, 0)--(4, 2);
\draw[line width = 3, blue](0, 2)--(4, 2);
\draw[fill=white](4, 2)circle(0.10);
\draw[->, dashed](1.5, 1)--(3-0.1, 2-0.1);
\draw[->](1.5, 1)--(4-0.1, 2-0.1);
\draw[->](1.5, 2)--(3-0.1, 6-0.1);
\draw(3, 0-0.2)node[below]{$k$};
\draw(4, 0-0.2)node[below]{$k+1$};
\draw(0-0.2, 2)node[left]{$\ell$};
\draw(0-0.12, 6)node[left]{$\infty$};
\end{tikzpicture}
\subcaption{For $p=2$, $\wc{\eta}^{2, \epsilon}\to\wc{\omega}$ implies the convergence of increments: $\ul{\I}_{\bfu, (k+1, \ell)}[\wc{\eta}^{2, \epsilon}]\to\ul{\I}_{\bfu, (k+1, \ell)}[\wc{\omega}]$ as in \eqref{E:147}.
It is essential to note that on the grid $[k+1]\times[\ell]$ (thick blue sides), the weights $\wc{\omega}$ are equal in distribution to the desired environment $\eta$.
}
\end{subfigure} \\[0.5\baselineskip]
\begin{subfigure}[t]{0.45\textwidth}
\centering
\begin{tikzpicture}[scale = 0.8]
\draw[](0, 0)--(0, 4.5);
\draw[](0, 0)--(4, 0);
\draw[](0, 0)rectangle(3, 2);
\draw[](0, 4.5)--(0, 6);
\draw[](0, 3)--(3, 3);
\draw[line width = 3, red](3, 0)--(3, 6);
\draw[line width = 3, blue](4, 0)--(4, 2);
\draw[line width = 3, blue](0, 2)--(4, 2);
\draw[fill=white](4, 2)circle(0.10);
\draw[->, dashed](1.5, 1)--(3-0.1, 3-0.1);
\draw[->](1.5, 1)--(4-0.1, 2-0.1);
\draw[->](1.5, 3)--(3-0.1, 6-0.1);
\draw(3, 0-0.2)node[below]{$k$};
\draw(4, 0-0.2)node[below]{$k+1$};
\draw(0-0.2, 2)node[left]{$\ell$};
\draw(0-0.2, 3)node[left]{$\ell+1$};
\draw(0-0.12, 6)node[left]{$\infty$};
\end{tikzpicture}
\subcaption{For $p=1$, $\wc{\eta}^{1,\epsilon}_{(i,\ell)}\to\one_{\{i < k\}} \cdot \sI_{(i, \ell)}^{k, \uparrow}[\wc{\omega}]$ implies the convergence of increments to thin Busemann functions:
$\ul{\I}_{\bfu, (k, \ell)}[\wc{\eta}^{1, \epsilon}]\to\sI_{\bfu}^{k, \uparrow}[\wc{\omega}]$ as in \eqref{E:148}. 
In turn, $\sI_{\bfu}^{k, \uparrow}[\wc{\omega}]$ is equal to the increment (dashed arrow) $\ul{\I}_{\bfu, (k, \ell+1)}[\wt{\omega}^3]$ with induced weights $\wt{\omega}^3$ that agree with $\wc{\omega}$ on $[k]\times[\ell]$ and are defined on row $\ell+1$ as in \eqref{E:150}.}
\end{subfigure}
\qquad 
\begin{subfigure}[t]{0.45\textwidth}
\centering
\begin{tikzpicture}[scale = 0.8]
\draw[](0, 0)--(0, 3);
\draw[](0, 0)--(4, 0);
\draw[](0, 0)rectangle(3, 2);
\draw[line width = 3, red](0, 3)--(3, 3);
\draw[line width = 3, red](3, 0)--(3, 3);
\draw[line width = 3, blue](4, 0)--(4, 2);
\draw[line width = 3, blue](0, 2)--(4, 2);
\draw[fill=white](4, 2)circle(0.10);
\draw[fill=white](3, 3)circle(0.10);
\draw[->](1.5, 1)--(3-0.1, 3-0.1);
\draw[->](1.5, 1)--(4-0.1, 2-0.1);
\draw(3, 0-0.2)node[below]{$k$};
\draw(4, 0-0.2)node[below]{$k+1$};
\draw(0-0.2, 2)node[left]{$\ell$};
\draw(0-0.2, 3)node[left]{$\ell+1$};
\end{tikzpicture}
\subcaption{
By properties of thin Busemann functions on row $\ell+1$ (Section~\ref{S:iBusFn}), the restriction of $\wt{\omega}^3$ to the grid $[k]\times[\ell+1]$ (thick red sides) is equal in distribution to $\eta$.
We thus have $(\ul{\I}_{\bfu, (k, \ell+1)}[\wt{\omega}^3], \ul{\I}_{\bfu, (k+1, \ell)}[\wc{\omega}])_\bfu \deq (\ul{\I}_{\bfu, (k, \ell+1)}[\eta], \ul{\I}_{\bfu, (k+1, \ell)}[\eta])_{\bfu}$ as in \eqref{E:151}. 
Now use (b) and (c) to complete the proof.
}
\end{subfigure}
\caption{\small Illustrates the rest of the proof after \eqref{E:140}.
The goal is to send $\epsilon\to0$.
A black dot indicates an $\Exp\{\epsilon\}$-weight (tending to $\infty$), while a white dot indicates a zero weight.}
\label{F:TBusDisId2}
\end{figure}

The next step involves letting $\epsilon \to 0$ in \eqref{E:140}. 
See Figure~\ref{F:TBusDisId2} for an illustration. To compute the limit of the weights on row $\ell$ as $\epsilon \to 0$, we note that, through the  increment recursions noted in \eqref{E:IncRec}, the weight $\wc{\eta}^{1, \epsilon}_{(i, \ell)}$ for $i \in [k]$ (which is equal to the thin Busemann function $\sI_{(i, \ell)}^{k, \uparrow}[\wc{\omega}^{\epsilon}]$)  can be obtained via a fixed continuous function of the thin Busemann functions on row $\ell + 1$, $(\sI_{(i, \ell+1)}^{k, \uparrow}[\wc{\omega}^{\epsilon}])_{i \in [k-1]}$ as well as the weights on row $\ell$, $(\wc{\omega}^{\epsilon}_{(i, \ell)})_{i \in [k]}$ (see the discussion around \eqref{I1}--\eqref{E:65} for the precise proof). But $\sI_{(i, \ell+1)}^{k, \uparrow}[\wc{\omega}^{\epsilon}] = \sI_{(i, \ell+1)}^{k, \uparrow}[\omega^1]$ since this Busemann function only sees the weights on rows $\ell + 1$ and above. Hence, $\wc{\eta}^{1, \epsilon}_{(i, \ell)}$ can be written as a fixed continuous function of  $(\wc{\omega}^{\epsilon}_{(i, \ell)})_{i \in [k]}$ and a quantity that does not depend on $\epsilon$.

Since these weights $\wc{\omega}^{\epsilon}_{(i, \ell)}$ have the marginal distributions given in \eqref{E:130}, it follows that, in a suitable coupling,
\begin{align}
\label{E:142}
\wc{\omega}^{\epsilon}_{(i, \ell)} \stackrel{\rm{a.s.}}{\to} \wc{\omega}_{(i, \ell)} \quad \text{ as } \epsilon \to 0 \quad \text{ for } i \in [k], \quad\text{and }\quad\wc{\omega}^{\epsilon}_{(k+1, \ell)} \to \infty,
\end{align}
where the limit weights $\wc{\omega}$ on $[k] \times \{\ell\}$ are independent, with the marginal distributions 
\begin{align}
\label{E:143}
\wc{\omega}_{(i, \ell)} \sim \Exp\{\zeta(r_2)\} \quad \text{ for } i \in [k-1], \quad\text{and}\quad
\wc{\omega}_{(k, \ell)} \sim \Exp\{\zeta(r_2)-\zeta(r_1)\}.  
\end{align}
We extend $\wc{\omega}$ to an independent collection on $[k+1] \times \bbZ_{>0}$ by setting $\wc{\omega}_{(k+1, \ell)} = 0$ and $\wc{\omega} = \omega^1$ on all rows except row $\ell$. 
Then, by \eqref{E:139} and our discussion below \eqref{E:140},
\begin{align}
\label{E:144}
\wc{\eta}^{1}_{(i, \ell)}  \coloneqq \lim_{\epsilon \to 0} \wc{\eta}^{1, \epsilon}_{(i, \ell)} \stackrel{\rm{a.s.}}{=} \one_{\{i < k\}} \cdot \sI_{(i, \ell)}^{k, \uparrow}[\wc{\omega}] \quad \text{ for } i \in [k].
\end{align}
To identify the limits of $\wc{\eta}^{2, \epsilon}$ on row $\ell$, one needs to proceed differently because the weight $\wc{\omega}_{(k+1, \ell)}^{\epsilon} \sim \Exp\{\epsilon\}$ becomes infinite as $\epsilon \to 0$. In this case, it becomes advantageous to use the weight $\wc \omega_{(k+1,\ell)}^\epsilon$ for sufficiently small $\epsilon$. This implies that $\sI_{(i, \ell)}^{k+1, \uparrow}[\wc{\omega}^{\epsilon}] \stackrel{\rm{a.s.}}{\to} \wc{\omega}_{(i, \ell)}$ for $i \in [k]$ as $\epsilon \to 0$ (made precise in \eqref{eq:rec_lim}--\eqref{E:68}). Hence, from \eqref{E:139},
\begin{align}
\label{E:145}
\lim_{\epsilon \to 0} \wc{\eta}^{2, \epsilon}_{(i, \ell)} \stackrel{\rm{a.s.}}{=} \one_{\{i < k+1\}} \cdot \wc{\omega}_{(i, \ell)}  = \wc{\omega}_{(i, \ell)} \quad \text{ for } i \in [k+1].
\end{align}
One then concludes from \eqref{E:144} and \eqref{E:145} that, almost surely, as $\epsilon \to 0$, 
\begin{align}
\label{E:146}    
\begin{split}
&\wc{\eta}^{1, \epsilon}_{(i, j)} \to \one_{\{j < \ell\}} \cdot \wc{\omega}_{(i, j)} + \one_{\{j = \ell\}} \cdot \one_{\{i < k\}} \cdot \sI_{(i, \ell)}^{k, \uparrow}[\wc{\omega}] \quad \text{ for } (i, j) \in [k] \times [\ell], \text{ and} \\
&\wc{\eta}^{2, \epsilon}_{(i, j)} \to \wc{\omega}_{(i, j)} \quad \text{ for } (i, j) \in [k+1] \times [\ell].
\end{split}
\end{align}
 The second limit in \eqref{E:146} implies that 
\begin{align}
\label{E:147}
\ul{\I}_{\bfu, (k+1, \ell)}[\wc{\eta}^{2, \epsilon}] \stackrel{\rm{a.s.}}{\to} \ul{\I}_{\bfu, (k+1, \ell)}[\wc{\omega}] \quad \text{ as } \epsilon \to 0 \text{ for } \bfu \in \rightset_{k, \ell}.   
\end{align}
One can also recognize from \eqref{E:127} that the first limit in \eqref{E:146} is the induced weights obtained by applying the map $\wt{(\cdot)}^{k, \ell}$ to $\wc{\omega}$. Therefore, using the identity \eqref{E:126} twice, one obtains that 
\begin{align}
\label{E:148}    
\ul{\I}_{\bfu, (k, \ell)}[\wc{\eta}^{1, \epsilon}] \stackrel{\rm{a.s.}}{\to} \sI_{\bfu}^{k, \uparrow}[\wc{\omega}] = \ul{\I}_{\bfu, (k, \ell+1)}[\wt{\omega}^{3}],
\end{align}
where $\wt{\omega}^3$ is the image of $\wc{\omega}$ under the map $\wt{(\cdot)}^{k, \ell+1}$. Now, sending $\epsilon \to 0$ at \eqref{E:140} and using \eqref{E:147} and \eqref{E:148} yields the identity 
\begin{align}
\label{E:149}
\begin{split}
&\bfP\{\ul{\I}_{\bfu, (k, \ell)}[\wt{\omega}^1] > x_{\bfu, 1} \text{ and } \ul{\I}_{\bfu, (k+1, \ell)}[\wt{\omega}^2] > x_{\bfu, 2}\text{ for } \bfu \in \rightset_{k, \ell}\} \\
&= \bfP\{\ul{\I}_{\bfu, (k, \ell+1)}[\wt{\omega}^{3}] > x_{\bfu, 1} \text{ and } \ul{\I}_{\bfu, (k+1, \ell)}[\wc{\omega}] > x_{\bfu, 2}\text{ for } \bfu \in \rightset_{k, \ell}\}. 
\end{split}
\end{align}

Since $d = 2$, the $\eta$-weights are defined on the grid $[k+1] \times [\ell+1]$. Using the distribution of the $\eta$-weights given in \eqref{E:105} and comparing to \eqref{E:143}, one can see that $\eta$ matches in distribution with $\wc{\omega}$ on $[k+1] \times [\ell]$. Also, the induced weights $\wt{\omega}^3$ are independent, agree with $\wc{\omega}$ on $[k] \times [\ell]$, and, by Proposition \ref{P:ThinBuse}, have the following marginal distributions on row $\ell+1$: 
\begin{align}
\label{E:150}
\begin{split}
\wt{\omega}^3_{(i, \ell+1)} &= \sI_{(i, \ell+1)}^{k, \uparrow}[\wc{\omega}] = \sI_{(i, \ell+1)}^{k, \uparrow}[\omega] \sim \Exp\{\zeta(r_1)\} \text{ for } i \in [k-1], \text{ and }\\ 
\wt{\omega}^3_{(k, \ell+1)} &= 0. 
\end{split}
\end{align}
The second equality in \eqref{E:150} holds because $\wc{\omega}$ and $\omega$ are the same strictly above row $\ell$. The marginal distributions and independence of the weights $\wt{\omega}^3$ on row $\ell+1$ invoke the properties of thin Busemann functions, which are stated in Propositions \ref{P:ThinBuse} and \ref{P:TBuseInd}\footnote{Proposition \ref{P:TBuseInd}, which states an independence property of thin Busemann functions, is employed here to shorten the exposition for two directions by one iteration. Our iterative argument in Section \ref{S:PfBuse} for an arbitrary number of directions does not appeal to this proposition.} ahead. It follows that the weights $\eta$ and $\wt{\omega}^3$ are identical in distribution on $[k] \times [\ell+1]$. Combining the preceding observations leads to the identity 
\begin{align}
\label{E:151}
\begin{split}
&\bfP\{\ul{\I}_{\bfu, (k, \ell+1)}[\wt{\omega}^{3}] > x_{\bfu, 1} \text{ and } \ul{\I}_{\bfu, (k+1, \ell)}[\wc{\omega}] > x_{\bfu, 2}\text{ for } \bfu \in \rightset_{k, \ell}\} \\ 
&= \bfP\{\ul{\I}_{\bfu, (k, \ell+1)}[\eta] > x_{\bfu, 1} \text{ and } \ul{\I}_{\bfu, (k+1, \ell)}[\eta] > x_{\bfu, 2}\text{ for } \bfu \in \rightset_{k, \ell}\}. 
\end{split}
\end{align}
Now, revisiting \eqref{E:125} and following the chain of identities \eqref{E:128}, \eqref{E:149}, and \eqref{E:151} yields the desired lower bound 
\begin{align}
\label{E:152}
\begin{split}
&\bfP\{\sI_{\bfu}^{r_1}[\omega] > x_{\bfu, 1} \text{ and } \sI_{\bfu}^{r_2}[\omega] > x_{\bfu, 2} \text{ for } \bfu \in \rightset_{k, \ell} \} \\ 
&\ge \bfP\{\ul{\I}_{\bfu, (k, \ell+1)}[\eta] > x_{\bfu, 1} \text{ and } \ul{\I}_{\bfu, (k+1, \ell)}[\eta] > x_{\bfu, 2}\text{ for } \bfu \in \rightset_{k, \ell}\}. 
\end{split}
\end{align}
As previously mentioned, the opposite inequality can be obtained through similar steps. Hence, the claim \eqref{E:114} holds.  

\section{Invariance of inhomogeneous exponential LPP}
\label{Ss:LppInv}

As indicated in the introduction and Section \ref{Ss:BMdiscuss}, our proof of Theorem \ref{T:BusMar} naturally proceeds through a well-known integrable generalization of the exponential LPP model \cite{Baik_BenA_Pech_05, Boro_Pech_08, Joha_00b, Joha_08}, which can be defined on the lattice $\bbZ^2$ as follows. Given inhomogeneity parameters $\bfa = (a_i)_{i \in \bbZ} \in \bbR^{\bbZ}$ and $\bfb = (b_j)_{j \in \bbZ} \in \bbR^{\bbZ}$
such that 
\begin{align}
\label{E:Minab}
a_i + b_j > 0 \quad \text{ for } i, j \in \bbZ, 
\end{align} 
consider independent weights $\omega^{\bfa, \bfb} = \{\omega^{\bfa, \bfb}_{(i, j)}: i, j \in \bbZ\}$
with the marginals 
\begin{align}
\label{E:iw}
\omega^{\bfa, \bfb}_{(i, j)} \sim \Exp\{a_i+b_j\} \quad \text{ for } i, j \in \bbZ. 
\end{align}

Our argument relies extensively on a remarkable joint invariance property of the inhomogeneous LPP process $\Lp[\omega^{\bfa, \bfb}]$ under certain permutations of columns and rows. To formulate the property precisely, let $\sigma$ and $\tau$ be finite permutations of $\bbZ$, meaning that $\sigma(i) = \tau(i) = i$ for all but finitely many $i \in \bbZ$.  Let $\sU_{\sigma, \tau}$ denote the set of all $(\bfx, \bfy) \in \bbZ^4$ with $\bfx = (x_1, x_2) \le (y_1, y_2) = \bfy$ subject to the following conditions.  
\begin{align}
\label{E:U}
\begin{split}
&\sigma(\bbZ_{<x_1}) = \bbZ_{<x_1} \quad \text{ and } \quad \sigma(\bbZ_{>y_1}) = \bbZ_{>y_1}, \\
&\tau(\bbZ_{<x_2}) = \bbZ_{<x_2} \quad \text{ and } \quad \tau(\bbZ_{>y_2}) = \bbZ_{>y_2}. 
\end{split}
\end{align}
In other words, $\sigma$ restricts to permutations of both $\bbZ_{<x_1}$ and $\bbZ_{>y_1}$, and a similar remark applies to $\tau$.  
See Figure \ref{fig:U-paths} below for an illustration of the set $\sU_{\sigma, \tau}$. A basic yet important special case is when $\sigma$ is a transposition and $\tau$ is the identity permutation, or vice versa. Then the set $\sU_{\sigma, \tau}$ can be described more simply as in Remark \ref{R:1Swap} below. 

The following invariance theorem shows that the joint distribution of the last-passage times between pairs of vertices in $\sU_{\sigma, \tau}$ is preserved after permuting the columns by $\sigma$ and rows by $\tau$. 

\begin{thm}
\label{T:LppInv}
Let $\bfa = (a_i)_{i \in \bbZ} \in \bbR^\bbZ$ and $\bfb = (b_j)_{j \in \bbZ} \in \bbR^\bbZ$ be subject to \eqref{E:Minab}, and $\sigma$ and $\tau$ be finite permutations of $\bbZ$. Then 
\begin{align*}
\{\Lp_{\bfx, \bfy}[\omega^{\bfa, \bfb}]: (\bfx, \bfy) \in \sU_{\sigma, \tau}\} \deq \{\Lp_{\bfx, \bfy}[\omega^{\bfa_\sigma, \bfb_\tau}]: (\bfx, \bfy) \in \sU_{\sigma, \tau}\},
\end{align*}
where $\bfa_\sigma = (a_{\sigma(i)})_{i \in \bbZ}$ and $\bfb_\tau = (b_{\tau(j)})_{j \in \bbZ}$.
\end{thm}

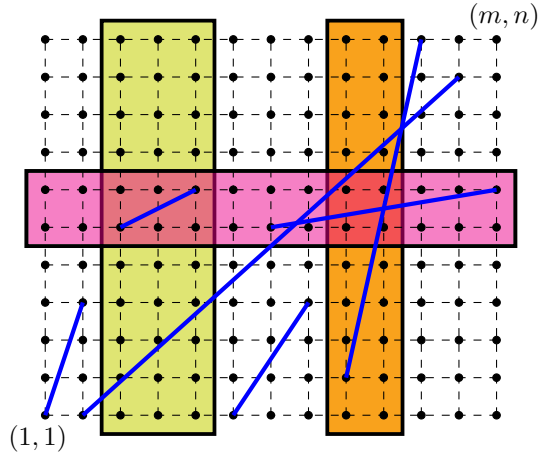
\begin{figure}[h]
\centering
\begin{tikzpicture}[scale = 1]
\draw[color=black, fill=GreenYellow,line width = 0.5mm](1-0.25, -0.25)rectangle(2+0.25, 5.25);
\draw[color=black, fill=YellowOrange,line width = 0.5mm](4-0.25, -0.25)rectangle(4.5+0.25, 5.25);
\draw[color=black, fill=Magenta, opacity = 0.5,line width = 0.5mm](-0.25, 2.5-0.25)rectangle(6.25, 3+0.25);
\draw[color=black, line width = 0.5mm](-0.25, 2.5-0.25)rectangle(6.25, 3+0.25);
\draw[step=0.5, thin, dashed](0, 0)grid(6, 5);
\foreach \x in {0, 0.5, ..., 6} 
\foreach \y in {0, 0.5, ..., 5}
\filldraw[color=black](\x, \y)circle(0.05);
\draw[-, color=blue, ultra thick](0, 0)--(0.5, 1.5);
\draw[-, color=blue, ultra thick](0.5, 0)--(5.5, 4.5);
\draw[-, color=blue, ultra thick](4, 0.5)--(5, 5);
\draw[-, color=blue, ultra thick](1, 2.5)--(2, 3);
\draw[-, color=blue, ultra thick](3, 2.5)--(6, 3);
\draw[-, color=blue, ultra thick](2.5, 0)--(3.5, 1.5);
\draw(-0.1+0, 0)node[below]{\small$(1, 1)$};
\draw(6+0.1, 5)node[above]{\small$(m, n)$};
\end{tikzpicture}
\caption{\small Illustrates an application of Theorem \ref{T:LppInv} on the finite grid $[m] \times [n]$ with $m = 13$ and $n = 11$. 
By the action of $\sigma$, the parameters of columns $9$ and $10$ (orange) are interchanged, and the parameters of columns $3$, $4$, and $5$ (green) are rearranged in some way.
By the action of $\tau$, the parameters of rows $6$ and $7$ (magenta) are interchanged.
Each solid blue line segment represents a last-passage time between its endpoints. 
Since each endpoint pair satisfies condition \eqref{E:U}, the associated last-passage times have the same joint distribution before and after the actions of $\sigma$ and $\tau$.
}
\label{fig:U-paths}
\end{figure}

In the subsequent remarks, we consider some implications and special cases of Theorem \ref{T:LppInv}. 
\begin{rem}[Recovery of some known invariances]
\label{R:SchInv}    
Let $m, n \in \bbZ_{>0}$, and assume that $\sigma(i) = i$ for $i \in \bbZ \smallsetminus [m]$, and $\tau(j) = j$ for $j \in \bbZ \smallsetminus [n]$. Then $\sU_{\sigma, \tau}$ contains the point $((1, 1), (m, n))$, which satisfies the condition \eqref{E:U}. Therefore,  Theorem \ref{T:LppInv} implies that
\begin{align}
\label{E:LppInv1}
\Lp_{(1, 1), (m, n)}[\omega^{\bfa, \bfb}] \deq \Lp_{(1, 1), (m, n)}[\omega^{\bfa_\sigma, \bfb_\tau}].
\end{align}
The preceding identity is a well-known invariance property 
and follows, for example, from the explicit formula \cite[Equation (12)]{Boro_Pech_08} for the CDF of last-passage times. 

Assume further that $\tau = \Id$, the identity permutation on $\bbZ$. As can be seen from \eqref{E:U}, one now has $((1, 1), (m, j)) \in \sU_{\sigma, \tau}$ for each $j \in [n]$.  Hence, by Theorem \ref{T:LppInv}, 
\begin{align}
\label{E:LppInv2}
\{\Lp_{(1, 1), (m, j)}[\omega^{\bfa, \bfb}]: j \in [n]\} \deq 
\{\Lp_{(1, 1), (m, j)}[\omega^{\bfa_\sigma, \bfb}]: j \in [n]\}.  
\end{align}
This generalization of \eqref{E:LppInv1} can also be seen, for example, by expressing the joint CDFs for each side of \eqref{E:LppInv2} in terms of the explicit correlation kernel in \cite[Theorem 3]{Boro_Pech_08}. See also \cite[Theorem 2.3]{Dauv_22} for a multi-path extension of \eqref{E:LppInv2}. 
\qedex
\end{rem}
\begin{rem}
[Interchanging the parameters of two columns or two rows]
\label{R:1Swap}
In the present work, Theorem \ref{T:LppInv} will be invoked only with the sets of the form $\sU_{\sigma, \Id}$ and $\sU_{\Id, \sigma}$ where $\sigma$ is a transposition. This means that $\sigma(k) = \ell$ and $\sigma(\ell) = k$ for some $k, \ell \in \bbZ$ with $k < \ell$, and $\sigma(i) = i$ for $i \in \bbZ \smallsetminus \{k, \ell\}$. Then \eqref{E:U} shows that the set $\sU_{\sigma, \Id}$ consists of all pairs $(\bfx, \bfy) \in \bbZ^4$ with $\bfx = (x_1, x_2) \le (y_1, y_2) = \bfy$ such that    
\begin{align}
\label{E:U1Swp}
x_1 > \ell, \quad \text{ or } \quad y_1 < k, \quad \text{ or } \quad x_1 \le k \text{ and } y_1 \ge \ell. 
\end{align}
Similarly, the set $\sU_{\Id, \sigma}$ is characterized with the condition  
\begin{align}
\label{E:U2Swp}
x_2 > \ell, \quad \text{ or } \quad y_2 < k, \quad \text{ or } \quad x_2 \le k \text{ and } y_2 \ge \ell    
\end{align}
in place of \eqref{E:U1Swp}. 
\qedex
\end{rem}

\begin{rem}
[Connection to an invariance result of Dauvergne]
\label{R:DauvInv}
Consider again a transposition $\sigma$. Then, the set $\sU_{\sigma, \Id}$ is the disjoint union of three subsets $U_1$, $U_2$, and $U_3$ corresponding, respectively, to the three conditions in \eqref{E:U1Swp}. Theorem \ref{T:LppInv} implies that   
\begin{align}
\label{E:IdUi}
\{\Lp_{\bfx, \bfy}[\omega^{\bfa, \bfb}]: (\bfx, \bfy) \in U_i\} \deq \{\Lp_{\bfx, \bfy}[\omega^{\bfa_\sigma, \bfb}]: (\bfx, \bfy) \in U_i\}  
\end{align}
for each $i \in \{1, 2, 3\}$. Note that \eqref{E:IdUi} is trivial for $i \in \{1, 2\}$ because, in these cases, the action of $\sigma$ is not visible to the last-passage times on $U_i$. In the nontrivial case where $i = 3$, a somewhat weaker version of the identity \eqref{E:IdUi} can also be derived from \cite{Dauv_22} as follows. 

Fix $j \in \bbZ$ and assume that $\sigma$ interchanges $k, \ell \in \bbZ$ with $k < \ell$. Consider the set 
\begin{align}
\label{E:DauvU}
U = \bbZ_{\le k} \times  \bbZ_{\le j} \times \bbZ_{\ge \ell} \times \bbZ_{\ge j+1}, 
\end{align}
which is properly contained in $U_3$. For any $p^- \in \bbZ_{\le k}$ and $p^+ \in \bbZ_{\ge \ell}$, 
\begin{align}
\label{E:DauvCond}
\begin{split}
(a_{\sigma(p^-)}, \dotsc, a_{\sigma(p^+)}) \text{ is a permutation of } (a_{p^-}, \dotsc, a_{p^+}), 
\end{split}    
\end{align}
which is immediate from the definitions of $\sigma$ and $U_3$. It follows from \eqref{E:DauvCond} that the first part of the rearrangement condition stated in \cite[p.\ 25]{Dauv_22} holds. (In the notation of \cite{Dauv_22}, we consider a single set $U_i$ called $U$ here and choose the corresponding shift variable as $c_{i, 1} = 0$). The second part of the same condition trivially holds since the rows are not being permuted. We can then conclude via \cite[Theorem 1.5]{Dauv_22} that 
\[
\{\Lp_{\bfx, \bfy}[\omega^{\bfa, \bfb}]: (\bfx, \bfy) \in U\} \deq \{\Lp_{\bfx, \bfy}[\omega^{\bfa_\sigma, \bfb}]: (\bfx, \bfy) \in U\}.   
\]
It is plausible that one can use the methods of \cite{Dauv_22} to derive the stronger versions of \eqref{E:IdUi} with $U_3$ or even $\sU^{\sigma, \Id} = U_1 \cup U_2 \cup U_3$ in place of $U$. However, it is presently unclear to us how to do so. \qedex
\end{rem}

\section{Proof of invariance}
\label{S:PfInv}

We turn to the proof of Theorem \ref{T:LppInv}. In the present section, let $\omega = \omega^{\bfa, \bfb}$ denote the independent exponential weights with the marginals as in \eqref{E:iw}. 
For convenience, assume that 
\begin{equation}\label{constant-a}
a_i=a>0 \text{ for }|i|\text{ sufficiently large}.
\end{equation}
This suffices for the proof of Theorem \ref{T:LppInv} because the joint distribution of passage times is determined by its finite-dimensional distributions.  

To begin, we consider just the rows $j=1,2$. Assume $b_2>b_1$ 
(so that the weights on row $1$ tend to be larger than those on row $2$). For $t\in\ZZ$, let
\begin{equation}\label{cdef}
c_t=\omega_{(t+1,1)}-\omega_{(t,2)}
\end{equation}
(the difference between two diagonally adjacent weights). 
Now, define 
\begin{equation}\label{Udef}
U_t=U_t[\omega]=\left(
\inf_{s<t}
\sum_{i=s}^{t-1}
c_i
\right)_{\!+},
\end{equation}
where $x_+ = \max\{x,0\}$.
Then, let
\begin{equation}\label{wtilde-def}
\begin{split}
\tw_{(t,1)}&=\omega_{(t,1)}-U_t, 
\\
\tw_{(t,2)}&=\omega_{(t,2)}+U_t. 
\end{split}
\end{equation}

The following proposition is a deterministic statement that holds for arbitrary weights, as long as the infimum in \eqref{Udef} is achieved at some finite $s < t$ for all $t \in \Z$, and for each $T \in \Z$, there exists $t> T$ such that $U_t > 0$. This holds under assumption \eqref{constant-a}; see Lemma~\ref{lemma:inf-attained} and the paragraph above it. 
\begin{prop}\label{prop:passage times-preserved}
For all $i\leq j$,
\[
\Lp_{(i,1),(j,2)}[\omega]=\Lp_{(i,1),(j,2)}[\tw].
\]
\end{prop}

\begin{prop}\label{prop:rates-exchanged}
The weights $\bigl\{\tw_{(i,j)}: i\in\ZZ, j\in\{1,2\}\bigr\}$ are independent with 
$\tw_{(i,j)}\sim\Exp\{a_i+b^*_j\}$, where $b^*_1=b_2$, $b^*_2=b_1$. 
\end{prop}

Together, these propositions will tell us that the transformation in (\ref{wtilde-def}) interchanges the rates of rows $1$ and $2$, while preserving all passage times which cross both row $1$ and row $2$. 

We will prove Proposition 4.2 by demonstrating a version of Burke’s theorem for a time-inhomogeneous queue whose arrival and service data correspond to the weights $\omega$ on rows 1 and 2 respectively. In the queueing context, the quantity $U_t$ will represent unused service at time $t$. Interchangeability properties for queues date back to Weber \cite{Weber1979}, who showed that the law of the departure process from two $\cdot/M/1$ queues in series is unchanged when the rates are swapped. An alternative proof of Weber's result by Tsoucas and Walrand \cite{TsoWal1987} introduced a coupling idea involving unused service (which is generalised by \eqref{wtilde-def}). Related constructions have more recently been deployed in a wide variety of contexts in percolation and polymer models and in interacting particle systems. A positive temperature version of this invariance was previously seen in the work of Noumi and Yamada \cite{noum-yama-04}. A continuous version of the invariance of passage times in Proposition \ref{prop:passage times-preserved} was used in \cite[Lemma 4.12]{Dauv_Ortm_Vira_22} to construct the directed landscape -- see also \cite{corwin21,Dauv_Nica_Vira_22, dauvergne-virag-24} for related results. 

We prove Propositions \ref{prop:passage times-preserved} and \ref{prop:rates-exchanged} in the subsequent sections. First, we use them to prove Theorem \ref{T:LppInv}. 

\begin{proof}[Proof of Theorem \ref{T:LppInv}]
Consider two neighboring rows $r$ and $r+1$, and consider the joint distribution of all passage times with lower-left endpoint on row $r$ and upper-right endpoint on row $r+1$. Proposition \ref{prop:passage times-preserved} and Proposition \ref{prop:rates-exchanged} tell us that the joint distribution
of this collection 
\begin{equation}\label{Lp-collection}
(\Lp_{(s_1,r), (s_2, r+1)}[\omega]: s_1\leq s_2)
\end{equation}
is invariant under interchange of the parameters $b_r$ and $b_{r+1}$ associated to the two rows. In fact, these two results give an explicit coupling under which all these passage times are the same before and after the interchange. 

Note that in fact, many other passage times are also preserved by the same operation. Passage times with lower-left endpoint above row $r+1$, or with upper-right endpoint below row $r$ do not use the weights in rows $r$ and $r+1$ and are unaffected. Meanwhile, all passage times whose lower-left endpoint is on or below row $r$, and whose upper-right endpoint is on or above row $r+1$, depend on the weights in rows $r$ and $r+1$ only through the collection of passage times in (\ref{Lp-collection}). For example, let $x_1\leq y_1$ and $x_2<r, y_2>r+1$. Then, 
\begin{equation}
\begin{split}\label{eq:2-to-many}
&\Lp_{(x_1,x_2), (y_1, y_2)}[\omega]
\\
&\,\,\,\,\,\,\,\,=
\max_{x_1\leq s_1\leq s_2\leq y_1}
\bigl[\Lp_{(x_1, x_2), (s_1, r-1)}[\omega]
+\Lp_{(s_1, r), (s_2, r+1)}[\omega]
+\Lp_{(s_2, r+2), (y_1, y_2)}[\omega]\bigr].
\end{split}
\end{equation}
The cases where $x_2=r$ or $y_2=r+1$ are similar. 
The only passage times not preserved are those which have a lower-left endpoint in row $r+1$ or an upper-right endpoint in row $r$. 

Now consider finite permutations $\sigma$ and $\tau$ which permute the column parameters and row parameters respectively. The permutation $\tau$ can be represented as a product of nearest-neighbor transpositions $(r_1,r_1+1), \dots, (r_k, r_k+1)$. These transpositions can be chosen such that 
for all endpoint pairs $((x_1, x_2), (y_1, y_2))$
in $\sU_{\sigma, \tau}$, we have 
$x_2\notin \{r_1, \dots, r_k\}$ and  
$y_2\notin \{r_1+1, \dots, r_k+1\}$
(since for all such $(x_1, x_2)$ and $(y_1, y_2)$, 
the permutation $\tau$ preserves $\ZZ_{<x_2}$ and $\ZZ_{>y_2}$).
Consequently, by applying the transpositions in sequence, we can permute the row parameters according to $\tau$ while preserving all passage times 
in $\sU_{\sigma, \tau}$.

Analogously, we can permute the column parameters according to $\sigma$ 
while preserving all passage times 
in $\sU_{\sigma, \tau}$. Since the joint distribution is determined by its finite-dimensional distributions, we can consider the distribution in an arbitrarily large box and remove the assumption \eqref{constant-a}.  
\end{proof}

\begin{rem}
\label{rem:multi-point-extension}
Following exactly the same method, one can extend 
the result of Theorem 
\ref{T:LppInv} to the case of 
\textit{multi-point} last-passage times, 
as considered for example in \cite{Dauv_22}. 
For a set of lower-left endpoints 
$\{\bfx^{(1)}, \dots, \bfx^{(k)}\}$,
and a set of upper-right endpoints $\{\bfy^{(1)}, \dots, \bfy^{(k)}\}$ of the same size $k$, the multipoint last-passage
time is the maximum total weight of $k$ 
vertex-disjoint up-right paths whose set of 
lower left endpoints is $\{\bfx^{(1)}, \dots, \bfx^{(k)}\}$
and whose set of lower right endpoints is $\{\bfy^{(1)}, \dots, \bfy^{(k)}\}$.

Restricted to just two rows, any such multi-point passage time can be re-expressed as a sum of simple last-passage times, so that Proposition \ref{prop:passage times-preserved} immediately extends to that context. 
Then one can argue as at (\ref{eq:2-to-many})
that if no lower-left endpoint lies in row $r+1$, and no 
upper-right endpoint lies in row $r$, then
any multi-point passage-time depends on the weights 
in rows $r$ and $r+1$ only through the multi-point passage times with all lower-left endpoints in row $r$ and all 
upper right endpoints in row $r+1$. 

We could then consider a set $\cM_{\sigma, \tau}$ of 
allowable collections of endpoints 
$\big(\{\bfx^{(1)}, \dots, \bfx^{(k)}\},$
$\{\bfy^{(1)}, \dots, \bfy^{(k)}\}\big)$; 
similarly to (\ref{E:U}), the required
condition is that $\sigma$ preserves each $\ZZ_{<x^{(i)}_1}$ and 
each $\ZZ_{>y^{(i)}_i}$, and $\tau$ preserves each $\ZZ_{<x^{(i)}_2}$ and each $\ZZ_{>y^{(i)}_2}$.
Then the joint distribution of all multi-point last-passage
times with endpoint sets in $\cM_{\sigma,\tau}$
has the same distribution for the weights $\omega^{\bfa, \bfb}$
as for the weights $\omega^{\bfa_\sigma, \bfb_\tau}$.
\qedex
\end{rem}




\subsection{Path properties and proof of Proposition \ref{prop:passage times-preserved}}
Let $i\leq j$. For $x\in[i,j]$, write $K^x_{i,j}[\omega]$ for the $\omega$-weight
of the path which starts at $(i,1)$, passes through $(x,1)$ and $(x,2)$, and finishes at $(j,2)$. Then we get
\begin{equation}\label{Lp2b}
\Lp_{(i,1),(j,2)}[\omega]=
\max_{x\in[i,j]} K^x_{i,j}[\omega].
\end{equation}
Note that, for $x<x'$, 
\begin{equation}\label{Kdiff2}
K^{x'}_{i,j}[\omega]-K^{x}_{i,j}[\omega]
=
\sum_{s=x}^{x'-1}c_s,
\end{equation}
where $c_s$ is defined in \eqref{cdef}. See Figure \ref{fig:passage times2}.
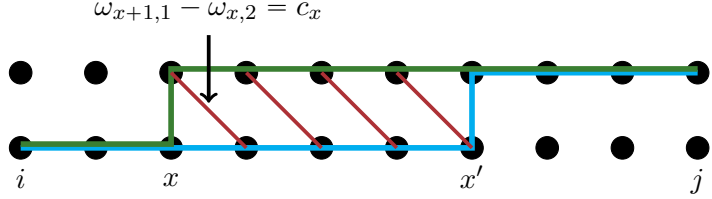
\begin{figure}[h]
\centering
\begin{tikzpicture}[scale = 1]
\foreach \x in {0, 1, ..., 9} 
\foreach \y in {0, 1}
\filldraw[color=black](\x, \y)circle(0.15);
\draw[cyan, line width=0.075cm](0,0)--(6, 0)--(6, 1)--(9, 1);
\draw[OliveGreen, line width=0.075cm](0,0.05)--(2, 0.05)--(2, 1.05)--(9, 1.05);
\draw(0, -0.15)node[below]{$i$};
\draw(2, -0.2)node[below]{$x$};
\draw(6, -0.1)node[below]{$x'$};
\draw(9, -0.15)node[below]{$j$};
\draw[Maroon, line width=0.05cm](2, 1)--(3, 0);
\draw[Maroon, line width=0.05cm](3, 1)--(4, 0);
\draw[Maroon, line width=0.05cm](4, 1)--(5, 0);
\draw[Maroon, line width=0.05cm](5, 1)--(6, 0);
\draw[black, line width=0.05cm, ->](2.5, 1.5)node[above]{$\omega_{x+1, 1}-\omega_{x, 2} = c_x$}--(2.5, 0.6);
\end{tikzpicture}
\caption{\small The green (upper) path has weight $K_{i,j}^x$ and the blue (lower) path has weight $K_{i,j}^{x'}$. The difference in their weights is the sum of the difference terms $c_s$ (indicated in red). 
}
\label{fig:passage times2}
\end{figure}

We will fix $i$ and establish the conclusion of Proposition \ref{prop:passage times-preserved} by induction on $j$. 

Let $x_0=\min\{x\geq i: U_x>0\}$. The set over which the minimum is taken is nonempty almost surely by the definition of $U_t$ \eqref{Udef} and the strong law of large numbers since the difference $\omega_{(t+1,1)} - \omega_{(t,2)}$ has positive mean (here, to apply the law of large numbers, we are using the assumption \eqref{constant-a}). The following property will be useful:
\begin{lem}\label{lemma:inf-attained}
For all $t>x_0$, the infimum in definition (\ref{Udef}), namely
\[
U_t=\left(\inf_{s<t} \sum_{r=s}^{t-1} c_r\right)_+,
\]
is attained at some $s\geq x_0$. 
\end{lem}
\begin{proof}
Let $t>x_0$. Because $U_{x_0}>0$, 
we have that for all $s<x_0$,
\begin{align*}
\sum_{r=s}^{t-1} c_r &= \sum_{r=s}^{x_0-1} c_r + \sum_{r=x_0}^{t-1} c_r\\
&\geq \inf_{s'<x_0} \sum_{r=s'}^{x_0-1} c_r + \sum_{r=x_0}^{t-1} c_r\\
&= U_{x_0} + \sum_{r=x_0}^{t-1} c_r
> \sum_{r=x_0}^{t-1} c_r. \qedhere
\end{align*}
\end{proof}

\begin{proof}[Proof of Proposition \ref{prop:passage times-preserved}]
Fix $i \in \Z$. We will show by induction on $j$ that for all $j\geq i$,
\[
\Lp_{(i,1),(j,2)}[\omega]=\Lp_{(i,1),(j,2)}[\tw].
\]
The base case is
\[
\Lp_{(i,1),(i,2)}[\omega] = \omega_{(i,1)} + \omega_{(i,2)} = \tw_{(i,1)} + \tw_{(i,2)} = \Lp_{(i,1),(i,2)}[\tw],
\]
where the middle equality simply holds by definition \eqref{wtilde-def}. Now, assume the statement holds for some $j-1 \ge i$. We split up into cases, according to whether $j<x_0$, $j=x_0$, $j>x_0$.

\medskip\noindent\textbf{Case 1}: $j< x_0$. 

If $j< x_0$, then $U_x=0$ for all $x\in[i,j]$, and so also $\omega_{(x,1)}=\tw_{(x,1)}$ and 
$\omega_{(x,2)}=\tw_{(x,2)}$. Since all the relevant $\omega$-weights and $\tw$-weights coincide, 
we immediately have $\Lp_{(i,1),(j,2)}[\omega]=\Lp_{(i,1),(j,2)}[\tw]$.

\medskip\noindent\textbf{Case 2}: $j=x_0$.

By the definition of $x_0$, we have $U_j>0$. This implies that the optimal path from $(i,1)$ to $(j,2)$ for the $\omega$-weights passes through $(j,1)$. To see this, note that by \eqref{Kdiff2}, for any $x$ with $i\leq x<j$, 
\begin{equation}\label{Keq}
K^{j}_{i,j}[\omega]-K^{x}_{i,j}[\omega]
=
\sum_{s=x}^{j-1}c_s \geq U_j >0.
\end{equation}

Now consider the optimal $\tw$-path. Because the $\omega$-weights and $\tw$-weights strictly to the left of $x_0$ are the same (see Case 1), and $\tw_{(x_0,1)}+\tw_{(x_0,2)}=\omega_{(x_0,1)}+\omega_{(x_0,2)}$, the path that passes through $(j,1)$ has $\tw$-weight equal to its $\omega$-weight. 
This gives $\Lp_{(i,1),(j,2)}[\tw]\geq\Lp_{(i,1),(j,2)}[\omega]=K_{i,j}^j[\omega]$.

Any other path passes instead through, say, $(x,1)$ and $(x,2)$, with 
weight $K_{i,j}^x[\tw]=K_{i,j}^x[\omega]+U_j$. But as in (\ref{Keq}),
\begin{align*}
K_{i,j}^j[\omega]-\left(K_{i,j}^x[\omega]+U_j\right)
&\ge 0.
\end{align*}
Thus, any path from $(i,1)$ to $(j,2)$ that does not pass through $(j,1)$ has  $\tw$-weight no greater than $K_{i,j}^j[\tw] = K_{i,j}^j[\omega]$, and so $\Lp_{(i,1),(j,2)}[\tw]=\Lp_{(i,1),(j,2)}[\omega]$, as desired. 

\medskip\noindent\textbf{Case 3a}: $j>x_0$, $U_j=0$.

Because $U_j=0$, Lemma \ref{lemma:inf-attained} and (\ref{Kdiff2}) tell us that there
is some $s$ with $x_0\leq s<j$ such that $K_{i,j}^s[\omega]\geq K_{i,j}^j[\omega]$. So
some optimal path for the $\omega$-weights passes through $(j-1,2)$ rather than $(j,1)$. 

Furthermore, $U_j = 0$ implies that $\tw_{(j,2)}=\omega_{(j,2)}$. By the induction hypothesis, we also have that 
the passage time to $(j-1,2)$ is the same for the $\omega$ and $\tw$ weights. Thus, 
the optimal path passing through $(j-1,2)$ achieves a $\tw$-weight equal to the optimal $\omega$-weight. 

On the other hand, the path passing through $(j,1)$ has $\tw$-weight no larger
than its $\omega$-weight (because all points in row $1$ have $\tw$-weight no larger than their $\omega$-weight). This gives $\Lp_{(i,1),(j,2)}[\tw]=\Lp_{(i,1),(j,2)}[\omega]$.

\medskip\noindent\textbf{Case 3b}: $j>x_0$, $U_j>0$.

As in Case 2, the optimal $\omega$-path goes through $(j,1)$. As in Case 3a, the 
$\tw$-weight of this path is no greater (in fact, strictly smaller) than its $\omega$-weight. Hence, it remains to consider the optimal path which passes through $(j-1,2)$. Using the induction hypothesis, its 
$\tw$-weight is
\begin{align} \label{eq:Wwstar}
W\coloneqq
\tw_{(j,2)} + \Lp_{(i,1), (j-1,2)}[\tw]
=
\omega_{(j,2)} + U_j + \Lp_{(i,1), (j-1,2)}[\omega].
\end{align}
On the other hand,  
\begin{equation} \label{eq:Lwdif}
\Lp_{(i,1), (j,1)}[\omega]
-
\Lp_{(i,1), (j-1,2)}[\omega]
=
\inf_{i\leq s < j} \sum_{r=s}^{j-1} c_s = U_j,
\end{equation}
where the last step follows by Lemma \ref{lemma:inf-attained}. 
By combining Equations \eqref{eq:Wwstar} and \eqref{eq:Lwdif}, we have $W=\Lp_{(i,1), (j,1)}[\omega)]+ \omega_{(j,2)}$, so that the optimal 
$\omega$-weight and optimal $\tw$-weight coincide, as desired. 
\end{proof}

\subsection{Queueing interpretation and proof of Proposition \ref{prop:rates-exchanged}}
We will interpret the definitions in (\ref{Udef}) and (\ref{wtilde-def})
in terms of the operation of a queue in discrete time. 

Write $S_i=\omega_{(i,1)}$ and $A_i=\omega_{(i,2)}$. We consider a queue in
which, at time-step $i$, first an amount $S_i$ of service is available, 
and then an amount $A_i$ of work arrives. Define
\begin{equation}\label{Qdef}
Q_t:=\sup_{s<t}\left(\sum_{i=s}^{t-1}A_i - \sum_{i=s+1}^{t-1} S_i\right)=\sup_{s<t}\left(\sum_{i=s}^{t-1}\omega_{(i,2)} - \sum_{i=s+1}^{t-1} \omega_{(1,i)}\right).
\end{equation}
Since $b_2>b_1$, and (\ref{constant-a})
holds, the strong law of large numbers gives that this sup is finite with probability $1$.
Then, $Q_t$ has the interpretation as the queue length just before time step $t$, i.e.\ after the arrival $A_{t-1}$ and before the service $S_t$. From the definition (\ref{Qdef}) we get the following recursion:
\begin{equation}\label{Qrecursion}
Q_{t+1}=\left(Q_t-S_t\right)_+ + A_t.
\end{equation}

Finally, we define $D_t=\min(S_t, Q_t)$, with the interpretation that
$D_t$ is the amount of departure from the queue at time $t$. We can then observe that $U_t$ defined at (\ref{Udef}) satisfies $U_t=S_t-D_t=\left(S_t-Q_t\right)_+$, with the interpretation that $U_t$ is the amount of 
unused service at time $t$. From this and (\ref{wtilde-def}) we can then identify
\begin{equation} \label{eq:DAU}
\tw_{(t,1)}=D_t, \quad\text{and}\quad \tw_{(t,2)}=A_t + U_t.
\end{equation}
Since $(S_t, A_t)$ are independent for different $t$, and since 
$Q_t$ is a function of $(S_i, A_i, i<t)$, recursion (\ref{Qrecursion})
shows that the process $Q = (Q_t)_{t\in\bbZ}$ is a Markov chain. Although the chain
is inhomogeneous in time, we see in the following lemma that the dynamics at different times all share the same stationary distribution.
Part~\eqref{lemma:QBurke_c} gives a version of Burke's theorem for the inhomogeneous queue.

\begin{lem}\label{lemma:QBurke}
\leavevmode
\begin{enumerate}[\normalfont (a)]
\item \label{lemma:QBurke_a} 
$Q_t\sim \Exp\{b_2-b_1\}$ for all $t$.

\item \label{lemma:QBurke_b} For all $t$, 
\[
(Q_t, S_t, A_t) \deq (Q_{t+1}, A_t+U_t, D_t).
\]

\item \label{lemma:QBurke_c}
For each $t$, $A_t+U_t\sim\Exp\{a_t+b_1\}$
and $D_t\sim\Exp\{a_t+b_2\}$, and all these
quantities are independent (also across $t$).
\end{enumerate}
\end{lem}
This Burke property can already be deduced
from results given in 
\cite[Proposition 4.7]{Emra_Janj_Sepp_25} in 
the last-passage-percolation setting.
In the queueing context, 
``extended’’ Burke’s theorems, involving more quantities than just arrival and departure, have been proved for a variety of models --
see for example \cite{OCon_Yor_02, KOR2002}, and \cite{Martinbatch, DMO} 
for further results and relevant discussion. 
Such results have generally been written for queues which are homogeneous in time -- however, 
even in an inhomogeneous context, the case of a single queue with exponential data has a very straightforward proof, which we include here for the purpose of a short and self-contained argument. 

\begin{proof}[Proof of Lemma \ref{lemma:QBurke}]
We consider a single time-step of the queue, 
starting with queue-length $Q$, receiving service $S$ and then arrival $A$. 
Suppose $Q$, $S$ and $A$ are independent with $Q\sim\Exp\{\theta-\kappa\}$, 
$S\sim\Exp\{\kappa\}$, and 
$A\sim\Exp\{\theta\}$ (where $\theta>\kappa>0$). 
The density of $(Q, S, A)$ is then
\[
f_{Q, S, A}(q,s,a)=
C\exp[-\left((\theta-\kappa)q+\kappa s+\theta a\right)],
\]
where $C$ is the constant $(\theta-\kappa)\theta\kappa$. 

In line with our definitions above, 
let $D=\min(Q, S)$, $U=S-D$, 
and $\tQ=Q-D+A$, and write also $Z=A+U$. 

\begin{figure}[hbt]
\begin{center}
\includegraphics[width=0.7\textwidth]{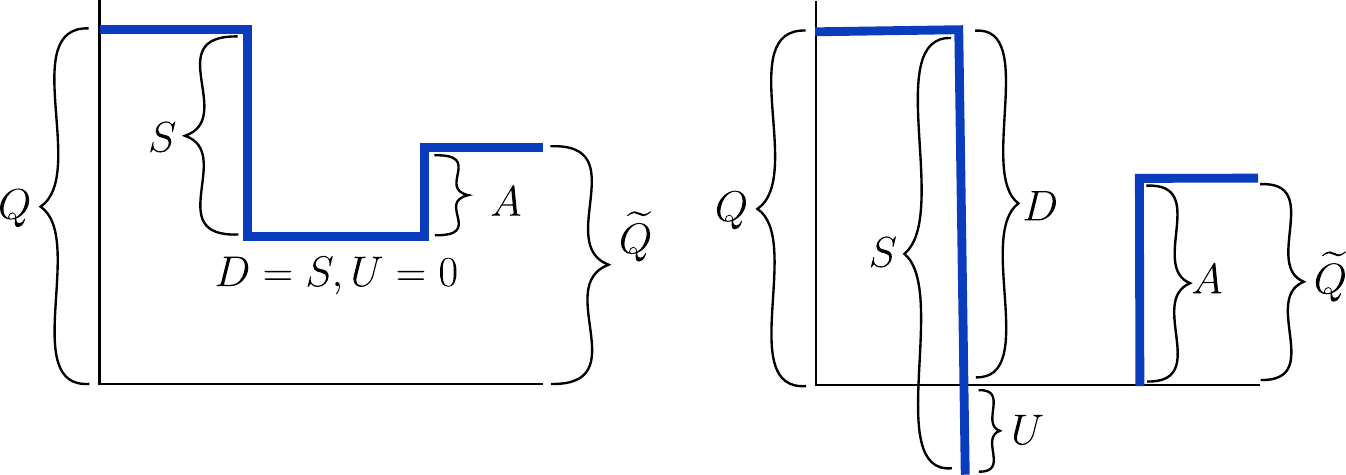}
\caption{ \label{fig:queue-step} \small
Illustration of two cases of a single time-step of the queue. 
The initial queue-length is $Q$. Service $S$ is available, followed by an arrival $A$. The queue-length at the end of the step is $\tQ$. On the left, the case $S\leq Q$, where $U=0$ and $D=S$. On the right, the case $S>Q$, where $U>0$, $D=Q$ and $\tQ=A$. 
}
\end{center}
\end{figure}

The left side of 
Figure \ref{fig:queue-step} shows the case $S\leq Q$,
i.e.\ $U=0$, 
in which case 
\[
Z=A, \,\,\, D=S, \,\,\, \tQ=Q-S+A.
\]
The right side of 
Figure \ref{fig:queue-step}
shows the alternative case $S>Q$, i.e.\ $U>0$, in which case
\[
Z=A+S-Q, \,\,\, D=Q, \,\,\, \tQ=A.
\]
In each of the two cases, the transformation from $(Q, S, A)$ to $(Q, Z, D)$ is a linear map whose Jacobian is $1$. 
In the first case, we then get the density 
\begin{align}
\nonumber
f_{\tQ, Z, D}(\tq, z, d)
&=
c\exp[
-\left(\theta-\kappa)(\tq+d-z)
+\kappa d
+\theta z\right)]
\\
\label{density}
&=
c\exp[-\left((\theta-\kappa)\tq+\kappa z+\theta d\right)].
\end{align}
Similarly in the second case,
\[
f_{\tQ, Z, D}(\tq, z, d)
=
c\exp[
-\left(\theta-\kappa)a
+\kappa(z-\tq+d)
+\theta \tq\right)],
\]
which again simplifies to (\ref{density}).
Hence
\[
\tQ\sim\Exp\{\theta-\kappa\},\,\, 
Z\sim\Exp\{\kappa\},\,\, 
D\sim\Exp\{\theta\},
\]
independently.

Now we apply the above in the case $\theta=a_t+b_2$, $\kappa=a_t+b_1$, corresponding to the operation of the queue at time $t$. First, noticing that $Q$ and $\tQ$ have the same distribution, namely $\Exp\{b_2-b_1\}$, we see that this distribution is invariant for every step. From (\ref{constant-a}), for $t$ below some $t_0$, we have that the queue-length process forms a stationary and time-homogeneous Markov chain. It is easy to show that the invariant distribution $\Exp\{b_2-b_1\}$
is unique. This gives $Q_t\sim\Exp\{b_2-b_1\}$
for all sufficiently small $t$, and by induction this extends to all $t$. 

This gives part \eqref{lemma:QBurke_a}, and then part \eqref{lemma:QBurke_b} immediately follows from the form of the density in (\ref{density}). We already have the distributions of $A_t+U_t$ and of $D_t$ in part \eqref{lemma:QBurke_c}, and the independence between them. It remains to show the independence as $t$ varies. 
Note the following facts:
\begin{itemize}
\item[(1)]
$Q_{t+1}$ is independent of $(A_t+U_t, D_t)$
(again from part \eqref{lemma:QBurke_b}). 
\item[(2)] $(S_j, A_j, j>t)$ is independent of $Q_{t+1}$ and $(A_t+U_t, D_t)$ (because those are functions of $(S_i, A_i, i\leq t)$, and the $(S_i, A_i, i\in \ZZ)$ are an independent sequence).
\item[(3)] $(A_j+U_j, D_j, j>t)$ are functions of $Q_{t+1}$ and $(S_j, A_j, j>t)$. 
\end{itemize}
Combining, we get that $(A_t+U_t, D_t)$ is independent of $(A_j+U_j, D_j, j>t)$. Since this holds for all $t$, we have that in fact $(A_i+U_i, D_i, i\in\ZZ)$ are an independent sequence, as required.
\end{proof}

\begin{proof}[Proof of Proposition \ref{prop:rates-exchanged}]
This follows immediately from Lemma \ref{lemma:QBurke} and \eqref{eq:DAU}.  
\end{proof}

\section{Inhomogeneous Busemann functions}
\label{S:iBusFn}

Another group of results that will prominently enter the proof of Theorem \ref{T:BusMar} concerns the Busemann functions of the inhomogeneous exponential LPP. These objects have been recently introduced and studied in \cite{Emra_Janj_Sepp_25}. We now recall from that work  
a few properties of the inhomogeneous Busemann functions. 

Consider two real parameter sequences $\bfa = (a_i)_{i \in \bbZ_{>0}}$ and $\bfb = (b_j)_{j \in \bbZ_{>0}}$ subject to 
\begin{align}
\label{A:Inf}
\inf \bfa + \inf \bfb > 0. 
\end{align}
Let $\omega^{\bfa, \bfb} = \{\omega^{\bfa, \bfb}_\bfv: \bfv \in \Z^2_{>0}\}$ be independent weights with the marginals as in \eqref{E:iw}. Thus, $\Lp[\omega^{\bfa, \bfb}]$ is now an inhomogeneous LPP process on the quadrant $\bbZ_{>0}^2$. As required for the results to be imported from \cite{Emra_Janj_Sepp_25}, assume further that 
\begin{align}
\label{A:VagConv}
\f{1}{n}\sum_{i = 1}^n \delta_{a_i} \to \alpha \quad\text{ and } \quad \f{1}{n}\sum_{j = 1}^n \delta_{b_j} \to \beta \quad \text{ in the vague topology }
\end{align}
for some nonzero subprobability measures $\alpha$ and $\beta$ on $\R$. Here, vague convergence means the convergence of integrals against continuous functions on $\bbR$ that vanish at infinity. This differs from weak convergence where the integration is instead against bounded continuous functions on $\bbR$. Assumption \eqref{A:VagConv} is natural because it necessarily holds along a subsequence, although the limits may be zero measures. Thus, the real restriction in \eqref{A:VagConv} is the nontriviality of the limits. Due to potential loss of mass, $\alpha$ and $\beta$ need not be probability measures. The special case where $\alpha$ and $\beta$ are probabilities is equivalent to the tightness of the prelimit empirical measures. Then the vague convergence in \eqref{A:VagConv} coincides with weak convergence.  

The Busemann functions associated with the $\omega^{\bfa, \bfb}$-weights come in several types depending on the limit direction. We first discuss \emph{thin} Busemann functions arising from limits along a fixed column or row. Working with arbitrary real weights $\w = \{\w_\bfp: \bfp \in \bbZ_{>0}^2\}$ for the moment, define 
\begin{align}
\label{E:ThinBuse}
\begin{split}
\sI_{(i,j)}^{k, \uparrow}[\w] &= \sup_{n \in \Z_{\ge j}} \ul{\I}_{(i,j), (k, n)}[\w] = \lim_{n \to \infty} \ul{\I}_{(i,j), (k, n)}[\w], \\ 
\sJ_{(i,j)}^{k, \uparrow}[\w] &= \inf_{n \in \Z_{\ge j}} \ul{\J}_{(i,j), (k, n)}[\w] = \lim_{n \to \infty} \ul{\J}_{(i,j), (k, n)}[\w], \\
\sI_{(i,j)}^{\ell, \rightarrow}[\w] &= \inf_{m \in \Z_{\ge i}} \ul{\I}_{(i,j), (m, \ell)}[\w] = \lim_{m \to \infty} \ul{\I}_{(i,j), (m, \ell)}[\w], \\
\sJ_{(i,j)}^{\ell, \rightarrow}[\w] &= \sup_{m \in \Z_{\ge i}} \ul{\J}_{(i,j), (m, \ell)}[\w] = \lim_{m \to \infty} \ul{\J}_{(i,j), (m, \ell)}[\w]
\end{split}
\end{align}
for $i, j, k, \ell \in \Z_{>0}$ with $i \le k$ and $j \le \ell$. See Figure \ref{F:ThinBuse}. The second equality in each line of \eqref{E:ThinBuse} is a consequence of 
increment monotonicity due to Lemma \ref{L:Comp}. 
\begin{figure}[h]
\centering
\begin{tikzpicture}[scale = 1]
\draw[->](0, 0)--(4, 0);
\draw[->](0, 0)--(0, 4);
\draw(0, 0)node[below]{$1$};
\filldraw[color=black](0, 0)circle(0.10);
\draw(0.5, 0)node[below]{$2$};
\filldraw[color=black](0.5, 0)circle(0.10);
\draw(2, 0)node[below]{$k$};
\draw[->](1.9,3.1)--(1.9, 4);
\draw(2, 3)node[right]{$(k, n)$};
\filldraw[color=black](2, 3)circle(0.10);
\draw(0, 1.5)node[left]{$\ell$};
\draw(3.5, 1.5)node[above]{$(m, \ell)$};
\draw[->](3.6, 1.4)--(4, 1.4);
\filldraw[color=black](3.5, 1.5)circle(0.10);
\draw[-](2, 0)--(2, 4);
\draw[-](0, 1.5)--(4, 1.5);
\draw[->, color=red, thick](0, 0)--(2-0.1, 3-0.1);
\draw[->, color=red, thick](0.5, 0)--(2-0.1, 3-0.1);
\draw[->, color=blue, thick](0, 0)--(3.5-0.1, 1.5-0.1);
\draw[->, color=blue, thick](0.5, 0)--(3.5-0.1, 1.5-0.1);
\end{tikzpicture}
\caption{\small Illustrates the thin Busemann functions $\sI_{(1, 1)}^{k, \uparrow}[\w]$ and $\sI_{(1, 1)}^{\ell, \rightarrow}[\w]$. The last-passage times $\Lp_{(1, 1), (k, n)}[\w]$ and $\Lp_{(2, 1), (k, n)}[\w]$ are indicated with red arrows. The last-passage times $\Lp_{(1, 1), (m, \ell)}[\w]$ and $\Lp_{(2, 1), (m, \ell)}[\w]$ are indicated with blue arrows.}
\label{F:ThinBuse}
\end{figure}
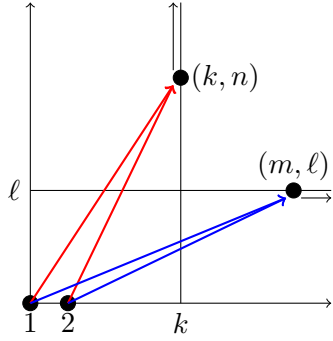

The next proposition describes the marginals of the thin Busemann functions for the $\omega^{\bfa, \bfb}$-weights. In the statement, $\bfa_{i:k}$ and $\bfb_{j:\ell}$ 
denote the sequences $a_{i}, \dotsc, a_{k}$ and $b_j, \dotsc, b_\ell$, respectively. The same notation will be used later also with $k = \ell = \infty$. 
\begin{prop}[{\cite[Theorem 3.1(l)]{Emra_Janj_Sepp_25}}]
\label{P:ThinBuse}
Assume \eqref{A:Inf} and \eqref{A:VagConv}. For $i, j, k, \ell \in \Z_{>0}$ with $i \le k$ and $j \le \ell$, 
\begin{align*}
\sI_{(i, j)}^{k, \uparrow}[\omega^{\bfa, \bfb}] &\sim \Exp\{a_i-\min \bfa_{i:k}\},  \\ 
\sJ_{(i, j)}^{k, \uparrow}[\omega^{\bfa, \bfb}] &\sim \Exp\{b_j + \min \bfa_{i:k}\}, \\
\sI_{(i, j)}^{\ell, \rightarrow}[\omega^{\bfa, \bfb}] &\sim \Exp\{a_i+\min \bfb_{j:\ell}\}, \\
\sJ_{(i, j)}^{\ell, \rightarrow}[\omega^{\bfa, \bfb}] &\sim \Exp\{b_j-\min \bfb_{j:\ell}\}. 
\end{align*}
\end{prop}

Proposition \ref{P:ThinBuse} shows that the distribution of $\sI_{(i, j)}^{k, \uparrow}[\omega^{\bfa, \bfb}]$, for example, depends on $k$ only through $\min \bfa_{i:k}$. Since also $\sI_{(i, j)}^{k, \uparrow}$ is nonincreasing in $k$ (due to Lemma \ref{L:Comp} again), one finds that $\sI_{(i, j)}^{k, \uparrow}[\omega^{\bfa, \bfb}]$ decreases in $k$ exactly when $\min \bfa_{i:k}$ changes. More precisely, writing
\begin{align}
\rec^{\bfa}(i:k) = \min\{k' \in \Z_{\ge i}: \min \bfa_{i:k'} = \min \bfa_{i:k}\}, \label{E:Rec}
\end{align}
for the minimal index where $\min \bfa_{i, k}$ is attained, one has the following corollary. 
\begin{cor}[{\cite[Theorem 3.1(j)]{Emra_Janj_Sepp_25}}]
\label{C:ThinBuse}
Assume \eqref{A:Inf} and \eqref{A:VagConv}. For $i, j, k, \ell \in \Z_{>0}$ with $i \le k$ and $j \le \ell$, a.s., 
\begin{align*}
\sI_{(i, j)}^{k, \uparrow}[\omega^{\bfa, \bfb}] &= \sI_{(i, j)}^{\rec^{\bfa}(i: k), \uparrow}[\omega^{\bfa, \bfb}], \\
\sJ_{(i, j)}^{k, \uparrow}[\omega^{\bfa, \bfb}] &= \sJ_{(i, j)}^{\rec^{\bfa}(i: k), \uparrow}[\omega^{\bfa, \bfb}], \\
\sI_{(i, j)}^{\ell, \rightarrow}[\omega^{\bfa, \bfb}] &= \sI_{(i, j)}^{\rec^{\bfb}(j:\ell), \rightarrow}[\omega^{\bfa, \bfb}], \\
\sJ_{(i, j)}^{\ell, \rightarrow}[\omega^{\bfa, \bfb}] &= \sJ_{(i, j)}^{\rec^{\bfb}(j:\ell), \rightarrow}[\omega^{\bfa, \bfb}]. 
\end{align*}
\end{cor}
For example, the first equality above reveals the useful fact that $\Bh_{(i, j)}^{k, \uparrow}[\omega^{\bfa, \bfb}]$ depends on the $\omega^{\bfa, \bfb}$-weights only through columns $[\rec^{\bfa}(i:k)] \smallsetminus [i-1]$; in particular, there is no dependence on columns $[k] \smallsetminus [\rec^{\bfa}(i:k)]$ contrary to what definition \eqref{E:ThinBuse} initially suggests. It is also worth noting that the thin Busemann functions are trivial for the homogeneous exponential LPP. As can be seen from Corollary \ref{C:ThinBuse}, a.s.,   
\begin{align}
\label{E:ThinBuseHom}
\begin{split}
&\sI_{(i, j)}^{k, \uparrow}[\omega] = \sI_{(i, j)}^{i, \uparrow}[\omega] = \infty = \sJ_{(i, j)}^{j, \rightarrow}[\omega] = \sJ_{(i, j)}^{\ell, \rightarrow}[\omega], \\
&\sJ_{(i, j)}^{k, \uparrow}[\omega] = \sJ_{(i, j)}^{i, \uparrow}[\omega] = \omega_{(i, j)} =  \sI_{(i, j)}^{j, \rightarrow}[\omega] = \sI_{(i, j)}^{\ell, \rightarrow}[\omega]. 
\end{split}
\end{align}
Hence, this notion carries meaningful content only in the case of inhomogeneity. 

One part of our exposition in Section \ref{Ss:BMdiscuss} is streamlined with an appeal to the independence of each of the collections 
\begin{align}
\label{E:TBuseInd}
\{\sI_{(i, j)}^{k, \uparrow}[\omega^{\bfa, \bfb}]: i \in [k-1]\} \quad \text{ and } \quad \{\sJ_{(i, j)}^{\ell, \rightarrow}[\omega^{\bfa, \bfb}]: j \in [\ell-1]\}. 
\end{align}
These properties are contained in the following proposition. As previously noted, this independence property is not needed in the proof of Theorem \ref{T:BusMar}. 

\begin{prop}[{\cite[Theorem 3.1(m)]{Emra_Janj_Sepp_25}}]
\label{P:TBuseInd}    
Assume \eqref{A:Inf} and \eqref{A:VagConv}. Let $i, j, k, \ell \in \Z_{>0}$ with $i \le k$ and $j \le \ell$, and $\nu$ be a down-right path that contains $(i, \ell)$ and $(k, j)$. Then the following statements hold. 
\begin{enumerate}[\normalfont (a)]
\item 
If $\rec^\bfa(i:k) = k$ then the collection 
\begin{align*}
\{\sI_{\bfv}^{k, \uparrow}[\omega^{\bfa, \bfb}]: \bfv, \bfv + (1, 0) \in \nu\} \cup \{\sJ_{\bfv}^{k, \uparrow}[\omega^{\bfa, \bfb}]: \bfv, \bfv + (0, 1) \in \nu\}    
\end{align*} is independent. 
\item 
If $\rec^\bfb(j:\ell) = \ell$ then the collection 
\begin{align*}
\{\sI_{\bfv}^{\ell, \rightarrow}[\omega^{\bfa, \bfb}]: \bfv, \bfv + (1, 0) \in \nu\} \cup \{\sJ_{\bfv}^{\ell, \rightarrow}[\omega^{\bfa, \bfb}]: \bfv, \bfv + (0, 1) \in \nu\}    
\end{align*} is independent. 
\end{enumerate}
\end{prop}

We next move on to the Busemann functions corresponding to limit directions strictly into the quadrant. To describe the marginals, one needs to first introduce the appropriate generalization of the function $\zeta$ given by \eqref{E:zeta}. Recalling the limit measures $\alpha$ and $\beta$ from \eqref{A:VagConv}, define 
\begin{align}
\label{E:ABInt}
\A^\alpha(z) = \int_\bbR \frac{\alpha(\dd a)}{a+z} \quad \text{ for } z > -\ul{\alpha} \qquad \text{ and } \qquad \B^\beta(z) = \int_\bbR \frac{\beta(\dd b)}{b-z} \quad \text{ for } z < \ul{\beta},
\end{align}
where $\ul{\mu}$ denotes the infimum of the support of a Borel measure $\mu$ on $\bbR$.  The functions $\A^\alpha$ and $\B^\beta$ are both analytic and their derivatives can be computed via differentiation under the integral sign. Assumption \eqref{A:VagConv} implies that $\ul{\alpha} \ge \inf \bfa_{i:\infty}$ and $\ul{\beta} \ge \inf \bfb_{j:\infty}$ for each $i, j \in \bbZ_{>0}$. Consequently and by \eqref{A:Inf}, the interval $(-\ul{\alpha}, \ul{\beta})$ is nonempty. Thus, one can define 
\begin{align}
\label{E:rho}
\rho^{\alpha, \beta}(z) = -\frac{(\B^\beta)'(z)}{(\A^{\alpha})'(z)} = \int_\bbR \frac{\beta(\dd b)}{(b-z)^2}  \cdot \bigg(\int_\bbR \frac{\alpha(\dd a)}{(a+z)^2}\bigg)^{-1}\quad \text{ for } z \in (-\underline{\alpha}, \underline{\beta}). 
\end{align}
The function $\rho^{\alpha, \beta}$ is continuous and (strictly) increasing, and maps the interval \\ $(-\inf \bfa_{i:\infty}, \inf \bfb_{j:\infty}) \subset (-\underline{\alpha}, \underline{\beta})$ onto the interval $(\mfc^{\ver}_i, \mfc^{\hor}_j) \subset (0, \infty)$ with the endpoints 
\begin{align}
\label{E:crit}
\begin{split}
\mfc^\ver_i &= \mfc^{\ver, \alpha, \beta, \inf \bfa_{i:\infty}} = \rho^{\alpha, \beta}(-\inf \bfa_{i:\infty}) = \int \frac{\beta(\dd b)}{(b+\inf \bfa_{i:\infty})^2}  \cdot \bigg(\int \frac{\alpha(\dd a)}{(a-\inf \bfa_{i:\infty})^2}\bigg)^{-1}, \\
\mfc^{\hor}_j &= \mfc^{\hor, \alpha, \beta, \inf \bfb_{j:\infty}} = \rho^{\alpha, \beta}(\inf \bfb_{j:\infty}) = \int \frac{\beta(\dd b)}{(b-\inf \bfb_{j:\infty})^2}  \cdot \bigg(\int \frac{\alpha(\dd a)}{(a+\inf \bfb_{j:\infty})^2}\bigg)^{-1}. 
\end{split}
\end{align}
Note in particular that 
\begin{align}
\label{E:crit-2}
\begin{split}
\mfc^{\ver}_i = 0 \quad \text{ if and only if } \quad \int_\bbR \frac{\alpha(\dd a)}{(a-\inf \bfa_{i:\infty})^2} = \infty, \\ 
\mfc^{\hor}_j = \infty \quad \text{ if and only if } \quad \int_\bbR \frac{\beta(\dd b)}{(b-\inf \bfb_{j:\infty})^2} = \infty. 
\end{split}
\end{align}
For $r \in [0, \infty]$, define  
\begin{align}
\label{E:ShpMin}
\zeta^{\bfa, \bfb}_{(i, j)}(r) = 
\begin{cases}
-\inf \bfa_{i:\infty} \quad &\text{ if }r \le \mfc^{\ver}_i \\ 
(\rho^{\alpha, \beta})^{-1}(r) \quad &\text{ if } \mfc^{\ver}_i < r < \mfc^{\hor}_j \\ 
\inf \bfb_{j:\infty} \quad &\text{ if } r \ge \mfc^{\hor}_j, 
\end{cases}
\end{align}
which recovers \eqref{E:zeta} in the homogeneous case $\bfa \equiv 0$ and $\bfb \equiv 1$. 

The function $\zeta^{\bfa, \bfb}_{(i, j)}$ arises as the unique minimizer in a natural variational description of the shape function for the inhomogeneous LPP process $\Lp[\omega^{\bfa, \bfb}]$, with the initial point fixed at $(i, j)$. The shape function encodes the law of large numbers limit as the terminal point tends to infinity in a given direction. This includes the axis directions provided that both terminal coordinates become unbounded. If one of the coordinates is bounded, the corresponding limit is given instead by Proposition \ref{P:ThinLLN} ahead. The limits there agree with the boundary values of the shape function when $\inf \bfa_{i:\infty} = \min \bfa_{i:m}$ and $\inf \bfb_{j:\infty} = \min \bfb_{j:n}$. Note the appearance of $\zeta^{\bfa, \bfb}$ in Proposition \ref{P:ThinLLN} in these special cases. For further discussion, we refer to \cite[Section 3.3]{Emra_Janj_Sepp_21} and \cite[Section 2.4]{Emra_Janj_Sepp_25}. 

With \eqref{E:ShpMin} in place, the inhomogeneous extension of Proposition \ref{P:BusFn}\eqref{P:BusFn_a}--\eqref{P:BusFn_b} can be stated as follows. 
\begin{prop}[{\cite[Theorem 3.1(g), (l)]{Emra_Janj_Sepp_25}}]
\label{P:iBusFn}
Assume \eqref{A:Inf} and \eqref{A:VagConv}. Fix $r  \in [0, \infty]$. The following statements hold for each $\bfu = (i, j) \in \bbZ_{>0}^2$. 
\begin{enumerate}[\normalfont (a)]
\item There exist random reals $\sI_{\bfu}^r[\omega^{\bfa, \bfb}]$ and $\sJ_{\bfu}^r[\omega^{\bfa, \bfb}]$ such that, a.s.,  
\begin{align*}
\lim_{k \to \infty} \ul{\I}_{\bfu, \bfv_k}[\omega^{\bfa, \bfb}] = \sI^r_\bfu[\omega^{\bfa, \bfb}] \quad \text{ and } \quad \lim_{k \to \infty} \ul{\J}_{\bfu, \bfv_k}[\omega^{\bfa, \bfb}] = \sJ^r_\bfu[\omega^{\bfa, \bfb}]
\end{align*}
for any $\bfv_k = (m_k, n_k) \in \bbZ_{>0}^2$ for $k \in \bbZ_{>0}$ with $\min \{m_k, n_k\} \to \infty$ and $\dfrac{m_k}{n_k} \to r$ as $k \to \infty$. 
\item $\sI_{\bfu}^r[\omega^{\bfa, \bfb}] \sim \Exp\{a_i + \zeta_{\bfu}^{\bfa, \bfb}(r)\}$ and $\sJ_{\bfu}^r[\omega^{\bfa, \bfb}] \sim \Exp\{b_j-\zeta^{\bfa, \bfb}_\bfu(r)\}$. 
\end{enumerate}
\end{prop}

The next proposition records that the Busemann functions $\sI_{\bfu}^r$ and $\sJ_{\bfu}^r$ can be obtained as the limits of thin Busemann functions for each $\bfu = (i, j) \in \bbZ_{>0}^2$ and $r \in [0, \mfc_i^{\ver}] \cup [\mfc_j^{\hor}, \infty]$. As a particularly useful consequence for our argument, $\sI_{\bfu}^r$ and $\sJ_{\bfu}^r$ are both constant in $r$ in each of the intervals $[0, \mfc_i^{\ver}]$ and $[\mfc_j^{\hor}, \infty]$. 

\begin{prop}[{\cite[Theorem 3.1(i), (k)]{Emra_Janj_Sepp_25}}]
\label{P:iBusFlat}
Assume \eqref{A:Inf} and \eqref{A:VagConv}. The following statements hold for each $\bfu = (i, j) \in \bbZ_{>0}^2$ a.s.\  
\begin{enumerate}[\normalfont (a)]
\item For $r \in [0, \mfc_i^\ver]$, 
\begin{align*}
\sI_{\bfu}^{r}[\omega^{\bfa, \bfb}] &= \inf_{k \in \Z_{\ge i}} \sI_{\bfu}^{k, \uparrow}[\omega^{\bfa, \bfb}] = \lim_{k \to \infty} \sI_{\bfu}^{k, \uparrow}[\omega^{\bfa, \bfb}], \\ 
\sJ_{\bfu}^{r}[\omega^{\bfa, \bfb}] &= \sup_{k \in \Z_{\ge i}} \sJ_{\bfu}^{k, \uparrow}[\omega^{\bfa, \bfb}] = \lim_{k \to \infty} \sJ_{\bfu}^{k, \uparrow}[\omega^{\bfa, \bfb}]. 
\end{align*}
\item For $r \in [\mfc_j^\hor, \infty]$, 
\begin{align*}
\sI_{\bfu}^{r}[\omega^{\bfa, \bfb}] &= \sup_{\ell \in \Z_{\ge j}} \sI_{\bfu}^{\ell, \rightarrow}[\omega^{\bfa, \bfb}] = \lim_{\ell \to \infty} \sI_{\bfu}^{\ell, \rightarrow}[\omega^{\bfa, \bfb}], \\ 
\sJ_{\bfu}^{r}[\omega^{\bfa, \bfb}] &= \inf_{\ell \in \Z_{\ge j}} \sJ_{\bfu}^{\ell, \rightarrow}[\omega^{\bfa, \bfb}] = \lim_{\ell \to \infty} \sJ_{\bfu}^{\ell, \rightarrow}[\omega^{\bfa, \bfb}].
\end{align*}
\end{enumerate}
\end{prop}

\section{Proof of Theorem \ref{T:BusMar}}
\label{S:PfBuse}
We now begin to work towards Theorem \ref{T:BusMar}. Recall the sets $\rightset_{k,\ell}$ and $\upset_{k,\ell}$ from \eqref{E:rightset} and \eqref{E:upset}, and recall the definition of the point $\bfz_p = (k+p-1, \ell+d-p)$. Let $\bigl(x_{\bfu,p},y_{\bfv,p}: \bfu \in \rightset_{k,\ell},\bfv \in \upset_{k,\ell},p \in [d]\bigr)$ be an arbitrary collection of real numbers.
Let $\bfP$ denote the probability measure on a common sample space for the weights $\omega$ and $\eta$ considered below, as well as an additional i.i.d.\ sequence $\bfc = (c_i)_{i \in \Z_{>0}}$ of $\Exp\{1\}$ random variables. For convenience, we will take the $\omega$, $\eta$, and $\bfc$ weights to be independent. 
To obtain the theorem, it suffices to establish the identity
\begin{align}
\label{E:CDFId}
\begin{split}
&\bfP\{\sI^{r_p}_{\bfu}[\omega] > x_{\bfu, p},\; \sJ^{r_p}_{\bfv}[\omega] < y_{\bfv, p} \text{ for } \bfu \in \rightset_{k,\ell},\bfv \in \upset_{k,\ell}, p \in [d]\} \\
&=\bfP\{\ul{\I}_{\bfu, \bfz_p}[\eta] > x_{\bfu, p},\;  \ul{\J}_{\bfv, \bfz_p}[\eta] < y_{\bfv, p}  \text{ for } \bfu \in \rightset_{k,\ell},\bfv \in \upset_{k,\ell}, p \in [d]\}. 
\end{split}
\end{align}

Recalling that $\bfr = (r_p)_{p \in [d]} \in \bbR_{>0}^d$ denotes an increasing sequence of directions,  
pick $\bfk_n = (\bfk_n(p))_{p \in [d]} \in \bbZ_{\ge k}^d$ for $n \in \bbZ_{\ge \ell}$ such that $\bfk_n(p)/n \to r_p$ for $p \in [d]$ as $n \to \infty$. Analogously, for $m \in \bbZ_{\ge k}$, select a sequence $\bfs_m = (\bfs_m(p))_{p \in [d]} \in \bbZ^d_{\ge \ell}$ such that $\bfs_m(p)/m \to 1/r_p$ for $p \in [d]$ as $m \to \infty$. We will also make the choice 
\begin{align}
\label{E:kl0}
\begin{split}
\bfk_0 = (\bfk_0(p))_{p \in [d]} = (k+p-1)_{p \in [d]} \quad \text{ and } \quad \bfs_0 = (\bfs_0(p))_{p \in [d]} = (\ell+d-p)_{p \in [d]}. 
\end{split}
\end{align}
Proposition \ref{P:BusFn}\eqref{P:BusFn_a} implies that, almost surely, for $\bfu \in \rightset_{k,\ell},\bfv \in \upset_{k,\ell}$, and $p \in [d]$,  
\begin{align}
\label{E:BusLim}
\begin{split}
\sI_{\bfu}^{r_p}[\omega] &= \lim_{n \to \infty} \ul{\I}_{\bfu, (\bfk_n(p), n)}[\omega]= \lim_{m \to \infty} \ul{\I}_{\bfu, (m, \bfs_m(p))}[\omega],\quad\text{and}\\
\sJ_{\bfv}^{r_p}[\omega] &= \lim_{n \to \infty} \ul{\J}_{\bfv, (\bfk_n(p), n)}[\omega] = \lim_{m \to \infty} \ul{\J}_{\bfv, (m, \bfs_m(p))}[\omega].  
\end{split}
\end{align}
Consequently, since the limiting marginal distributions below have continuous distributions, the left-hand side of \eqref{E:CDFId} can be written as the following two limits: 
\begin{align}
\label{E:CDFLim}
\begin{split}
&\bfP\{\sI^{r_p}_{\bfu}[\omega] > x_{\bfu, p},\; \sJ^{r_p}_{\bfv}[\omega] < y_{\bfv, p}, \quad  \forall \;\bfu,\bfv,p\} \\
&= \lim_{n \to \infty} \bfP\{\ul{\I}_{\bfu, (\bfk_n(p), n)}[\omega] > x_{\bfu, p},\;\ul{\J}_{\bfv, (\bfk_n(p), n)}[\omega] < y_{\bfv, p}, \quad  \forall \; \bfu,\bfv,p \} \\ 
&= \lim_{m \to \infty} \bfP\{\ul{\I}_{\bfu, (m, \bfs_m(p))}[\omega] > x_{\bfu p},\; \ul{\J}_{\bfv, (m, \bfs_m(p))}[\omega] < y_{\bfv, p}, \quad  \forall \; \bfu,\bfv,p\}. 
\end{split}
\end{align}
 In various limits ahead, the continuity of the limiting marginals will be used again without explicit mention. In each of the events  above, we have written $\forall \; \bfu,\bfv,p$, which we will henceforth use as shorthand for for all $\bfu \in \rightset_{k,\ell}, \bfv \in \upset_{k,\ell},p \in [d]$.

 Since $\bfr$ is increasing, one may assume that there exists a fixed gap $L > 0$ such that, for each $m \in \bbZ_{\ge k}$, $n \in \bbZ_{\ge \ell}$ and $p \in [d-1]$,  
\begin{align}
\label{E:Gap}
\bfk_n(p+1) - \bfk_n(p) \ge L \quad \text{ and } \quad \bfs_m(p)-\bfs_m(p+1) \ge L. 
\end{align}
 In particular, $p\mapsto \bfk_n(p)$ is increasing and $p \mapsto \bfs_m(p)$ is decreasing. 

\subsection{Initial bounds}
Our argument will proceed by first bounding the prelimit probabilities in \eqref{E:CDFLim}. To express these bounds, we introduce some auxiliary weights, as follows. Recall the function $\zeta:\bbR_{>0} \to (0, 1)$ defined by \eqref{E:zeta}. For each increasing sequence $\bfk = (k_p)_{p \in [d]} \in \bbZ_{\ge k}^d$, let $\omega^{\col, \bfk} = \{\omega_{(i, j)}^{\col, \bfk}: i, j \in \bbZ_{>0}\}$ be independent weights such that 
\begin{align}
\label{E:wh}
\begin{split}
&\omega^{\col, \bfk}_{(i, j)} \sim 
\begin{cases}
\Exp\{1-\zeta(r_{p})\} &\text{ if } i = k_p \\ 
\Exp\{1\}&\text{ otherwise, } 
\end{cases}
\end{split}
\end{align}
where $\zeta$ is the function defined at \eqref{E:zeta}. To construct these weights on the same probability space as the i.i.d.\ $\Exp\{1\}$ weights $\omega$, we will set $\omega_{(i,j)}^{\col, \bfk} = (1 - \zeta(r_p))^{-1} \omega_{(i,j)}$  when $i = k_p$ and $\omega_{(i,j)}^{\col, \bfk} = \omega_{(i,j)}$ otherwise.  
Through definition \eqref{E:iw}, one can view $\omega^{\col, \bfk}$ as the inhomogeneous weights $\omega^{\bfa, \bfb}$ with 
parameters 
\begin{align}
\label{E:whPar}
\begin{split}
&a_i = 
\begin{cases}
-\zeta(r_{p}) \quad &\text{ if } i = k_p \\ 
0 \quad &\text{ otherwise} 
\end{cases}
\qquad \text{ and } \qquad b_j = 1 
\end{split}
\end{align}
for $i, j \in \bbZ_{>0}$. Similarly, for each decreasing sequence $\bfs = (s_p)_{p \in [d]} \in \bbZ_{\ge \ell}^d$, let $\omega^{\row, \bfs} = \{\omega^{\row, \bfs}_{(i, j)}:i, j \in \bbZ_{>0}\}$ be independent weights such that 
\begin{align}
\label{E:whor}
\begin{split}
&\omega^{\row, \bfs}_{(i, j)} \sim 
\begin{cases}
\Exp\{\zeta(r_{p})\} &\text{ if } j = s_p \\ 
\Exp\{1\} &\text{ otherwise. } 
\end{cases}
\end{split}
\end{align}
Similarly as before, we will construct these weights on our same probability space. One can again interpret $\omega^{\row, \bfs}$ as a specialization of $\omega^{\bfa, \bfb}$ by choosing  
\begin{align}
\label{E:wvPar}
\begin{split}
&a_i = 0 \quad \text{ and } \quad b_j = 
\begin{cases}
\zeta(r_{p}) \quad &\text{ if } j = s_p \\ 
1 \quad &\text{ otherwise} 
\end{cases}
\end{split}
\end{align}
for $i, j \in \bbZ_{>0}$. 

Next, introduce the sequences $\wt{\bfk}_n = (\wt{\bfk}_n(p))_{p \in [d]}$ and $\wt{\bfs}_m = (\wt{\bfs}_m(p))_{p \in [d]}$ by 
\begin{align}
\label{E:wtkl}
\begin{split}
\wt{\bfk}_n(p) \coloneqq \bfk_n(p) + p, \quad \text{ and } \quad \wt{\bfs}_m(p) \coloneqq \bfs_m(p) + d+1-p \quad \text{ for } p \in [d]. 
\end{split}
\end{align}
The following lemma bounds the prelimit probabilities in \eqref{E:CDFLim} with the corresponding probabilities computed with the weights $\omega^{\col, \wt{\bfk}_n}$ and $\omega^{\row, \wt{\bfs}_m}$. Figure \ref{F:CDFBd} in Section \ref{Ss:BMdiscuss} traces the proof in a small case. 
\begin{lem}
\label{L:CDFBd}
The following inequalities hold for $m \in \bbZ_{\ge k}$ and $n \in \bbZ_{\ge \ell}$. 
\begin{enumerate}[\normalfont (a)]
\item \label{L:CDFBd_a} $
\begin{aligned}[t]
&\bfP\{\ul{\I}_{\bfu, (\bfk_n(p), n)}[\omega] > x_{\bfu, p},\;\ul{\J}_{\bfv, (\bfk_n(p), n)}[\omega] < y_{\bfv, p}, \quad  \forall \; \bfu,\bfv,p\} \\ 
&\ge \bfP\{\ul{\I}_{\bfu, (\wt{\bfk}_n(p), n)}[\omega^{\col, \wt{\bfk}_n}] > x_{\bfu, p},\; \ul{\J}_{\bfv, (\wt{\bfk}_n(p), n)}[\omega^{\col, \wt{\bfk}_n}] < y_{\bfv, p}, \; \forall \;\bfu,\bfv,p\}. 
\end{aligned}
$
\vspace{0.1in}
\item \label{L:CDFBd_b} $
\begin{aligned}[t]
&\bfP\{\ul{\I}_{\bfu, (m, \bfs_m(p))}[\omega] > x_{\bfu, p},\; \ul{\J}_{\bfv, (m, \bfs_m(p))}[\omega] < y_{\bfv, p}, \quad  \forall \; \bfu,\bfv,p\} \\ 
&\le \bfP\{\ul{\I}_{\bfu, (m, \wt{\bfs}_m(p))}[\omega^{\row, \wt{\bfs}_m}] > x_{\bfu, p},\;\ul{\J}_{\bfv, (m, \wt{\bfs}_m(p))}[\omega^{\row, \wt{\bfs}_m}] < y_{\bfv, p}\quad, \forall \; \bfu,\bfv,p \}. 
\end{aligned}
$
\end{enumerate}
\end{lem}
\begin{proof}  We prove \eqref{L:CDFBd_a}, and the proof of \eqref{L:CDFBd_b} is symmetric. 
Let $n \in \bbZ_{\ge \ell}$. For each $q \in [d] \cup \{0\}$, define
\begin{align}
\label{E:24}
\bfk_n^q(p) &\coloneqq \one_{\{p \le d-q\}} \cdot \bfk_n(p) + \one_{\{p > d-q\}} \cdot \wt{\bfk}_n(p) \quad \text{ for } p \in [d]. 
\end{align}
We immediately observe that, by monotonicity of $p \mapsto \bfk_n(p)$, assumption \eqref{E:Gap}, and the definition of $\wt \bfk_n$ \eqref{E:wtkl}, the sequence $p \mapsto \bfk_n^q(p)$ is strictly increasing for each $q \in [d] \cup \{0\}$. One can think of \eqref{E:24} as transforming the initial sequence $\bfk_n^{0} = \bfk_n$ into the final sequence $\bfk_n^{d} = \wt{\bfk}_n$ in $d$ steps. At step $q \in [d]$, one obtains the sequence $\bfk_n^{q}$ from $\bfk_n^{q-1}$ by replacing the $(d-q+1)$th term $\bfk_n^{q-1}(d-q+1) = \bfk_n(d-q+1)$ with $\wt{\bfk}_n(d-q+1) = \bfk_n(d-q+1)+d-q+1$. For each $q \in [d-1] \cup \{0\}$, we also define 
\begin{align}
\label{E:22}
\begin{split}
\wt{\bfk}_n^q(p) &\coloneqq \one_{\{p \le d-q\}} \cdot \{\bfk_n(d-q) + p\} + \one_{\{p > d-q\}} \cdot \wt{\bfk}_n(p) \quad \text{ for } p \in [d]. 
\end{split}
\end{align}
In particular, $\wt{\bfk}_n^0(p) = \bfk_n(d) + p$ for $p \in [d]$, and $\wt{\bfk}_n^{d-1} = \wt{\bfk}_n$. We can similarly see that $p \mapsto \wt k_n^q(p)$ is strictly increasing for each $q \in [d-1] \cup\{0\}$.  Definition \eqref{E:22} encodes the following dynamics. For each $q \in [d-1]$, the sequence $\wt{\bfk}_n^{q}$ is produced from the sequence $\wt{\bfk}_n^{q-1}$  by replacing the first $d-q$ terms of the latter with $\bfk_n(d-q)+1, \dotsc, \bfk_n(d-q)+d-q$. 

Since $\bfk_n^0 = \bfk_n$, $\wt{\bfk}_n^0(p) > \bfk_n(d)$ for all $p \in [d]$, and the event on the left-hand side in part \eqref{L:CDFBd_a} does not depend on the weights to the right of column $\bfk_n(d)$, we have 
\begin{align}
\label{E:23}
\begin{split}
&\bfP\{\ul{\I}_{\bfu, (\bfk_n(p), n)}[\omega] > x_{\bfu, p},\;\ul{\J}_{\bfv, (\bfk_n(p), n)}[\omega] < y_{\bfv, p}, \; \forall \; \bfu,\bfv,p\} \\ 
&= \bfP\{\ul{\I}_{\bfu, (\bfk_n^0(p), n)}[\omega^{\col, \wt{\bfk}_n^0}] > x_{\bfu p},\;  \ul{\J}_{\bfv, (\bfk_n^0(p), n)}[\omega^{\col, \wt{\bfk}_n^0}] < y_{\bfv, p}, \;  \forall \; \bfu,\bfv,p\}.  
\end{split}
\end{align}
To perform an inductive argument, we continue with two observations. First, for each $q \in [d]$, it follows from definitions \eqref{E:wtkl} and \eqref{E:24} along with Lemma \ref{L:Comp} that 

\begin{equation}\label{E:25}
\begin{aligned}
&\bfP\{\ul{\I}_{\bfu, (\bfk_n^{q-1}(p), n)}[\omega^{\col, \wt{\bfk}_n^{q-1}}] > x_{\bfu, p},\; \ul{\J}_{\bfv, (\bfk_n^{q-1}(p), n)}[\omega^{\col, \wt{\bfk}_n^{q-1}}] < y_{\bfv, p}, \;\forall \bfu,\bfv,p\}  \\ 
&= \bfP\bigl\{\ul{\I}_{\bfu, (\bfk_n^{q-1}(p), n)}[\omega^{\col, \wt{\bfk}_n^{q-1}}] > x_{\bfu, p},\;  \ul{\J}_{\bfv, (\bfk_n^{q-1}(p), n)}[\omega^{\col, \wt{\bfk}_n^{q-1}}] < y_{\bfv, p}, \;  \forall \bfu,\bfv,\; p\neq d-q + 1,  \\ 
&\qquad   \ \ul{\I}_{\bfu, (\bfk_n(d-q+1), n)}[\omega^{\col, \wt{\bfk}_n^{q-1}}] > x_{\bfu, d-q+1},   \ul{\J}_{\bfv, (\bfk_n(d-q+1), n)}[\omega^{\col, \wt{\bfk}_n^{q-1}}] < y_{\bfv, d-q+1},\;  \forall \bfu,\bfv \;\bigr\}  \\
&\ge \bfP\bigl\{\ul{\I}_{\bfu, (\bfk_n^{q-1}(p), n)}[\omega^{\col, \wt{\bfk}_n^{q-1}}] > x_{\bfu, p},\;  \ul{\J}_{\bfv, (\bfk_n^{q-1}(p), n)}[\omega^{\col, \wt{\bfk}_n^{q-1}}] < y_{\bfv, p}, \quad  \forall \bfu,\bfv,\; p\neq d-q + 1,  \\ 
&\qquad   \ \ul{\I}_{\bfu, (\bfk_n(d-q+1) + d-q + 1, n)}[\omega^{\col, \wt{\bfk}_n^{q-1}}] > x_{\bfu, d-q+1},\; \forall \bfu,\bfv,  \\
&\qquad \ \ul{\J}_{\bfv, (\bfk_n(d-q+1) + d - q + 1, n)}[\omega^{\col, \wt{\bfk}_n^{q-1}}] < y_{\bfv, d-q+1},\; \forall \bfu,\bfv \;\bigr\}   \\
&= \bfP\{\ul{\I}_{\bfu, (\bfk_n^{q}(p), n)}[\omega^{\col, \wt{\bfk}_n^{q-1}}] > x_{\bfu, p} \text{ and } \ul{\J}_{\bfv, (\bfk_n^{q}(p), n)}[\omega^{\col, \wt{\bfk}_n^{q-1}}] < y_{\bfv, p}, \;\forall \bfu,\bfv,p\}.  
\end{aligned}
\end{equation}
 Specifically, in the first equality, we split up the condition separately for $p = d-q + 1$ and replaced $\bfk_n^{q-1}(d-q + 1)$ with $\bfk_n(d-q + 1)$ using the definition \eqref{E:24} (with $q$ replaced by $q -1$), in the inequality, we replaced $\bfk_n(d - q + 1)$ with $\bfk_n(d - q + 1) + d - q + 1$ (using Lemma \ref{L:Comp}). We get the last equality because $\bfk_n(d - q + 1) + d - q + 1 = \wt{\bfk}_n(d-q + 1) = \bfk_n^q(p)$, and $\bfk_n^{q-1}(p) = \bfk_n^q(p)$ for all $p \neq d-q + 1$.

Second, recalling the definition of the weights $\omega^{\col, \bfk}_{(i, j)}$ in \eqref{E:wh}, for each $q \in [d-1]$, an application of Theorem \ref{T:LppInv}  (justified below) yields 
\begin{equation} \label{E:26}
\begin{aligned}
&\bfP\{\ul{\I}_{\bfu, (\bfk_n^{q}(p), n)}[\omega^{\col, \wt{\bfk}_n^{q-1}}] > x_{\bfu, p},\; \ul{\J}_{\bfv, (\bfk_n^{q}(p), n)}[\omega^{\col, \wt{\bfk}_n^{q-1}}] < y_{\bfv, p}, \;\forall \bfu,\bfv,p\}   \\
&= \bfP\{\ul{\I}_{\bfu, (\bfk_n^{q}(p), n)}[\omega^{\col, \wt{\bfk}_n^{q}}] > x_{\bfu, p},\; \ul{\J}_{\bfv, (\bfk_n^{q}(p), n)}[\omega^{\col, \wt{\bfk}_n^{q}}] < y_{\bfv, p}, \;\forall \bfu,\bfv,p\}. 
\end{aligned}
\end{equation}
The use of of Theorem \ref{T:LppInv} is justified as follows: Note from definitions \eqref{E:wh} and \eqref{E:22} that replacing the $\omega^{\col, \wt{\bfk}_n^{q-1}}$-weights with the $\omega^{\col, \wt{\bfk}_n^q}$-weights as in \eqref{E:26} amounts to permuting the column parameters via the permutation $\sigma \in \sS_m$ that swaps 
$\bfk_n^q(d-q)+p= $ $ \bfk_n(d-q) + p$ with $\wt{\bfk}_n^{q-1}(p) = \bfk_n(d-q+1) + p$ for $p \in [d-q]$ and fixes every other column index. The events in \eqref{E:26} can be expressed in terms of the last-passage times $\Lp_{\bfx, \bfy}$ where $\bfx \in [k] \times [\ell]$ and $\bfy \in \{(\bfk_n^q(p), n): p \in [d]\}$. Because the permuted columns are strictly between columns $\bfk_n^q(d-q)$ and $\bfk_n(d-q+1)+d-q+1 = \bfk_n^q(d-q+1)$, the monotonicity of $p \mapsto \bfk_n^q(p)$ implies that the condition \eqref{E:U} required for Theorem \ref{T:LppInv} holds. 

Combining \eqref{E:25} and \eqref{E:26}, we obtain that, for $q \in [d-1]$,
\begin{align}
 \label{E:27}
\begin{split}
&\bfP\{\ul{\I}_{\bfu, (\bfk_n^{q-1}(p), n)}[\omega^{\col, \wt{\bfk}_n^{q-1}}] > x_{\bfu, p},\; \ul{\J}_{\bfv, (\bfk_n^{q-1}(p), n)}[\omega^{\col, \wt{\bfk}_n^{q-1}}] < y_{\bfv, p}, \;\forall \bfu,\bfv,p\}  \\ 
&\ge \bfP\{\ul{\I}_{\bfu, (\bfk_n^{q}(p), n)}[\omega^{\col, \wt{\bfk}_n^{q}}] > x_{\bfu, p}, \;\ul{\J}_{\bfv, (\bfk_n^{q}(p), n)}[\omega^{\col, \wt{\bfk}_n^{q}}] < y_{\bfv, p}, \;\forall \bfu,\bfv,p\}.
\end{split}
\end{align}
Now, starting from \eqref{E:23}, iterating \eqref{E:27} for $q \in [d-1]$, and concluding with another appeal to \eqref{E:25} with $q = d$, we obtain 
\begin{align}
&\bfP\{\ul{\I}_{\bfu, (\bfk_n(p), n)}[\omega] > x_{\bfu, p},\;\ul{\J}_{\bfv, (\bfk_n(p), n)}[\omega] < y_{\bfv, p}, \; \forall \; \bfu,\bfv,p\} \nonumber \\
&\ge \bfP\{\ul{\I}_{\bfu, (\bfk_n^{d-1}(p), n)}[\omega^{\col, \wt{\bfk}_n^{d-1}}] > x_{\bfu, p},\;\ul{\J}_{\bfv, (\bfk_n^{d-1}(p), n)}[\omega^{\col, \wt{\bfk}_n^{d-1}}] < y_{\bfv, p}, \forall \bfu,\bfv,p \} \nonumber \\
&\ge \bfP\{\ul{\I}_{\bfu, (\bfk_n^{d}(p), n)}[\omega^{\col, \wt{\bfk}_n^{d-1}}] > x_{\bfu,p},\; \ul{\J}_{\bfv, (\bfk_n^{d}(p), n)}[\omega^{\col, \wt{\bfk}_n^{d-1}}] < y_{\bfv,p}, \; \forall \bfu,\bfv,p \} \nonumber \\
&= \bfP\{\ul{\I}_{\bfu, (\wt{\bfk}_n(p), n)}[\omega^{\col, \wt{\bfk}_n}] > x_{\bfu,p},\;\ul{\J}_{\bfv, (\wt{\bfk}_n(p), n)}[\omega^{\col, \wt{\bfk}_n}] < y_{\bfv,p},\;\forall \bfu,\bfv,p \}. \label{E:28}
\end{align}
For the equality in \eqref{E:28} recall from definitions \eqref{E:24} and \eqref{E:22}  that $\bfk_n^d = \wt{\bfk}_n^{d-1} = \wt{\bfk}_n$. 
\end{proof}

\subsection{Limiting bounds}

We next connect the subsequential limits of the bounds in Lemma \ref{L:CDFBd} to the thin Busemann functions from Section \ref{S:iBusFn}. 
For an illustration of the proof below in a smaller case, see Figure \ref{F:CDFLim} in Section \ref{Ss:BMdiscuss}. Recall the definition of $\bfk_0$ and $\bfs_0$ in \eqref{E:kl0}. 
\begin{lem}
\label{L:CDFLim}
Recall the definitions in \eqref{E:ThinBuse}. The following inequalities hold. 
\begin{enumerate}[\normalfont (a)]
\item \label{itm:ge1} $
\begin{aligned}[t]
&\varliminf_{n \to \infty }\bfP\{\ul{\I}_{\bfu, (\wt{\bfk}_n(p), n)}[\omega^{\col, \wt{\bfk}_n}] > x_{\bfu,p},\;  \ul{\J}_{\bfv, (\wt{\bfk}_n(p), n)}[\omega^{\col, \wt{\bfk}_n}] < y_{\bfv,p},\quad \forall \bfu,\bfv,p  \} \\ 
&\ge \bfP\{\sI_{\bfu}^{\bfk_0(p), \uparrow}[\omega^{\col, \bfk_0}] > x_{\bfu,p},\;\sJ_{\bfv}^{\bfk_0(p), \uparrow}[\omega^{\col, \bfk_0}] < y_{\bfv,p},\quad \forall \bfu,\bfv,p \}. 
\end{aligned}
$
\item \label{itm:le1} $
\begin{aligned}[t]
&\varlimsup_{m \to \infty}\bfP\{\ul{\I}_{\bfu, (m, \wt{\bfs}_m(p))}[\omega^{\row, \wt{\bfs}_m}] > x_{\bfu,p},\; \ul{\J}_{\bfv, (m, \wt{\bfs}_m(p))}[\omega^{\row, \wt{\bfs}_m}] < y_{\bfv,p},\quad \forall \bfu,\bfv,p \}\\ 
&\le \bfP\{\sI_{\bfu}^{\bfs_0(p), \rightarrow}[\omega^{\row, \bfs_0}] > x_{\bfu,p},\; \sJ_{\bfv}^{\bfs_0(p), \rightarrow}[\omega^{\row, \bfs_0}] < y_{\bfv,p},\quad\forall \bfu,\bfv,p \}. 
\end{aligned}
$
\end{enumerate}
\end{lem}
\begin{proof}
We prove \eqref{itm:ge1}, with \eqref{itm:le1} following a symmetric proof. Let $n \in \bbZ_{\ge \ell}$. For each $q \in [d] \cup \{0\}$, we define the sequence $\wh{\bfk}_n^{q} = (\wh{\bfk}_n^{q}(p))_{p \in [d]} \in \bbZ_{\ge k}^d$ by 
\begin{align}
\label{E:knq}
\begin{split}
\wh{\bfk}_n^{q}(p) &\coloneqq \one_{\{p \le q\}} \cdot \bfk_0(p) + \one_{\{p > q\}} \cdot \wt{\bfk}_n(p) \quad \text{ for } p \in [d].
\end{split}
\end{align}
One now begins with the sequence $\wh{\bfk}_n^{0} = \wt{\bfk}_n$ and ends with the sequence $\wh{\bfk}_n^{d} = \bfk_0$. For each $q \in [d]$, one can obtain $\wh{\bfk}_n^{q}$ from $\wh{\bfk}_n^{q-1}$ by replacing the $q$th term $\wh{\bfk}_n^{q-1}(q) = \wt{\bfk}_n(q)$ with $\bfk_0(q) = k+q-1$. 

For each $q \in [d] \cup \{0\}$ and $\epsilon > 0$, define the event
\begin{equation}
\label{E:EvtE}
\begin{aligned}
E^{q, \epsilon}_n &\coloneqq \Bigl\{\ul{\I}_{\bfu, (\wt{\bfk}_n(p), n)}[\omega^{\col, \wh{\bfk}_n^q}] > x_{\bfu,p},\; \ul{\J}_{\bfv, (\wt{\bfk}_n(p), n)}[\omega^{\col, \wh{\bfk}_n^q}] < y_{\bfv,p},\; \forall \bfu,\bfv,\; p \in [d] \setminus[q], \\ 
&\qquad \ \ \sI_{\bfu}^{\bfk_0(p), \uparrow}[\omega^{\col, \wh{\bfk}_n^q}] > x_{\bfu,p} + \epsilon,\; \sJ_{\bfv}^{\bfk_0(p), \uparrow}[\omega^{\col, \wh{\bfk}_n^q}] < y_{\bfv,p} - \epsilon,\;\forall \bfu,\bfv,\;p \in [q]\Bigr\}.
\end{aligned}
\end{equation}
Note that $E_n \coloneqq E_n^{0, \epsilon}$ is precisely the prelimit event in Part \eqref{itm:ge1}, and the event $E^{d, \epsilon}_n$ does not depend on $n$ because $\wh{\bfk}_n^d = \bfk_0$ by the definition \eqref{E:knq}. Thus, we abbreviate $E^\epsilon \coloneqq E^{d, \epsilon}_n$. The event $E^\epsilon$ is a prelimiting version of the event on the right-hand side of \eqref{itm:ge1}.

Our next goal is to compare the probabilities of the events $E^{q-1, \epsilon}_n$ and $E^{q, \epsilon}_n$ for each $q \in [d]$. First, by definition \eqref{E:EvtE} and as a consequence of Theorem \ref{T:LppInv} (justified below),
\begin{equation}\label{E:1}
\begin{aligned}
&\quad \, \bfP\{E_n^{q-1, \epsilon}\} \\ 
&= \bfP \Bigl\{\ul{\I}_{\bfu, (\wt{\bfk}_n(p), n)}[\omega^{\col, \wh{\bfk}_n^{q-1}}] > x_{\bfu,p},\;\ul{\J}_{\bfv, (\wt{\bfk}_n(p), n)}[\omega^{\col, \wh{\bfk}_n^{q-1}}] < y_{\bfv,p},\;\; \forall \bfu,\bfv,\; p \in [d]\setminus[q-1], \\ 
&\quad \quad \ \ \sI_{\bfu}^{\bfk_0(p), \uparrow}[\omega^{\col, \wh{\bfk}_n^{q-1}}] > x_{\bfu,p} + \epsilon,\; \sJ_{\bfv}^{\bfk_0(p), \uparrow}[\omega^{\col, \wh{\bfk}_n^{q-1}}] < y_{\bfv,p} - \epsilon,\;\forall \bfu,\bfv,\; p \in [q-1] \Bigr\}  \\ 
&= \bfP \Bigl\{\ul{\I}_{\bfu, (\wt{\bfk}_n(p), n)}[\omega^{\col, \wh{\bfk}_n^{q}}] > x_{\bfu,p},\;\ul{\J}_{\bfv, (\wt{\bfk}_n(p), n)}[\omega^{\col, \wh{\bfk}_n^{q}}] < y_{\bfv,p},\;\; \forall \bfu,\bfv,\; p \in [d]\setminus[q-1], \\ 
&\quad \quad \ \ \sI_{\bfu}^{\bfk_0(p), \uparrow}[\omega^{\col, \wh{\bfk}_n^{q}}] > x_{\bfu,p} + \epsilon,\; \sJ_{\bfv}^{\bfk_0(p), \uparrow}[\omega^{\col, \wh{\bfk}_n^{q}}] < y_{\bfv,p} - \epsilon,\;\forall \bfu,\bfv,\; p \in [q-1] \Bigr\}.  
\end{aligned}
\end{equation}
The application of Theorem \ref{T:LppInv} is justified as follows: Recall from \eqref{E:knq} that switching the weights above from $\omega^{\col, \wh{\bfk}_n^{q-1}}$ to $\omega^{\col, \wh{\bfk}_n^{q}}$ means interchanging the parameters of columns $\wt{\bfk}_n(q)$ and $\bfk_0(q)$. Let $\sigma \in \sS_m$ denote the corresponding permutation. Because the event in \eqref{E:1} can be written entirely in terms of last-passage times $\Lp_{\bfx, \bfy}$ and limits of their differences with \begin{align*}
\bfx \in [k] \times [\ell] \quad \text{ and } \quad \bfy \in ([\bfk_0(q-1)] \times \bbZ_{\ge \ell}) \cup \{(\wt{\bfk}_n(p), n): p \in [d] \smallsetminus [q-1]\}, 
\end{align*}
and these endpoints fulfill condition \eqref{E:U} with the permutations $\sigma$ and $\tau = \Id$, the preceding application of Theorem \ref{T:LppInv} is justified. 

To relate to the probability of the event $E_n^{q,\epsilon}$, we apply a union bound to the last probability in \eqref{E:1} to obtain  
\begin{align*}
\begin{split}
&\bfP\{E_n^{q-1, \epsilon}\} \\ 
&= \bfP\Bigl\{\ul{\I}_{\bfu, (\wt{\bfk}_n(p), n)}[\omega^{\col, \wh{\bfk}_n^{q}}] > x_{\bfu,p},\; \ul{\J}_{\bfv, (\wt{\bfk}_n(p), n)}[\omega^{\col, \wh{\bfk}_n^{q}}] < y_{\bfv,p},\;\forall \bfu,\bfv,\;p \in [d]\setminus[q],  \\
&\quad \quad \ \ \ul{\I}_{\bfu, (\wt{\bfk}_n(q), n)}[\omega^{\col, \wh{\bfk}_n^{q}}] > x_{\bfu,q},\; \ul{\J}_{\bfv, (\wt{\bfk}_n(q), n)}[\omega^{\col, \wh{\bfk}_n^{q}}] < y_{\bfv,q},\;\forall \bfu,\bfv, \\
&\quad \quad \ \ \sI_{\bfu}^{\bfk_0(p), \uparrow}[\omega^{\col, \wh{\bfk}_n^{q}}] > x_{\bfu,p} + \epsilon,\; \sJ_{\bfv}^{\bfk_0(p), \uparrow}[\omega^{\col, \wh{\bfk}_n^{q}}] < y_{\bfv,p} - \epsilon,\;\forall \bfu,\bfv,\; p \in [q-1]\Bigr\} \\ 
&\ge \bfP\Bigl\{\ul{\I}_{\bfu, (\wt{\bfk}_n(p), n)}[\omega^{\col, \wh{\bfk}_n^{q}}] > x_{\bfu,p},\; \ul{\J}_{\bfv, (\wt{\bfk}_n(p), n)}[\omega^{\col, \wh{\bfk}_n^{q}}] < y_{\bfv,p},\;\forall \bfu,\bfv,\;p \in [d]\setminus[q], \\ 
&\quad \quad \ \ \sI_{\bfu}^{\bfk_0(q), \uparrow}[\omega^{\col, \wh{\bfk}_n^{q}}] > x_{\bfu,q} + \epsilon,\; \sJ_{\bfv}^{\bfk_0(q), \uparrow}[\omega^{\col, \wh{\bfk}_n^{q}}] < y_{\bfv,q} - \epsilon,\;\forall \bfu,\bfv, \\
&\quad \quad \ \ \sI_{\bfu}^{\bfk_0(p), \uparrow}[\omega^{\col, \wh{\bfk}_n^{q}}] > x_{\bfu,p} + \epsilon,\; \sJ_{\bfv}^{\bfk_0(p), \uparrow}[\omega^{\col, \wh{\bfk}_n^{q}}] < y_{\bfv,p} - \epsilon,\;\forall \bfu,\bfv,\;p \in [q-1] \Bigr\} \\
&\qquad\qquad-  \sum_{\bfu \in \rightset_{k,\ell}} \bfP\Bigl\{\sI_{\bfu}^{\bfk_0(q), \uparrow}[\omega^{\col, \wh{\bfk}_n^{q}}] - \ul{\I}_{\bfu, (\wt{\bfk}_n(q), n)}[\omega^{\col, \wh{\bfk}_n^{q}}]> \epsilon\Bigr\} \\ 
&\qquad\qquad\qquad\qquad- \sum_{\bfv \in \upset_{k,\ell}}\bfP\Bigl\{\sJ_{\bfv}^{\bfk_0(q), \uparrow}[\omega^{\col, \wh{\bfk}_n^{q}}]-\ul{\J}_{\bfv, (\wt{\bfk}_n(q), n)}[\omega^{\col, \wh{\bfk}_n^{q}}] < -\epsilon\Bigr\} \\ 
&= \bfP\{E_n^{q, \epsilon}\} - \sum_{\bfu \in \rightset_{k,\ell}} \bfP\Bigl\{\sI_{\bfu}^{\bfk_0(q), \uparrow}[\omega^{\col, \wh{\bfk}_n^{q}}] - \ul{\I}_{\bfu, (\wt{\bfk}_n(q), n)}[\omega^{\col, \wh{\bfk}_n^{q}}]> \epsilon\Bigr\} \\ 
&\qquad\qquad\qquad\qquad- \sum_{\bfv \in \upset_{k,\ell}}\bfP\Bigl\{\sJ_{\bfv}^{\bfk_0(q), \uparrow}[\omega^{\col, \wh{\bfk}_n^{q}}]-\ul{\J}_{\bfv, (\wt{\bfk}_n(q), n)}[\omega^{\col, \wh{\bfk}_n^{q}}] < -\epsilon\Bigr\}. 
\end{split}
\end{align*}
Rearranged, this yields
\begin{equation} \label{E:2}
\begin{aligned}
   \bfP\{E_n^{q-1, \epsilon}\} -  \bfP\{E_n^{q, \epsilon}\} &\ge - \sum_{\bfu \in \rightset_{k,\ell}} \bfP\Bigl\{\sI_{\bfu}^{\bfk_0(q), \uparrow}[\omega^{\col, \wh{\bfk}_n^{q}}] - \ul{\I}_{\bfu, (\wt{\bfk}_n(q), n)}[\omega^{\col, \wh{\bfk}_n^{q}}]> \epsilon\Bigr\}. \\ 
&\qquad- \sum_{\bfv \in \upset_{k,\ell}}\bfP\Bigl\{\sJ_{\bfv}^{\bfk_0(q), \uparrow}[\omega^{\col, \wh{\bfk}_n^{q}}]-\ul{\J}_{\bfv, (\wt{\bfk}_n(q), n)}[\omega^{\col, \wh{\bfk}_n^{q}}] < -\epsilon\Bigr\}.
\end{aligned}
\end{equation}
Recall that we defined $E_n = E_n^{0,\epsilon}$  and $E^\epsilon = E_n^{d,\epsilon}$. Summing up the bound in \eqref{E:2} over $q \in [d]$ and noting a telescoping of terms then gives 
\begin{align}
\label{E:3}
\begin{split}
\bfP\{E_n\} &\ge \bfP\{E^{\epsilon}\}- \sum_{q \in [d],\bfu \in \rightset_{k,\ell}} \bfP\Bigl\{\sI_{\bfu}^{\bfk_0(q), \uparrow}[\omega^{\col, \wh{\bfk}_n^{q}}] - \ul{\I}_{\bfu, (\wt{\bfk}_n(q), n)}[\omega^{\col, \wh{\bfk}_n^{q}}]> \epsilon\Bigr\} \\ 
&\qquad\qquad\quad- \sum_{q \in [d],\bfv \in \upset_{k,\ell}}\bfP\Bigl\{\sJ_{\bfv}^{\bfk_0(q), \uparrow}[\omega^{\col, \wh{\bfk}_n^{q}}]-\ul{\J}_{\bfv, (\wt{\bfk}_n(q), n)}[\omega^{\col, \wh{\bfk}_n^{q}}] < -\epsilon\Bigr\}.
\end{split}
\end{align}

Fixing $q \in [d]$ for the moment, consider the inhomogeneous weights $\omega^{\bfa, \bfb}$ with parameters 
\begin{align}
\label{E:29}
a_i = 
\begin{cases}
-\zeta(r_p) \quad &\text{ if } i = \bfk_0(p) \text{ for some } p \in [q] \\ 
0 \quad &\text{ otherwise}
\end{cases}
\quad \text{ and } \quad b_j = 1 
\end{align}
for $i, j \in \bbZ_{>0}$. By \eqref{E:whPar} and definition \eqref{E:knq}, the $\omega^{\col, \wh{\bfk}_n^q}$-weights coincide in distribution with $\omega^{\bfa, \bfb}$ on $[\wt{\bfk}_n(q)] \times \bbZ_{>0}$. For the parameters in \eqref{E:29}, the limit measures in \eqref{A:VagConv} are 
\begin{align}
\label{E:30}
\alpha = \delta_0 \quad \text{ and } \quad \beta = \delta_1. 
\end{align}
Since $(\zeta(r_p))_{p \in [d]}$ is increasing, for each $(i, j) \in [k] \times [\ell]$, one also has 
\begin{align}
\label{E:31}
\inf \bfa_{i:\infty} = - \zeta(r_q) \quad \text{ and } \quad \inf \bfb_{j:\infty} = 1. 
\end{align}
Using \eqref{E:30} and \eqref{E:31}, one computes the critical directions 
in \eqref{E:crit} as 
\begin{align}
\label{E:32}
\mfc_i^{\ver} = \frac{\zeta(r_q)^2}{(1-\zeta(r_q))^2} = r_q \quad \text{ and } \quad \mfc_j^{\hor} = \infty
\end{align}
for $i \in [k]$ and $j \in \bbZ_{>0}$. Note also from definition \eqref{E:wtkl} that 
\begin{align}
\label{E:33}
\lim_{n \to \infty}\frac{\wt{\bfk}_n(q)}{n} = \lim_{n \to \infty}\frac{\bfk_n(q)}{n} = r_q. 
\end{align}
By applying Proposition \ref{P:iBusFn} and \eqref{E:33} in the first equality below, followed by  Proposition \ref{P:iBusFlat} and \eqref{E:32}, we obtain that, almost surely, 
\begin{align}
\label{E:34}
\begin{split}
&\lim_{n \to \infty} \ul{\I}_{\bfu, (\wt{\bfk}_n(q), n)}[\omega^{\bfa, \bfb}] = \sI^{r_q}_{(i, j)}[\omega^{\bfa, \bfb}] = \lim_{k \to \infty} \sI_{\bfu}^{\uparrow, k}[\omega^{\bfa, \bfb}].
\end{split}
\end{align}
Then, by \eqref{E:29}, the first part of \eqref{E:31}, and Corollary \ref{C:ThinBuse}, for all $k \ge \mathbf k_0(q)$, we have
\[
\sI_{\bfu}^{\uparrow, k}[\omega^{\bfa, \bfb}]= \sI_{\bfu}^{\uparrow, \bfk_0(q)}[\omega^{\bfa, \bfb}].
\]
Hence, \eqref{E:34} gives us the almost sure limit
\[
\lim_{n \to \infty} \ul{\I}_{\bfu, (\wt{\bfk}_n(q), n)}[\omega^{\bfa, \bfb}]  = \sI_{\bfu}^{\uparrow, \bfk_0(q)}[\omega^{\bfa, \bfb}].
\]
Consequently,  
\begin{align}
\label{E:35}
\lim_{n \to \infty} \bfP\{\sI_{\bfu}^{\bfk_0(q), \uparrow}[\omega^{\col, \wh{\bfk}_n^{q}}] - \ul{\I}_{\bfu, (\wt{\bfk}_n(q), n)}[\omega^{\col, \wh{\bfk}_n^{q}}]> \epsilon\}  = 0,\quad \forall \bfu \in \rightset_{k,\ell}.
\end{align}
Similarly, one finds that 
\begin{align}
\label{E:36}
\lim_{n \to \infty} \bfP\{\sJ_{\bfv}^{\bfk_0(q), \uparrow}[\omega^{\col, \wh{\bfk}_n^{q}}]-\ul{\J}_{\bfv, (\wt{\bfk}_n(q), n)}[\omega^{\col, \wh{\bfk}_n^{q}}] < -\epsilon\} = 0,\quad \forall \bfv \in \upset_{k,\ell}.
\end{align}

Now, returning to \eqref{E:3} and using \eqref{E:35}--\eqref{E:36} for each $q \in [d]$, one obtains that 
\begin{align}
\label{E:37}
\begin{split}
&\varliminf_{n \to \infty} \bfP\{E_n\} \ge \bfP\{E^\epsilon\} \\
&= \bfP\{\sI_{\bfu}^{\bfk_0(p), \uparrow}[\omega^{\col, \bfk_0}] > x_{\bfu,p} + \epsilon,\; \sJ_{\bfv}^{\bfk_0(p), \uparrow}[\omega^{\col, \bfk_0}] < y_{\bfv,p} - \epsilon,\;\forall\bfu,\bfv,p \}. 
\end{split}
\end{align}
Through monotone convergence, as $\epsilon \to 0$, the last probability tends to 
\begin{align*}
&\bfP\{\sI_{\bfu}^{\bfk_0(p), \uparrow}[\omega^{\col, \bfk_0}] > x_{\bfu,p},\; \sJ_{\bfv}^{\bfk_0(p), \uparrow}[\omega^{\col, \bfk_0}] < y_{\bfv,p},\;\forall \bfu,\bfv,p \}, 
\end{align*}
completing the proof of part \eqref{itm:ge1}. 
\end{proof}

\subsection{Expressing the upper and lower bounds in terms of LPP in a finite grid}

At this point of the argument, one has the following inequalities by combining the limits in \eqref{E:CDFLim} with Lemmas \ref{L:CDFBd} and \ref{L:CDFLim}: 
\begin{align}
\label{E:CDFBds}
\begin{split}
&\quad\;\bfP\{\sI_{\bfu}^{\bfk_0(p), \uparrow}[\omega^{\col, \bfk_0}] > x_{\bfu,p},\; \sJ_{\bfv}^{\bfk_0(p), \uparrow}[\omega^{\col, \bfk_0}] < y_{\bfv,p},\;\forall \bfu,\bfv,p \} \\ 
&\le \bfP\{\sI^{r_p}_{(i, j)}[\omega] > x_{\bfu,p},\; \sJ^{r_p}_{(i, j)}[\omega] < y_{\bfv,p},\;\forall \bfu,\bfv,p\} \\
&\le \bfP\{\sI_{\bfu}^{\bfs_0(p), \rightarrow}[\omega^{\row, \bfs_0}] > x_{\bfu,p},\; \sJ_{\bfv}^{\bfs_0(p), \rightarrow}[\omega^{\row, \bfs_0}] < y_{\bfv,p},\;\forall \bfu,\bfv,p \}.
\end{split}
\end{align}
Hence, to obtain Theorem \ref{T:BusMar}, it remains to show 
that the inequalities in \eqref{E:CDFBds} are equalities and identify their common value in terms of the $\eta$-weights given by \eqref{E:auxw}. To achieve this, we first record a few more auxiliary facts.   

For each $\bfy \le \bfz \in \Z^2$, define the map $\w \mapsto \wt \w^{\bfy,\bfz}$ as follows: Here, the input and image are both configurations of weights indexed by $\Z^2$, and for $\w = \{\w_\bfx:\bfx \in \Z^2\}$, the image $\wt \w^{\bfy,\bfz} = \{\wt \w_\bfx^{\bfy,\bfz}: \bfx \in \Z^2\}$ is defined by
\begin{align}
\label{E:wInd}
\wt{\w}_{\bfx}^{\bfy,\bfz} \coloneqq 
\begin{cases}
\w_\bfx \quad &\text{ if } \bfx \le \bfy-(1, 1) \\ 
\ul{\I}_{\bfx, \bfz}[\w] \quad &\text{ if } \bfx \cdot (1, 0) < \bfy \cdot (1, 0) \text{ and }  \bfx \cdot (0, 1) = \bfy \cdot (0, 1) \\ 
\ul{\J}_{\bfx, \bfz}[\w] \quad &\text{ if } \bfx \cdot (0, 1) < \bfy \cdot (0, 1) \text{ and }  \bfx \cdot (1, 0) = \bfy \cdot (1, 0) \\ 
0 \quad &\text{ otherwise. }
\end{cases}
\end{align}
In words, this map leaves the weights strictly to the southeast of $\bfy$ unchanged, sets weights to the left of $\bfy$ using horizontal increments of last-passage times to the point $\bfz$, and sets the weights below $\bfy$ to be vertical increments of last-passage times to the point $\bfz$. All other weights are set to $0$; in particular, $\wt \w_{\bfy}^{\bfy,\bfz} = 0$.

The last-passage times computed with the $\w$-weights and the induced weights are related as follows:  
\begin{align}
\label{E:IndLpp}
\Lp_{\bfx, \bfy}[\wt{\w}^{\bfy,\bfz}] = \Lp_{\bfx, \bfz}[\w] - \Lp_{\bfy, \bfz}[\w] \quad \text{ for } \bfx \le \bf y. 
\end{align}
For a proof of \eqref{E:IndLpp}, see \cite[Lemma A.1]{sepp-cgm-18} (there, the boundary weights are along a southwest boundary, while our boundary weights are along a northeast boundary; the two are equivalent by reflection). As an immediate consequence of \eqref{E:IndLpp}, one obtains the 
following increment identities for $\bfx \le \bfy$: 
\begin{align}
\label{E:IndInc}
\begin{split}
\ul{\I}_{\bfx, \bfy}[\wt{\w}^{\bfy,\bfz}] &= \ul{\I}_{\bfx, \bfz}[\w] \quad \text{ if } \bfx \cdot (1, 0) < \bfy \cdot (1, 0),\\ 
\ul{\J}_{\bfx, \bfy}[\wt{\w}^{\bfy,\bfz}] &= \ul{\J}_{\bfx, \bfz}[\w] \quad \text{ if } \bfx \cdot (0, 1) < \bfy \cdot (0, 1). 
\end{split}
\end{align}

We next revisit the setting of the inhomogeneous weights $\omega^{\bfa, \bfb}$ at the outset of Subsection \ref{S:iBusFn} and record the following law of large numbers, which is a corollary of \cite[Theorem 3.7]{Emra_Janj_Sepp_21}. The functions $\A$ and $\B$ in the statement were defined in \eqref{E:ABInt}. We also recall that $\alpha$ and $\beta$ are the limit measures appearing in assumption \eqref{A:VagConv}.  
\begin{prop}
\label{P:ThinLLN}
The following limits hold almost surely for each $i, j \in \bbZ_{>0}$.
\begin{enumerate}[\normalfont (a)]
\item If $m \in \bbZ_{\ge j}$ then 
\begin{align*}
\lim_{n \to \infty} \frac{1}{n} \cdot \Lp_{(i, j), (m, n)}[\omega^{\bfa, \bfb}] &= \B^\beta(-\min \bfa_{i:m}) = \int_\bbR \frac{\beta(\dd b)}{b+\min \bfa_{i:m}}.
\end{align*}
\item If $n \in \bbZ_{\ge i}$ then 
\begin{align*} 
\lim_{m \to \infty} \frac{1}{m} \cdot \Lp_{(i, j), (m, n)}[\omega^{\bfa, \bfb}] &= \A^\alpha(\min \bfb_{j:n}) = \int_\bbR \frac{\alpha(\dd a)}{a+\min \bfb_{j:n}}. 
\end{align*}
\end{enumerate}
\end{prop}

For $m, n \in \bbZ_{>0}$ and any real sequence $\bfc = (c_k)_{k \in \bbZ_{>0}}$, define the weights $\omega^{\bfa, \bfb, \col, m, \bfc}$ and $\omega^{\bfa, \bfb, \row, n, \bfc}$ on $\bbZ_{>0}^2$ by 
\begin{align}
\label{E:iw-2}
\omega^{\bfa, \bfb, \col, m, \bfc}_{(i, j)} = 
\begin{cases}
c_j \quad &\text{ if } i = m \\ 
\omega^{\bfa, \bfb}_{(i, j)} \quad &\text{ otherwise }
\end{cases}
\quad \text{ and } \quad 
\omega^{\bfa, \bfb, \row, n, \bfc}_{(i, j)} = 
\begin{cases}
c_i \quad &\text{ if } j = n \\ 
\omega^{\bfa, \bfb}_{(i, j)} \quad &\text{ otherwise. }
\end{cases}
\end{align} 
The next lemma is a slight variation 
of \cite[Lemma 4.23]{Emra_Janj_Sepp_25} and proved in the same manner.  
\begin{lem}
\label{L:ThinShp}
The following statements hold for any $k, \ell \in \bbZ_{>0}$. 
\begin{enumerate}[\normalfont (a)]
\item \label{it:thinA} Let $m \in \bbZ_{>k}$ and assume that $a_m < a_i$ for $i \in [m-1] \smallsetminus [k]$. Then, almost surely, there exists a {\rm(}random{\rm)} $N_0 \in \bbZ_{>\ell}$ such that 
\begin{align*}
\Lp_{(k, \ell), (m, n)}[\omega^{\bfa, \bfb, \row, n, \bfc}] = c_m + \Lp_{(k, \ell), (m, n-1)}[\omega^{\bfa, \bfb}] \quad \text{ for } n \ge N_0. 
\end{align*}
\item \label{it:thinB} Let $n \in \bbZ_{>k}$ and assume that $b_n < b_j$ for $j \in [n-1] \smallsetminus [\ell]$. Then, almost surely, there exists a {\rm(}random{\rm)} $M_0 \in \bbZ_{>\ell}$ such that 
\begin{align*}
\Lp_{(k, \ell), (m, n)}[\omega^{\bfa, \bfb, \col, m, \bfc}] = c_n + \Lp_{(k, \ell), (m-1, n)}[\omega^{\bfa, \bfb}] \quad \text{ for } m \ge M_0. 
\end{align*}
\end{enumerate}
\end{lem}
\begin{proof} We prove \eqref{it:thinA}, with \eqref{it:thinB} following a symmetric proof.
Applying Proposition \ref{P:ThinLLN} twice, one obtains the a.s.\ limits 
\begin{align}
\label{E:78}
\begin{split}
\lim_{n \to \infty} \frac{1}{n} \cdot \Lp_{(k, \ell), (m, n-1)}[\omega^{\bfa, \bfb}] &= \int_{\bbR} \frac{\beta(\dd b)}{b + \min \bfa_{k: m}} = \int_{\bbR} \frac{\beta(\dd b)}{b + a_m}, \\ 
\lim_{n \to \infty} \frac{1}{n} \cdot \Lp_{(k, \ell), (m-1, n-1)}[\omega^{\bfa, \bfb}] &= \int_{\bbR} \frac{\beta(\dd b)}{b + \min \bfa_{k: m-1}}. 
\end{split}
\end{align}
Noting the bound
\begin{align}
\label{E:79}
|\Lp_{(k, \ell), (m-1, n)}[\omega^{\bfa, \bfb, \row, n, \bfc}]-\Lp_{(k, \ell), (m-1, n-1)}[\omega^{\bfa, \bfb}]| \le \sum_{i=k}^{m-1} |c_i|, 
\end{align}
one also has 
\begin{align}
\label{E:80}
\lim_{n \to \infty} \frac{1}{n} \cdot \Lp_{(k, \ell), (m-1, n)}[\omega^{\bfa, \bfb, \row, n, \bfc}] &= \int_{\bbR} \frac{\beta(\dd b)}{b + \min \bfa_{k: m-1}} \quad \text{ a.s. }
\end{align}
By our assumptions, $\beta \neq 0$ and $a_m < \min \bfa_{k: m-1}$. Therefore, the limit in \eqref{E:80} is strictly less than the first limit in \eqref{E:78}. This implies the conclusion in Part \eqref{it:thinA}. 
\end{proof}

We are now ready to carry out the last piece of our argument. Recall \eqref{E:kl0}, where we defined $\bfk_0(p) = k+p-1$ and $\bfs_0(p) = \ell+ d-p$ for $p \in [d]$. Observe that for $\bfz_p$ defined in Theorem \ref{T:BusMar}, we have $\bfz_p = (\bfk_0(p),\bfs_0(p))$. 
\begin{lem}
\label{L:BusDisId}
The following hold.  
\begin{enumerate}[\normalfont (a)]
\item \label{itm:eq1}$
\begin{aligned}[t]
&\quad \;\bfP\{\sI_{\bfu}^{\bfk_0(p), \uparrow}[\omega^{\col, \bfk_0}] > x_{\bfu,p},\; \sJ_{\bfv}^{\bfk_0(p), \uparrow}[\omega^{\col, \bfk_0}] < y_{\bfv,p},\;\forall \bfu,\bfv,p \} \\ 
&= \bfP\{\ul{\I}_{\bfu, (\bfk_0(p), \bfs_0(p))}[\eta] > x_{\bfu,p},\; \ul{\J}_{\bfv, (\bfk_0(p), \bfs_0(p))}[\eta] < y_{\bfv,p},\;\forall \bfu,\bfv,p \}. 
\end{aligned}
$
\vspace{0.1in}
\item \label{itm:eq2} $
\begin{aligned}[t]
&\quad\;\bfP\{\sI_{\bfu}^{\bfs_0(p), \rightarrow}[\omega^{\row, \bfs_0}] > x_{\bfu,p},\; \sJ_{\bfv}^{\bfs_0(p), \rightarrow}[\omega^{\row, \bfs_0}] < y_{\bfv,p},\;\forall \bfu,\bfv,p \} \\ 
&= \bfP\{\ul{\I}_{\bfu, (\bfk_0(p), \bfs_0(p))}[\eta] > x_{\bfu,p},\; \ul{\J}_{\bfv, (\bfk_0(p), \bfs_0(p))}[\eta] < y_{\bfv,p},\;\forall \bfu,\bfv,p \}.
\end{aligned}
$
\end{enumerate}
\end{lem}
%


Since the proof of Lemma \ref{L:BusDisId} is somewhat long and technical, we present it as a sequence of smaller lemmas. We also prove part \eqref{itm:eq1} only because part \eqref{itm:eq2} follows from a symmetric argument.  

Our first step is to express the thin Busemann functions in part (a) in terms of induced weights defined as follows. For each $p \in [d]$, let
\begin{align}
\label{E:39}
\wt{\omega}^p_{(i, j)} \coloneqq \one_{\{j < \ell\}} \cdot \omega^{\col, \bfk_0}_{(i, j)} + \one_{\{j = \ell\}} \cdot \one_{\{i < \bfk_0(p)\}} \cdot \sI_{(i, \ell)}^{\bfk_0(p), \uparrow}[\omega^{\col, \bfk_0}]
\end{align}
for $i \in [\bfk_0(p)]$ and $j \in [\ell]$. 

\begin{lem}
\label{L:BusDisId-1}
For each $p \in [d]$, 
\begin{align}
\label{E:40}
\begin{split}
\sI_{\bfu}^{\bfk_0(p), \uparrow}[\omega^{\col, \bfk_0}] &= \ul{\I}_{\bfu, (\bfk_0(p), \ell)}[\wt{\omega}^p] \quad \text{ for } \bfu \in \rightset_{k,\ell},\quad\text{and} \\
\sJ_{\bfv}^{\bfk_0(p), \uparrow}[\omega^{\col, \bfk_0}] &= \ul{\J}_{\bfv, (\bfk_0(p), \ell)}[\wt{\omega}^p] \quad \text{ for } \bfv \in \upset_{k,\ell}. 
\end{split}
\end{align}    
\end{lem}
\begin{proof}
For each $p \in [d]$ and $n \in \bbZ_{\ge \ell}$, consider the prelimit version of the $\wt{\omega}^p$-weights given by 
\begin{align*}
\wt{\omega}^{p, n}_{(i, j)} \coloneqq \one_{\{j < \ell\}} \cdot \omega^{\col, \bfk_0}_{(i, j)} + \one_{\{j = \ell\}} \cdot \one_{\{i < \bfk_0(p)\}} \cdot \ul{\I}_{(i, \ell), (\bfk_0(p), n)}[\omega^{\col, \bfk_0}]    
\end{align*}
for $i \in [\bfk_0(p)]$ and $j \in [\ell]$. 
Then it follows from \eqref{E:IndInc} that 
\begin{equation} \label{eq:40_prelim}
\begin{aligned}
    &\ul \I_{\bfu,(\bfk_0(p),n)}[\omega^{\col,\bfk_0}] = \ul \I_{\bfu,(\bfk_0(p),\ell)}[\wt \omega^{p,n}],\quad \bfu \in \rightset_{k,\ell},\quad\text{and} \\
    &\ul \J_{\bfv,(\bfk_0(p),n)}[\omega^{\col,\bfk_0}] = \ul \J_{\bfv,(\bfk_0(p),\ell)}[\wt \omega^{p,n}],\quad \bfv \in \upset_{k,\ell}.
    \end{aligned}
\end{equation}
As $n \to \infty$, the definition \eqref{E:ThinBuse}  implies that the $\omega^{p,n}$-weights converge vertex-wise to the $\wt{\omega}^p$-weights almost surely. Taking limits as $n \to\infty$ in \eqref{eq:40_prelim} and using continuity of last-passage times as a function of the weights, we obtain the claim. 
\end{proof}

As preparation for an inductive argument, for each $q \in [d]$ and $p \in [d-q+1]$, consider the $\eta^{p, q}$-weights on $[\bfk_0(p)] \times [\ell+q-1]$ given by  
\begin{align}
\label{E:38}
\begin{split}
\eta_{(i, j)}^{p, q} \coloneqq \one_{\{j < \ell+q-1\}} \cdot \eta_{(i, j)} + \one_{\{j = \ell+q-1\}} \cdot \one_{\{i < \bfk_0(p)\}} \cdot \sI_{(i, \ell+q-1)}^{\bfk_0(p), \uparrow}[\omega^{\col, \bfk_0}],
\end{split}
\end{align}
for $i \in [\bfk_0(p)]$ and $j \in [\ell+q-1]$. For $q = 1$ in particular, comparing definitions \eqref{E:auxw} and \eqref{E:wh} 
shows that 
\begin{align}
\label{E:41}
\begin{split}
\{\wt{\omega}_{(i, j)}^p: p \in [d], (i, j) \in [\bfk_0(p)] \times [\ell]\} \deq \{\eta_{(i, j)}^{p, 1}: p \in [d], (i, j) \in [\bfk_0(p)] \times [\ell]\}. 
\end{split}
\end{align}
Combining \eqref{E:41} with Lemma \ref{L:BusDisId-1}, one finds that the first probability in Lemma \ref{L:BusDisId}\eqref{itm:eq1} is 
\begin{align}
\label{E:42}
\begin{split}
&\quad \;\bfP\{\sI_{\bfu}^{\bfk_0(p), \uparrow}[\omega^{\col, \bfk_0}] > x_{\bfu,p},\; \sJ_{\bfv}^{\bfk_0(p), \uparrow}[\omega^{\col, \bfk_0}] < y_{\bfv,p},\;\forall \bfu,\bfv,p \} \\ 
&= \bfP\{\ul{\I}_{\bfu, (\bfk_0(p), \ell)}[\eta^{p, 1}] > x_{\bfu,p},\; \ul{\J}_{\bfv, (\bfk_0(p), \ell)}[\eta^{p, 1}] < y_{\bfv,p},\;\forall \bfu,\bfv,p\}, 
\end{split}
\end{align} 
which will be the starting point of our induction. 

For each $q \in [d+1]$, we next introduce the event
\begin{align}
\label{E:43}
\begin{split}
&E^q \coloneqq \\
&\quad \Bigl\{\ul{\I}_{\bfu, (\bfk_0(p), \ell+q-1)}[\eta^{p, q}] > x_{\bfu,p},\; \ul{\J}_{\bfv, (\bfk_0(p), \ell+q-1)}[\eta^{p, q}] < y_{\bfv,p},\;\forall \bfu,\bfv,\; p \in [d-q+1], \\ 
&\quad \ \ \ul{\I}_{\bfu, (\bfk_0(p), \bfs_0(p))}[\eta] > x_{\bfu,p},\; \ul{\J}_{\bfv, (\bfk_0(p), \bfs_0(p))}[\eta] < y_{\bfv,p},\;\forall \bfu,\bfv,\; p \in [d] \setminus [d-q+1]\Bigr\}.  
\end{split}
\end{align}
Note that $E^1$ is the event on the right-hand side of \eqref{E:42} while $E^{d+1}$ is the 
event on the right-hand side of 
Lemma \ref{L:BusDisId}\eqref{itm:eq1}. Thus, our aim is to prove that $\bfP\{E^1\} = \bfP\{E^{d+1}\}$. To this end, we will in fact show that $\bfP\{E^q\} = \bfP\{E^{q+1}\}$ for each $q \in [d]$. 

Recall that we defined a sequence $(c_i)_{i \in \Z_{>0}}$ of i.i.d.\ $\Exp\{1\}$  random variables  on our probability space, independent of both the i.i.d.\ environment $\omega$ and the inhomogeneous environment $\eta$ from \eqref{E:def_eta}. 
For $n \in \bbZ_{>\ell+q-1}$, $q \in [d]$,  $i \in [k+d-q]$ and  $\epsilon > 0$ sufficiently small so that
\begin{align}
\label{E:46}
\zeta(r_p) + \epsilon < \zeta(r_{p+1}) \quad \text{ for } p \in [d-1],
\end{align}
we define
\begin{equation} \label{ci_def}
c_i^{q,\epsilon} \coloneqq \begin{cases}
\bigl(\zeta(r_{d-q + 1}) + \epsilon\bigr)^{-1} c_i &\text{if }i< k \\
\bigl(\zeta(r_{d-q+1})+\epsilon-\zeta(r_{i-k+1})\bigr)^{-1} c_i &\text{if }i \ge k.
\end{cases}
\end{equation}
Now, for $q \in [d], n \ge 1$, and $\epsilon > 0$ satisfying \eqref{E:46}, consider the weights $\wh{\omega}^{q, n, \epsilon}$ defined on $[k+d-q] \times [n]$ such that, for $i \in [k+d-q]$, 
\begin{align}
\label{E:44}
\begin{split}
&\wh{\omega}^{q, n, \epsilon}_{(i, j)} \coloneqq 
\begin{cases} \eta_{(i, j)} \quad &\text{ if } j < \ell+q-1 \\ 
\omega_{(i, j)}^{\col, \bfk_0} \quad &\text { if } \ell + q -1 \le j < n  \\
c_i^{q,\epsilon} &\text{ if }j = n.
\end{cases} 
\end{split}
\end{align} 
 From the definition of $c_i^{q,\epsilon}$,
 \[
 \wh{\omega}^{q, n, \epsilon}_{(i, n)} \sim 
\begin{cases}
\Exp\{\zeta(r_{d-q+1})+\epsilon\} \quad &\text{ if } i < k, \\ 
\Exp\{\zeta(r_{d-q+1})+\epsilon-\zeta(r_{i-k+1})\} \quad &\text{ if } i \ge k.
\end{cases}
 \]
In particular, \eqref{E:44} places a large $\Exp\{\epsilon\}$-distributed weight at vertex $(k+d-q, n)$. One can interpret $\wh{\omega}^{q, n, \epsilon}$ as inhomogeneous weights $\omega^{\bfa, \bfb}$ on the grid $[k+d-q] \times [n]$ with parameters 
\begin{equation}\label{E:45}
\begin{aligned}
&a_i = 
\begin{cases}
0 \quad &\text{ if } i < k \\
-\zeta(r_{i-k+1}) \quad &\text{ if } i \ge k
\end{cases}
\quad \text{ and } \\
&b_j = 
\begin{cases}
\zeta(r_{\ell+d-j}) \quad &\text{ if } \ell \le j  < \ell+q-1 \\ 
\zeta(r_{d-q+1})+\epsilon \quad &\text{ if } j = n \\ 
1 \quad &\text{ otherwise }
\end{cases}
\end{aligned}
\end{equation}
for $i \in [k+d-q]$ and $j \in [n]$. 

Next, we introduce a prelimiting version of the $\eta^{p, q}$-weights in \eqref{E:38} by 
 \begin{align}
 \label{E:51}
 \wh{\eta}_{(i, j)}^{p, q, n, \epsilon} \coloneqq \one_{\{j < \ell+q-1\}} \cdot \eta_{(i, j)} + \one_{\{j = \ell+q-1\}} \cdot \one_{\{i < \bfk_0(p)\}} \cdot \ul{\I}_{(i, \ell+q-1), (\bfk_0(p), n)}[\wh{\omega}^{n, q, \epsilon}],
 \end{align}
 for $i \in [\bfk_0(p)]$ and $j \in [\ell+q-1]$. Likewise, consider the prelimiting  
 event
 \begin{align}
 \label{E:52}
 \begin{split}
 &\wh{E}^{q, n, \epsilon} \coloneqq \\ &\Bigl\{\ul{\I}_{\bfu, (\bfk_0(p), \ell+q-1)}[\wh{\eta}^{p, q, n, \epsilon}] > x_{\bfu,p},\;\ul{\J}_{\bfv, (\bfk_0(p), \ell+q-1)}[\wh{\eta}^{p, q, n, \epsilon}] < y_{\bfv,p},\;\forall \bfu,\bfv,\;p \in [d-q + 1], \\
& \ \ \ul{\I}_{\bfu, (\bfk_0(p), \bfs_0(p))}[\eta] > x_{\bfu,p},\; \ul{\J}_{\bfv, (\bfk_0(p), \bfs_0(p))}[\eta] < y_{\bfv,p},\;\forall \bfu,\bfv,\;p \in [d]\setminus[d-q+1]\Bigr\}.
 \end{split}
 \end{align}
 This is the analogue of the event $E^q$, with the $\eta^{p, q}$-weights replaced with $\wh \eta^{p, q, n, \epsilon}$. 

\begin{lem}
\label{L:BusDisId-2}    
$\lim \limits_{n \to \infty} \bfP\{\wh{E}^{q, n, \epsilon}\} = \bfP\{E^q\}$ for each $q \in [d]$. 
\end{lem}
\begin{proof}
Note that, for each $p \in [d-q+1]$ and $i \in [\bfk_0(p)]$, 
\begin{align}
\label{E:47}
\Lp_{(i, \ell+q-1), (\bfk_0(p), n-1)}[\omega^{\col, \bfk_0}] = \Lp_{(i, \ell+q-1), (\bfk_0(p), n-1)}[\wh{\omega}^{n, q, \epsilon}],
\end{align}
because the two collections of weights in \eqref{E:47} agree on $[\bfk_0(p)] \times ([n-1] \smallsetminus [\ell+q-2])$ by definition \eqref{E:44}.
Since column $\bfk_0(p)$ is more favorable than the earlier columns, Equation \eqref{E:47} and Lemma \ref{L:ThinShp}\eqref{it:thinA} imply the existence of a random $N_0 = N_0^\epsilon \in \bbZ_{\ge \ell+d}$ such that
\begin{align}
\label{E:48}
\Lp_{(i, \ell+q-1), (\bfk_0(p), n-1)}[\omega^{\col, \bfk_0}] + \wh{\omega}^{n, q, \epsilon}_{(\bfk_0(p), n)} = \Lp_{(i, \ell+q-1), (\bfk_0(p), n)}[\wh{\omega}^{n, q, \epsilon}] \quad \text{ for } n \ge N_0. 
\end{align}
It follows from \eqref{E:48} that, for $p \in [d-q+1]$, $i \in [\bfk_0(p)-1]$ and $n \ge N_0$, 
\begin{align}
\label{E:49}
\ul{\I}_{(i, \ell+q-1), (\bfk_0(p), n-1)}[\omega^{\col, \bfk_0}] = \ul{\I}_{(i, \ell+q-1), (\bfk_0(p), n)}[\wh{\omega}^{n, q, \epsilon}].
\end{align} 
 Passing to the limit as $n \to \infty$ and using the definition of the thin Busemann functions in \eqref{E:ThinBuse} yields
 \begin{align}
 \label{E:50}
 \lim_{n \to \infty } \ul{\I}_{(i, \ell+q-1), (\bfk_0(p), n)}[\wh{\omega}^{n, q, \epsilon}] = \sI_{(i, \ell+q-1)}^{\bfk_0(p), \uparrow}[\omega^{\col, \bfk_0}] \quad \text{ a.s. }
 \end{align}
 By \eqref{E:50} and definition of the weights $\eta^{p,q}$ \eqref{E:38} and $\wh \eta^{p,q,n,\epsilon}$ \eqref{E:51}, we see that the weights $\wh \eta_{(i,j)}^{p,q,n,\epsilon}$ converge almost surely to the weights $\eta_{(i,j)}^{p,q}$ as $n \to \infty$. Then, recalling the definition of the event $E^q$ from \eqref{E:43}, the continuity of last-passage increments as a function of the weights implies the claimed convergence. 
\end{proof}

Next introduce the weights $\wc{\omega}^{q, \epsilon}$  on $[k+d-q] \times \bbZ_{>0}$ by  
\begin{equation} \label{eq:etaqedef}
\wc{\omega}^{q,\epsilon}_{(i,\ell + q - 1)}
\coloneqq \begin{cases}
\f{\zeta(r_{d-q+1})}{\zeta(r_{d-q+1}) + \epsilon} \cdot \eta_{(i,\ell + q - 1)} &\text{if }1 \le i < k, \\ \\
\f{\zeta(r_{d-q+1}) - \zeta(r_{i- k + 1})}{\zeta(r_{d-q+1}) + \epsilon - \zeta(r_{i- k + 1})}\cdot \eta_{(i,\ell + q - 1)} &\text{if }k \le i < k + d -q, \\ \\
c_{k+d-q}^{q,\epsilon} &\text{if } i = k + d-q.
\end{cases}
\end{equation}
on row $\ell+q-1$, and by 
\begin{align}
\label{E:54}
\begin{split}
&\wc{\omega}^{q, \epsilon}_{(i, j)} \coloneqq 
\begin{cases} \eta_{(i, j)} \quad &\text{ if } j < \ell+q-1 \\  
\omega_{(i, j)}^{\col, \bfk_0} \quad &\text { if } j > \ell + q -1. 
\end{cases}
\end{split}
\end{align} 
on the remaining rows. By the distributions of the weights $\eta$ specified in \eqref{E:auxw} and definition of $c_{k+d - q}^{q,\epsilon}$ \eqref{ci_def}, we have, for $i \in [k+d-q]$ and $j \in \bbZ_{>0}$,
\begin{align} \label{eq:wcheck}
&\wc{\omega}^{q, \epsilon}_{(i, \ell+q-1)} \sim 
\begin{cases}
\Exp\{\zeta(r_{d-q+1})+\epsilon\} \quad &\text{ if } i < k \\ 
\Exp\{\zeta(r_{d-q+1})+\epsilon-\zeta(r_{i-k+1})\} \quad &\text{ if } i \ge k.
\end{cases}
\end{align}
From here, we can see that $\wc{\omega}^{q, \epsilon}$ is distributionally identical to the inhomogeneous weights $\omega^{\bfa, \bfb}$ on $[k+d-q] \times \bbZ_{>0}$ with parameters 
\begin{equation}\label{E:55}
\begin{aligned}
&a_i = 
\begin{cases}
0 \quad &\text{ if } i < k \\
-\zeta(r_{i-k+1}) \quad &\text{ if } i \ge k
\end{cases}
\quad \text{ and } \\
&b_j = 
\begin{cases}
\zeta(r_{\ell+d-j}) \quad &\text{ if } \ell \le j  < \ell+q-1 \\ 
\zeta(r_{d-q+1})+\epsilon \quad &\text{ if } j = \ell+q-1 \\ 
1 \quad &\text{ otherwise}
\end{cases}
\end{aligned}
\end{equation}
for $i \in [k+d-q]$ and $j \in \bbZ_{>0}$. Hence, the distribution of the $\wc{\omega}^{q, \epsilon}$-weights restricted to $[k+d-q] \times [n]$ is obtained from that of $\wh{\omega}^{q, n, \epsilon}$ by switching the row parameters $b_{\ell+q-1}$ and $b_n$. In particular, the large $\Exp\{\epsilon\}$-distributed weight has now been moved from the vertex $(k+d-q, n)$ to the vertex $(k+d-q, \ell+q-1)$. 

Define the analogue to the $\wh \eta^{p,q,n,\epsilon}$-weights \eqref{E:51} by 
 \begin{align}
 \label{E:56}
 \wc{\eta}_{(i, j)}^{p, q, n, \epsilon} \coloneqq \one_{\{j < \ell+q-1\}} \cdot \eta_{(i, j)} + \one_{\{j = \ell+q-1\}} \cdot \one_{\{i < \bfk_0(p)\}} \cdot \ul{\I}_{(i, \ell+q-1), (\bfk_0(p), n)}[\wc{\omega}^{q, \epsilon}],
 \end{align}
 for $i \in [\bfk_0(p)]$ and $j \in [\ell+q-1]$. Define the analogue to the event $\wh E^{q,n,\epsilon}$ \eqref{E:52} by   
 \begin{align}
 \label{E:57}
 \begin{split}
 &\wc{E}^{q, n, \epsilon} \\ &\coloneqq \Bigl\{\ul{\I}_{\bfu, (\bfk_0(p), \ell+q-1)}[\wc{\eta}^{p, q, n, \epsilon}] > x_{\bfu,p},\;  \ul{\J}_{\bfv, (\bfk_0(p), \ell+q-1)}[\wc{\eta}^{p, q, n, \epsilon}] < y_{\bfv,p},\;\forall \bfu,\bfv,\;p \in [d-q +1], \\ 
&\qquad \ \ \ul{\I}_{\bfu, (\bfk_0(p), \bfs_0(p))}[\eta] > x_{\bfu,p},\; \ul{\J}_{\bfv, (\bfk_0(p), \bfs_0(p))}[\eta] < y_{\bfv,p},\;\forall \bfu,\bfv,\;p \in [d] \setminus [d-q+1]\Bigr\}. 
 \end{split}
 \end{align}

\begin{lem}
\label{L:BusDisId-3}    
$\bfP\{\wh{E}^{q, n, \epsilon}\} = \bfP\{\wc{E}^{q, n, \epsilon}\}$.
\end{lem}
\begin{proof}
An application of Theorem \ref{T:LppInv} yields 
 \begin{align}
 \label{E:81}
 \begin{split}
& \{\ul{\I}_{(i, \ell+q-1), (\bfk_0(p), n)}[\wh{\omega}^{n, q, \epsilon}]: p \in [d-q+1], i \in [\bfk_0(p)-1]\} \\ 
 &\deq \{\ul{\I}_{(i, \ell+q-1), (\bfk_0(p), n)}[\wc{\omega}^{q, \epsilon}]: p \in [d-q+1], i \in [\bfk_0(p)-1]\}, 
 \end{split}
 \end{align}
 and from the independence of these quantities with the weights below row $\ell + q - 1$, the definitions of the weights $\wh \eta^{p,q,n,\epsilon}$ \eqref{E:51} and $\wc  \eta^{p,q,n,\epsilon}$ \eqref{E:56}, imply that these two collections of weights are equal in distribution. Therefore, the claim follows. 
\end{proof}

Using the weights given by 
\begin{align}
\label{E:60}
 \wc{\eta}_{(i, j)}^{p, q, \epsilon} \coloneqq \one_{\{j < \ell+q-1\}} \cdot \eta_{(i, j)} + \one_{\{j = \ell+q-1\}} \cdot \one_{\{i < \bfk_0(p)\}} \cdot \sI_{(i, \ell+q-1)}^{\bfk_0(p), \uparrow}[\wc{\omega}^{q, \epsilon}]
\end{align}
for $i \in [\bfk_0(p)]$ and $j \in [\ell+q-1]$, let  
\begin{align}
\label{E:61}
\begin{split}
 \wc{E}^{q, \epsilon} &\coloneqq\Bigl\{\ul{\I}_{\bfu, (\bfk_0(p), \ell+q-1)}[\wc{\eta}^{p, q, \epsilon}] > x_{\bfu,p},\;  \ul{\J}_{\bfv, (\bfk_0(p), \ell+q-1)}[\wc{\eta}^{p, q, \epsilon}] < y_{\bfv,p},\;\forall \bfu,\bfv,\;p \in[d-q +1], \\ 
&\qquad \ \ \ul{\I}_{\bfu, (\bfk_0(p), \bfs_0(p))}[\eta] > x_{\bfu,p},\; \ul{\J}_{\bfv, (\bfk_0(p), \bfs_0(p))}[\eta] < y_{\bfv,p},\;\forall \bfu,\bfv,\;p \in [d] \setminus [d-q+1]\Bigr\}
 \end{split}
\end{align}
for the limit version of the event in \eqref{E:57}. 

\begin{lem}
\label{L:BusDisId-4}    
$\lim \limits_{n \to \infty} \bfP\{\wc{E}^{q, n, \epsilon}\} = \bfP\{\wc{E}^{q, \epsilon}\}$. 
\end{lem}
\begin{proof}
From \eqref{E:ThinBuse}, one has 
\begin{align}
\label{E:59}
\begin{split}
\lim_{n \to \infty} \ul{\I}_{(i, \ell+q-1), (\bfk_0(p), n)}[\wc{\omega}^{q, \epsilon}] = \sI_{(i, \ell+q-1)}^{\bfk_0(p), \uparrow}[\wc{\omega}^{q, \epsilon}] \quad \text{ for } i \in [\bfk_0(p)-1]. 
\end{split}
\end{align}
Therefore, as $n \to \infty$, the $\wc{\eta}^{p, q, n, \epsilon}$-weights in \eqref{E:56} converge vertex-wise to $\wc{\eta}^{p, q, \epsilon}$. Consequently, one obtains the claim via 
the bounded convergence theorem. 
\end{proof}

\begin{lem}
\label{L:BusDisId-5}    
$\bfP\{E^q\} = \bfP\{\wc{E}^{q, \epsilon}\}$.
\end{lem}
\begin{proof}
Combining Lemmas \ref{L:BusDisId-2}, \ref{L:BusDisId-3}, and \ref{L:BusDisId-4} yields the claim. 
\end{proof}

\begin{rem}
Note that there is no $\epsilon$-dependence on the left-hand side, which can also be seen 
as follows. Proposition \ref{P:ThinBuse} and the parameter structure in \eqref{E:55} imply that the marginals of these weights are in fact given by 
\begin{align}
\label{E:64}
\wc{\eta}_{(i, \ell+q-1)}^{p, q, \epsilon} = \sI_{(i, \ell+q-1)}^{\bfk_0(p), \uparrow}[\wc{\omega}^{q, \epsilon}] \sim 
\begin{cases}
\Exp\{\zeta(r_{p})\} \quad &\text{ if } i < k \\ 
\Exp\{\zeta(r_p) - \zeta(r_{i-k+1})\} \quad &\text{ if } i \ge k
\end{cases}
\end{align}
for $i \in [\bfk_0(p)-1]$.    
\end{rem}

We define the $\wc{\omega}^{q, 0}$-weights on $[k+d-q] \times \bbZ_{>0}$ by \eqref{eq:etaqedef}-\eqref{E:54} with $\epsilon = 0$. Note 
that 
$\wc{\omega}^{q,  \epsilon}$ converges vertex-wise almost surely, as $\epsilon \to 0$, to 
$\wc{\omega}^{q, 0}$. Along row $j = \ell + q - 1$, the limiting weights satisfy the properties in \eqref{eq:wcheck} with $\epsilon = 0$. In particular, an infinite weight appears on row $\ell+q-1$, that is, $\wc{\omega}_{(k+d-q, \ell+q-1)}^{q, 0} = \infty$. 
Define a variation of the $\wc{\omega}^{q, 0}$-weights by redefining the infinite weight as $0$: precisely, for $i \in [k+d-q]$ and $j \in \bbZ_{>0}$,  
\begin{align}
\label{E:66}
\wc{\omega}_{(i, j)}^q \coloneqq 
\begin{cases}
0 \quad &\text{ if } (i,j) = (k+d-q, \ell+q-1) \\
\wc{\omega}_{(i, j)}^{q, 0} \quad &\text{ otherwise.}
\end{cases} 
\end{align}
As can be seen from \eqref{E:auxw} and \eqref{eq:etaqedef} (with $\epsilon = 0$), we have
\begin{align}
\label{E:73}
\begin{split}
&\wc{\omega}^{q}_{(i, j)} = 
\begin{cases} \eta_{(i, j)} \quad &\text{ if } j \le \ell+q-1 \\ 
\omega_{(i, j)}^{\col, \bfk_0} \quad &\text { if } j > \ell + q -1
\end{cases}
\end{split}
\end{align}  
for $i \in [k+d-q]$ and $j \in \bbZ_{>0}$. Then define 
\begin{align}
\label{E:69}
\begin{split}
 \wc{\eta}_{(i, j)}^{p, q} &= \one_{\{j < \ell+q-1\}} \cdot \eta_{(i, j)} \\ 
& + \one_{\{j = \ell+q-1\}} \cdot \one_{\{i < \bfk_0(p)\}} \cdot \{ \one_{\{p < d-q+1\}} \cdot \sI_{(i, \ell+q-1)}^{\bfk_0(p), \uparrow}[\wc{\omega}^{q}]+ \one_{\{p = d-q+1\}} \cdot \eta_{(i, \ell+q-1)} \}
\end{split}
\end{align}
for $i \in [\bfk_0(p)]$ and $j \in [\ell+q-1]$. Let $\wc{E}^q$ denote the limiting analogue of the event $\wc{E}^{q, \epsilon}$ from \eqref{E:61}: 
\begin{align}
\label{E:71}
\begin{split}
 \wc{E}^{q} &\coloneqq \Bigl\{\ul{\I}_{\bfu, (\bfk_0(p), \ell+q-1)}[\wc{\eta}^{p, q}] > x_{\bfu,p},\; \ul{\J}_{\bfv, (\bfk_0(p), \ell+q-1)}[\wc{\eta}^{p, q}] < y_{\bfv,p},\;\forall \bfu,\bfv,\;p \in  [d-q+1], \\ 
&\qquad \ \ \ul{\I}_{\bfu, (\bfk_0(p), \bfs_0(p))}[\eta] > x_{\bfu,p},\; \ul{\J}_{\bfv, (\bfk_0(p), \bfs_0(p))}[\eta] < y_{\bfv,p},\;\forall \bfu,\bfv,\;p \in [d]\setminus[d-q+1]\Bigr\}. 
 \end{split}
\end{align}


\begin{lem}
\label{L:BusDisId-6} 
$\bfP\{E^q\} = \bfP\{\wc{E}^q\}$ for $q \in [d]$. 
\end{lem}
\begin{proof}
Let $p \in [d-q+1]$. The increment recursions in \eqref{E:IncRec} give us that, for $i \in [\bfk_0(p) - 1]$,
\begin{equation} \label{I1}
\begin{aligned}
&\quad \,\ul\I_{(i,\ell + q - 1),(\bfk_0(p),n)}[\wc \omega^{q,\epsilon}]  \\
&= \wc \omega^{q,\epsilon}_{(i,\ell + q-1)} + \max\bigl\{\ul\I_{(i,\ell + q),(\bfk_0(p),n)}[\wc \omega^{q,\epsilon}] -\ul\J_{(i+1 ,\ell + q-1),(\bfk_0(p),n)}[\wc \omega^{q,\epsilon}]  ,0\bigr \},
\end{aligned}  
\end{equation}
and 
\begin{equation} \label{J1}
\begin{aligned}
&\quad \, \ul \J_{(i,\ell + q - 1),(\bfk_0(p),n)}[\wc \omega^{q,\epsilon}]  \\
&= \wc \omega^{q,\epsilon}_{(i,\ell + q-1)} + \max\bigl\{\ul\J_{(i+1 ,\ell + q-1),(\bfk_0(p),n)}[\wc \omega^{q,\epsilon}]  -\ul\I_{(i,\ell + q),(\bfk_0(p),n)}[\wc \omega^{q,\epsilon}] ,0\bigr \}.
\end{aligned}
\end{equation}
Starting from the observation that 
\[
\ul\J_{(\bfk_0(p) ,\ell + q-1),(\bfk_0(p),n)}[\wc \omega^{q,\epsilon}] = \wc \omega^{q,\epsilon}_{(\bfk_0(p),\ell + q - 1)},
\]
we can see from \eqref{I1}, \eqref{J1}, and induction that the collection
\[
\bigl\{\ul \I_{(i,\ell + q - 1),(\bfk_0(p),n)}[\wc \omega^{q,\epsilon}]: i \in [\bfk_0(p) -1]\bigr\}
\]
can be obtained as a fixed continuous function of
\[
\bigl\{\ul \I_{(i',\ell + q ,(\bfk_0(p),n)}[\wc \omega^{q,\epsilon}] : i' \in [\bfk_0(p) -1]\bigr\} \cup\bigl\{ \wc \omega^{q,\epsilon}_{(i',\ell + q-1)}: i' \in [\bfk_0(p)] \bigr\}.
\]
Taking limits as $n \to \infty$, we see that the collection of thin Busemann functions on row $\ell + q-1$, namely
\[
\Bigl\{\sI_{(i, \ell+q-1)}^{\bfk_0(p), \uparrow}[\wc{\omega}^{q, \epsilon}]: i \in [\bfk_0(p)-1] \Bigr\},
\]
can be written as a continuous function of the weights $\{\wc{\omega}^{q, \epsilon}_{(i', \ell+q-1)}: i' \in [\bfk_0(p)]\}$ on row $\ell+q-1$ and the thin Busemann functions on row $\ell + q$: 
\begin{align}
\label{E:65}
\{\sI_{(i', \ell+q)}^{\bfk_0(p), \uparrow}[\wc{\omega}^{q, \epsilon}]: i' \in [\bfk_0(p)-1]\} = \{\sI_{(i', \ell+q)}^{\bfk_0(p), \uparrow}[\omega^{\col, \bfk_0}]: i' \in [\bfk_0(p)-1]\},
\end{align}
where the equality follows by definition of the $\wc \omega^{q,\epsilon}$ weights \eqref{E:54}. In particular, the collection in \eqref{E:65} does not depend on $\epsilon$, so it follows from the convergence of the $\wc{\omega}^{q, \epsilon}$-weights that, if $p < d-q+1$ (or, equivalently, $\bfk_0(p) < k+d-q$), then 
\begin{align}
\label{E:67}
\lim_{\epsilon \to 0} \sI_{(i, \ell+q-1)}^{\bfk_0(p), \uparrow}[\wc{\omega}^{q, \epsilon}] = \sI_{(i, \ell+q-1)}^{\bfk_0(p), \uparrow}[\wc{\omega}^{q, 0}] = \sI_{(i, \ell+q-1)}^{\bfk_0(p), \uparrow}[\wc{\omega}^{q}] \quad \text{ for } i \in [\bfk_0(p)-1]. 
\end{align}
In the remaining case $p = d-q+1$, the increment recursions \eqref{E:IncRec} (after passing to the thin Busemann limits) give us that, for $i  \in [k+d - q -1]$
\begin{equation} \label{eq:rec_lim}
\begin{aligned}
\sI_{(i, \ell+q-1)}^{k+d-q, \uparrow}[\wc{\omega}^{q, \epsilon}] &= \wc{\omega}^{q, \epsilon}_{(i,\ell + q-1)} + \max\bigl\{\sI_{(i, \ell+q)}^{k+d-q, \uparrow}[\wc{\omega}^{q, \epsilon}] - \sJ_{(i+1, \ell+q-1)}^{k+d-q, \uparrow}[\wc{\omega}^{q, \epsilon}],0 \bigr\} \\
&=\wc{\omega}^{q, \epsilon}_{(i,\ell + q-1)} + \max\bigl\{\sI_{(i, \ell+q)}^{k+d-q, \uparrow}[\omega^{\col,\bfk_0}] - \sJ_{(i+1, \ell+q-1)}^{k+d-q, \uparrow}[\wc{\omega}^{q, \epsilon}],0 \bigr\},\quad\text{and} \\
\sJ_{(i, \ell+q-1)}^{k+d-q, \uparrow}[\wc{\omega}^{q, \epsilon}] &= \wc{\omega}^{q, \epsilon}_{(i,\ell + q-1)} + \max\bigl\{\sJ_{(i+1, \ell+q-1)}^{k+d-q, \uparrow}[\wc{\omega}^{q, \epsilon}] - \sI_{(i, \ell+q)}^{k+d-q, \uparrow}[\omega^{\col,\bfk_0}],0 \bigr\},
\end{aligned}
\end{equation}
where the second equality follows by \eqref{E:54}. Furthermore, we have 
\[
\sJ_{k+d-q,\ell + q-1}^{k+d-q,\uparrow }[\wc \omega^{q,\epsilon}] = \wc \omega^{q,\epsilon}_{(k+d-q,\ell + q-1)} \to \infty \text{ as }\epsilon \to 0.
\]
Then, an inductive argument  using \eqref{eq:rec_lim} (starting from $i = k+d-q-1$ and going backwards) gives us $\sJ_{(i, \ell+q-1)}^{k+d-q, \uparrow}[\wc{\omega}^{q, \epsilon}] \to \infty$ as $\epsilon \to 0$ for $i \in [k+d-q-1]$,  and
\begin{align}
\label{E:68}
\lim_{\epsilon \to 0} \sI_{(i, \ell+q-1)}^{k+d-q, \uparrow}[\wc{\omega}^{q, \epsilon}] = \wc{\omega}^{q, 0}_{(i, \ell+q-1)} = \wc{\omega}^{q}_{(i, \ell+q-1)} = \eta_{(i, \ell+q-1)} \quad \text{ for } i \in [k+d-q-1]. 
\end{align}
On account of \eqref{E:67} and \eqref{E:68}, the $\wc{\eta}^{p, q, \epsilon}$-weights in \eqref{E:60} converge vertex-wise, almost surely, as $\epsilon \to 0$, to the limiting weights $\wc{\eta}^{p, q}$. Therefore, sending $\epsilon \to 0$ in Lemma \ref{L:BusDisId-5} and using definition \eqref{E:61} and bounded convergence again completes the proof.
\end{proof}

We next observe that the event $\wc{E}^q$ is identical to $E^{q+1}$. 
\begin{lem}
\label{L:BusDisId-7} 
$\wc{E}^q = E^{q+1}$ for $q \in [d]$. 
\end{lem}
\begin{proof}
First, for $p = d-q+1$, \eqref{E:69} shows that 
\begin{align}
\label{E:74}
\begin{split}
\ul{\I}_{\bfu, (\bfk_0(p), \ell+q-1)}[\wc{\eta}^{p, q}] &= \ul{\I}_{\bfu, (\bfk_0(p), \ell+q-1)}[\eta] = \ul{\I}_{\bfu, (\bfk_0(p), \bfs_0(p))}[\eta] \quad \text{ for } \bfu \in \rightset_{k,\ell},\quad\text{and} \\
\ul{\J}_{\bfv, (\bfk_0(p), \ell+q-1)}[\wc{\eta}^{p, q}] &= \ul{\J}_{\bfv, (\bfk_0(p), \ell+q-1)}[\eta] = \ul{\J}_{\bfv, (\bfk_0(p), \bfs_0(p))}[\eta] \quad \text{ for } \bfv \in \upset_{k,\ell}. 
\end{split}
\end{align}
Second, for $p < d-q+1$, it follows from \eqref{E:73}, \eqref{E:69} and the increment identities in \eqref{E:IndInc} that 
\begin{align}
\label{E:75}
\begin{split}
\ul{\I}_{\bfu, (\bfk_0(p), \ell+q-1)}[\wc{\eta}^{p, q}] &= \sI_{\bfu}^{\bfk_0(p), \uparrow}[\wc{\omega}^q] = \ul{\I}_{\bfu, (\bfk_0(p), \ell+q)}[\gr{\omega}^{p, q}] \quad \text{ for } \bfu \in \rightset_{k,\ell}, \\
\ul{\J}_{\bfv, (\bfk_0(p), \ell+q-1)}[\wc{\eta}^{p, q}] &= \sJ_{\bfv}^{\bfk_0(p), \uparrow}[\wc{\omega}^q] = \ul{\J}_{\bfv, (\bfk_0(p), \ell+q)}[\gr{\omega}^{p, q}] \quad \text{ for } \bfv \in \upset_{k,\ell},
\end{split}
\end{align}
where $\gr{\omega}^{p, q}$ is the temporary name for the induced weights given by 
\begin{align}
\label{E:76}
\begin{split}
\gr{\omega}^{p, q}_{(i, j)} 
&= \one_{\{j < \ell+q\}} \cdot \wc{\omega}^q_{(i, j)} + \one_{\{j = \ell+q\}} \cdot \one_{\{i < \bfk_0(p)\}} \cdot \sI_{(i, \ell+q)}^{\bfk_0(p), \uparrow}[\wc{\omega}^q] \\
&= \one_{\{j < \ell+q\}} \cdot \eta_{(i, j)} + \one_{\{j = \ell+q\}} \cdot \one_{\{i < \bfk_0(p)\}} \cdot \sI_{(i, \ell+q)}^{\bfk_0(p), \uparrow}[\omega^{\col, \bfk_0}] \\ 
&= \eta^{p, q+1}_{(i, j)}
\end{split}
\end{align}
for $i \in [\bfk_0(p)]$ and $j \in [\ell+q]$. The second equality in \eqref{E:76} comes from \eqref{E:73}, and the last equality is due to definition \eqref{E:38}. Using \eqref{E:74} and \eqref{E:76}, one obtains from \eqref{E:43} and \eqref{E:71} that 
\begin{align}
\label{E:77}
\begin{split}
 \wc{E}^{q} &= \Bigl\{\ul{\I}_{\bfu, (\bfk_0(p), \ell+q)}[\eta^{p, q+1}] > x_{\bfu,p},\; \ul{\J}_{\bfv, (\bfk_0(p), \ell+q)}[\eta^{p, q+1}] < y_{\bfv,p},\;\forall \bfu,\bfv;\;p \in [d-q], \\ 
&\qquad \ \ \ul{\I}_{\bfu, (\bfk_0(p), \bfs_0(p))}[\eta] > x_{\bfu,p},\; \ul{\J}_{\bfv, (\bfk_0(p), \bfs_0(p))}[\eta] < y_{\bfv,p},\;\forall \bfu,\bfv;\;p \in [d]\setminus[d-q]\Bigr\}.
 \end{split}
\end{align}
From here, we recognize that $\wc{E}^{q}$ is exactly the event $E^{q+1}$ \eqref{E:43}.
\end{proof}

\begin{proof}[Proof of Lemma \ref{L:BusDisId}]
Lemmas \ref{L:BusDisId-6} and \ref{L:BusDisId-7} give $\bfP\{E^q\} = \bfP\{\wc{E}^q\} = \bfP\{E^{q+1}\}$ for each $q \in [d]$. Consequently, $\bfP\{E^1\} = \bfP\{E^{d+1}\}$, which completes the proof of part \eqref{itm:eq1}. As mentioned, part \eqref{itm:eq2} follows from a symmetric argument. 
\end{proof}

\begin{proof}[Proof of Theorem \ref{T:BusMar}]
The result is now immediate from \eqref{E:CDFBds} and Lemma \ref{L:BusDisId}. 
\end{proof}


\bibliographystyle{References/habbrv}
\bibliography{References/Refs,References/growthrefs,References/references,References/erikbib, References/interchangeability}
\end{document}